\documentclass[11pt,preprint]{article}

\def\vv{1.45}
\renewcommand{\baselinestretch}{\vv}
\renewcommand\arraystretch{0.8}

\usepackage{amsmath, amssymb, amsfonts, amsbsy, amsthm, epsfig, graphicx,rotating, subfigure}
\usepackage{caption}
\usepackage[usenames]{color}
\usepackage{enumerate}
\RequirePackage[colorlinks,citecolor=blue,urlcolor=red]{hyperref}
\usepackage{natbib}
\usepackage{ulem}
\usepackage{fancyhdr}
\usepackage{multirow,xspace,setspace,booktabs}
\usepackage{hyperref}
\usepackage{hhline}

\allowdisplaybreaks

\numberwithin{equation}{section}  
\def\be{\begin{equation}}
\def\ee{\end{equation}}
\def\bea{\begin{eqnarray}}
\def\eea{\end{eqnarray}}
\def\bd{\begin{displaymath}}
\def\ed{\end{displaymath}}
\def\bda{\begin{eqnarray*}}
\def\eda{\end{eqnarray*}}
\def\bsm{\begin{small}}
\def\esm{\end{small}}

\def\bnt{\begin{enumerate}}
\def\ent{\end{enumerate}}

\def\bsc{\begin{scriptsize}}
\def\esc{\end{scriptsize}}

\newtheorem{lemma}{Lemma}
\newtheorem{proposition}{Proposition}
\newtheorem{theorem}{Theorem}[section]
 
\newtheorem{corollary}{Corollary}
\newtheorem{condition}{Condition}
\renewcommand\thecondition{C\arabic{condition}}

\newtheorem{remark}{Remark}[section]

\theoremstyle{definition}

\newcommand{\beq}{\begin{eqnarray*}}
\newcommand{\eeq}{\end{eqnarray*}}
\newcommand{\beqn}{\begin{eqnarray}}
\newcommand{\eeqn}{\end{eqnarray}}

\def\Var{\mathrm {Var}}
\def\Cov{\mathrm {Cov}}
\def\MSE{\mathrm {MSE}}
\def\Bias{\mathrm {Bias}}
\def\E{\mathrm {E}}
\def\P{\mathrm {P}}
\def\DM{\mathrm {DM}}
\def\fn{\footnote} 
\def\bX{\mathbf {X}}
\def\bY{\mathbf {Y}}
\def\bZ{\mathbf {Z}}
\def\bt{\mathbf {t}}
\def\bs{\mathbf {s}}

\def\bG{\mathbf {G}}

\newcommand{\blind}{1}

\begin{document}

\def\spacingset#1{\renewcommand{\baselinestretch}%
	{#1}\small\normalsize} \spacingset{1}


\if1\blind
{
	\title{\bf Distributed Statistical Inference for Massive Data}
	\author{\setcounter{footnote}{0} Song Xi Chen\footnote{
			Author of Correspondence. Song Xi Chen is Chair Professor, Department of Business Statistics and Econometrics, Guanghua School of Management and Center for Statistical Science, Peking University, Beijing 100651, China (Email: csx@gsm.pku.edu.cn). His research is partially supported by Chinas National Key Research Special Program Grants 2016YFC0207701 and 2015CB856000, and National Natural Science Foundation of China grants 11131002, 71532001 and 71371016.} ~and \setcounter{footnote}{1} Liuhua Peng 
			\footnote{Liuhua Peng is Lecturer, School of Mathematics and Statistics, the University of Melbourne, Victoria 3010, Australia (Email: liuhua.peng@unimelb.edu.au).}}
		\date{}
	\maketitle
} \fi

\if0\blind
{
	\bigskip
	\bigskip
	\bigskip
	\begin{center}
		{\LARGE\bf Distributed Statistical Inference for Massive Data}
	\end{center}
	\medskip
} \fi

\bigskip
\begin{abstract}
	This paper considers distributed statistical inference for a general type of statistics that encompasses the U-statistics and the M-estimators in the context of massive data where the data can be stored at  multiple platforms at different locations. In order to facilitate effective  computation and  to avoid expensive data communication among different platforms, 
	we formulate distributed statistics which can be computed over smaller data blocks. 
	The statistical properties of the distributed statistics are investigated  in terms of the mean square error of estimation and their asymptotic distributions with respect to the number of data blocks. In addition, we propose two distributed bootstrap algorithms which are computationally effective and are able to capture the underlying distribution of the distributed statistics. 
	Numerical simulation and real data applications of the proposed approaches  are provided to demonstrate the empirical performance.
\end{abstract}

\noindent%
{\it Keywords:}  Distributed bootstrap, distributed statistics, massive data, pseudo-distributed bootstrap

\spacingset{1.45} 
\section{Introduction}\label{sec:01}

Massive data with rapidly increasing size are generated in many fields of scientific studies that create needs for statistical analyzes.  Not only the size of the data is an issue, but also a fact that the data are often stored in multiple locations, 
where each contains a subset of the data which can be massive in its own right.  This implies that a statistical procedure that involves the entire data has to involve data communication between different storage facilities, which  can slow down the computation.  For statistical analyzes involving  massive data, 
the computational and data storage implications of a statistical procedure should be taken into consideration in additional to the usual criteria of statistical inference. 

Two
methods 
have been developed to deal with the challenges with massive data. One is the so-called ``split-and-conquer" (SaC) method 
considered in \citet{Zhang2013}, \citet{ChenXie2014} and \citet{Battey2015}; and the other is the resampling-based methods advocated by  \citet{Kleiner2014} and \citet{Sengupta2015}.
From the estimation point of view, the SaC 
first partitions the entire data into blocks with smaller sizes, performs the estimation on each data block  and then aggregate the estimators from these blocks to form the final estimator. 
SaC has been used in  different settings under the massive data scenario, for instance the M-estimation by \citet{Zhang2013}, 
the generalized linear models by \citet{ChenXie2014}, 
see also 
\citet{ChenHuo2015}, \citet{Battey2015} and \citet{Lee2015}.

The other method involves making the bootstrap procedure adaptive to the massive data setting in order to obtain standard errors or confidence intervals for statistical inference.  
As the bootstrap resampling of the entire dataset  is not feasible 
 due to its being computationally too expensive, 
\citet{Kleiner2014} introduced the bag of little bootstrap (BLB) 
that incorporates subsampling and the $ m $ out of $ n $ bootstrap to assess the quality of estimators for various inference purposes. 
\citet{Sengupta2015} proposed the subsampled double bootstrap (SDB) method that combines the BLB with a fast double bootstrap \citep{Davidson2002, Chang2015} that has computational advantages over the BLB.  
However,  both the BLB and SDB have some limitations. 
Firstly, the core idea of the BLB and SDB is to construct subsets of small size from the entire dataset, and then resample these small subsets to obtain resamples. 
The computational efficiency of the BLB and SDB heavily rely on that the estimator of interest admits a weighted subsample representation 
\citep{Kleiner2014}, which is usually satisfied for linear statistics. 
Secondly, the SDB 
requires resampling from the entire dataset, which can be  computationally expensive 
 to accomplish when the data are stored in different locations. 

In this paper, we consider distributed statistical inference for a class of symmetric statistics \citep{Lai1993}, which are more general than linear statistics. We first carry out the SaC step for a distributed formulation of the statistics for better computational efficiency. 
Estimation efficiency and distribution approximation for the distributed statistics are studied relative to the full sample formulation. 
We propose computationally efficient bootstrap algorithms to approximate the distribution of the distributed statistics: 
a distributed bootstrap (DB) and a pseudo-distributed bootstrap (PDB) . 
Both strains of the bootstrap method can be implemented distributively and are suitable for parallel and distributed computing platforms, 
while maintaining sufficient level of statistical efficiency and accuracy in the inference.

The paper is organized as follows. Aspects of symmetric statistics 
 are introduced in Section $2$. The distributed statistics are formulated in Section $3$ together with its statistical efficiency. Section $4$ concerns computationally complexity and the selection of the number of data blocks. In Section $5$, asymptotic distributions of the distributed statistics are established for both non-degenerate and degenerate cases. Bootstrap procedures designed to approximate the distribution of the distributed statistics are discussed in Section $6$. 
Section $7$ provides simulation  verification to the theoretical results. An application 
on the airline on-time performance data is presented in Section $8$. 
Proofs,  
more technical details and some simulation results are reported in the supplementary materials.
 
\section{Symmetric statistics}

Let $\mathfrak{X}_N=\{X_1,\ldots,X_N\}$ be  a sequence of independent random vectors taking values in a measurable space $(\mathcal{X},\mathcal{B})$ with a common distribution $F$. A statistic $T_N=T(\mathfrak{X}_N)$ for a parameter  $\theta=\theta(F)$, that is invariant under any data permutations, admits a general non-linear form 
\begin{align}\label{eq:T_N}   
	T_N = \theta +N^{-1}\sum_{i=1}^{N}\alpha(X_i;F)+N^{-2}\sum_{1\leq i<j\leq N}\beta(X_i,X_j;F)+R_N,
\end{align}
where  $\alpha(x;F)$ and $\beta(x,y;F)$ are known 
functions,  
and $R_N=R(\mathfrak{X}_N; F)$ is a remainder term of the expansion that constitutes the first three terms on the right side. 

The statistic $T_N$ encompasses some commonly used statistics, for example, the U- and L-statistics, and the smoothed functions of the sample means. The form \eqref{eq:T_N} was introduced in \citet{Lai1993} while Edgeworth expansions for $T_N$ were studied in \citet{Jing2010}. In these two papers, it is assumed that $\alpha(X_1;F)$ is non-lattice distributed and $R_N$ is of smaller order of the quadratic term involving $\beta$.  For instance, $T_N$ can be a U-statistic with $\alpha(X_i;F)$ and $\beta(X_i,X_j;F)$ terms to be the first and second order terms in the Hoeffding's decomposition \citep{Hoeffding1948,Serfling1980}. In this case, $\E(R_N)=0$ and $\Var(R_N)=O(N^{-3})$. Here, we relax the conditions that $\alpha(X_1;F)$ is non-lattice and $R_N$ is of a small order  in \citet{Lai1993} and \citet{Jing2010} such that $T_N$ includes more types of statistics. The linear term involving $\alpha(X_i;F)$ can vanish, for instance in the case of the degenerate U-statistics. The M-estimators are included in \eqref{eq:T_N} with $R_N$ having certain order of magnitude depending on the forms of the score functions. When the score function of the M-estimator is twice differentiable and its second derivative is Lipschitz continuous, the M-estimator can be expressed in the form of \eqref{eq:T_N} with an explicit $\beta$ and the remainder term $R_N = O_p(N^{-1})$. 
However, when the score function is not smooth enough, the M-estimator may not be expanded to the second-order $\beta(X_i,X_j;F)$ term. In this situation, we can absorb the quadratic term into $R_N$, which is often of order $O_p(N^{-3/4})$ \citep{HeAndShao1996}. 
We assume the following regarding $\alpha$ and $\beta$.

\begin{condition}\label{condition:alpha}
	The functions $\alpha(x; F)$ and $\beta(x,y;F)$, depending on $F$, are known measurable functions of $x$ and $y$, satisfying $\E\{\alpha(X_1;F)\}=0$ and $\Var\{\alpha(X_1;F)\}=\sigma_{\alpha}^2\in[0,\infty)$, and $\beta(x,y;F)$ is symmetric in $x$ and $y$ such that $\E\{\beta(X_1,X_2;F)|X_1\}=0$ and $\Var\{\beta(X_1,X_2;F)\}=\sigma_{\beta}^2\in[0,\infty)$.
\end{condition}

\section{Distributed statistics}\label{sec:SACE}

To improve the computation of $T_N$, 
we divide the full dataset $\mathfrak{X}_N$ into $K$ data blocks. Let $\mathfrak{X}_{N,K}^{(k)}=\left\{X_{k,1},\ldots,X_{k,n_k}\right\}$ be the $k$-th data block of size $n_k$, for $k=1,\ldots,K$. Such division is naturally available when $\mathfrak{X}_N$ is stored over $K$ storage facilities. Otherwise, the blocks can be obtained by random sampling.  

\begin{condition}\label{condition:K}
	There exist finite positive constants $c_1$ and $c_2$ such that $c_1\leq\inf_{k_1,k_2}n_{k_1}/n_{k_2}\leq\sup_{k_1,k_2}n_{k_1}/n_{k_2}\leq c_2$, and $K$ can be either finite or diverging to infinity as long as it satisfies that $K/N\to0$ as $N\to\infty$.
\end{condition} 

Condition \ref{condition:K} implies that $\{n_k\}_{k=1}^K$ are of the same order and the number of data blocks $K$ should be of a smaller order of $N$.

For $k=1,\ldots,K$, let $T_{N, K}^{(k)}=T\big(\mathfrak{X}_{N,K}^{(k)}\big)$ be a version of $T_N$, but is based on the $k$-th data block $\mathfrak{X}_{N,K}^{(k)}$ such that $T_{N, K}^{(k)} = \theta+n_k^{-1}\sum_{i=1}^{n_k}\alpha(X_{k,i};F)+n_k^{-2}\sum_{1\leq i<j\leq n_k}\beta(X_{k,i},X_{k, j};F)+R_{N, K}^{(k)}$, 
where $R_{N, K}^{(k)}=R\big(\mathfrak{X}_{N,K}^{(k)};F\big)$ is the remainder term specific to the $k$-th block. 

By averaging the $K$ block-wise statistics, we arrive at the distributed statistic 
\begin{align}\label{eq:T_N_K_a} 
	T_{N, K} = N^{-1}\sum_{k=1}^{K}n_kT_{N, K}^{(k)},   
\end{align} 
which can be expressed as  
\begin{align}\label{eq:T_N_K}
	T_{N, K} = \theta+N^{-1}\sum_{i=1}^{N}\alpha(X_i;F)+N^{-1}\sum_{k=1}^K n_k^{-1}\sum_{1\leq i<j\leq n_k}\beta(X_{k,i},X_{k,j};F)+R_{N,K},
\end{align}
where $R_{N,K}=N^{-1}\sum_{k=1}^{K}n_kR_{N, K}^{(k)}$. It is clear that the first two terms 
of $T_{N,K}$ and $T_N$ are the same. The difference between them occurs at the terms involving $\beta$ and the remainders.

	The distributed statistics for {non-degenerate} U-statistics were considered  in \citet{LinandXi2010}. The authors suggested the ``aggregated U-statistic" which is just $T_{N,K}$ when $T_N$ is a U-statistic. They showed that the aggregated U-statistic is asymptotically equivalent to the corresponding U-statistic. 
	In this paper, we consider a more general class of statistics with different assumptions on $\alpha$, $\beta$ and $R_N$ without the limitation of non-degenerate statistics.
	In addition, we study  distributed bootstrap strategies to approximate the distribution of $T_{N,K}$, that was not considered in \citet{LinandXi2010}.

To facilitate analyzes on the statistical efficiency of $T_{N, K}$, we assume the following. 

\begin{condition}\label{condition:R_N}
	If $\E(R_N)=b_1N^{-\tau_1}+o(N^{-\tau_1})$ and $\Var(R_N)=o(N^{-\tau_2})$ for some $b_1\neq0$, $\tau_1\geq1$ and $\tau_2\geq1$, then $\E\big(R_{N, K}^{(k)}\big)=b_{1,k} n_k^{-\tau_1}+o(n_k^{-\tau_1})$ for some $b_{1,k}\ne0$ and $\Var\big(R_{N,K}^{(k)}\big)=o(n_k^{-\tau_2})$, for $k=1,\ldots,K$.
\end{condition}

Condition \ref{condition:R_N} prescribes a divisible property of $R_N$ in its first two moments with respect to the sample size 
$N$. That is, $R_{N,K}^{(k)}$ inherits the properties of $R_N$'s first two moments with $N$ substituted by $n_k$.

We first present results concerning $T_N$ to provide benchmarks for our analyzes. 

\begin{proposition}\label{pro:MSE_T_N}
	Under Condition \ref{condition:alpha}, if $\E(R_N)=b_1N^{-\tau_1}+o(N^{-\tau_1})$ for some $b_1\neq0$, $\tau_1\geq1$, and $\Var(R_N)=o(N^{-2})$, then
	$\Bias(T_N) = b_1N^{-\tau_1}+o(N^{-\tau_1})$ {and}
	$$ \Var(T_N) = \sigma_{\alpha}^2N^{-1}+2^{-1}\sigma_{\beta}^2N^{-2}+N^{-1}\sum_{i=1}^{N}\Cov\left\{\alpha(X_i;F), R_N\right\}+o(N^{-2}). $$
\end{proposition}

The bias of $T_N$ is  determined by the expectation of $R_N$. In the variance expression, we keep the term involving the covariance between $\alpha(X_i, F)$ and $R_N$ as it may not be negligible in a general case. For the special case when $R_N$ is uncorrelated with $\alpha(X_i;F)$, which is the case of U-statistics, the covariance term vanishes. 

According to Proposition \ref{pro:MSE_T_N}, the mean square error (MSE) of $T_N$ is
\begin{align}\label{eq:mse_T_N}
	\notag \MSE(T_N) = \sigma_{\alpha}^2N^{-1} & +2^{-1}\sigma_{\beta}^2N^{-2}+b_1^2N^{-2\tau_1} \\
			& +N^{-1}\sum_{i=1}^{N}\Cov\left\{\alpha(X_i;F), R_N\right\}+o(N^{-2}+N^{-2\tau_1}).
\end{align} 
\indent The bias and variance of $T_{N,K}$ is given in the following theorem.


\begin{theorem}\label{theo:mse1}
	Under Conditions \ref{condition:alpha}, \ref{condition:K}, \ref{condition:R_N}, if $\E(R_N)=b_1N^{-\tau_1}+o(N^{-\tau_1})$ for some $b_1\neq0$, $\tau_1\geq1$, and $\Var(R_N)=o(N^{-2})$, then
	\bea 
		&& \Bias(T_{N,K}) = N^{-1}\sum_{k=1}^{K}b_{1,k}n_k^{1-\tau_1}+o(K^{\tau_1}N^{-\tau_1})~~~~~\text{and} \nonumber \\ 
	&& \Var(T_{N,K}) = \sigma_{\alpha}^2N^{-1}+2^{-1}\sigma_{\beta}^2KN^{-2}+N^{-2}\sum_{k=1}^{K}n_k\sum_{i=1}^{n_k}\Cov\left\{\alpha(X_{k.i}; F), R_{N,K}^{(k)}\right\}+o(KN^{-2}). \nonumber 
	\eea 
\end{theorem}

If $T_N$ is an unbiased estimator of $\theta$, namely $\tau_1=\infty$, $T_{N,K}$ is also unbiased so that $\E(T_{N,K})=\theta$. Thus the distributed approach can preserve the unbiasedness of the statistics. For the biased case of $\tau_1<\infty$, 
as $N^{-1}\sum_{k=1}^{K}b_{1,k}n_k^{1-\tau_1}$ is $O(K^{\tau_1}N^{-\tau_1})$ under Condition \ref{condition:K}, the bias is enlarged by a factor of $K^{\tau_1}$ for $T_{N,K}$ relative to that of $T_N$. This indicates that the data blocking accumulates the biases from each data block, leading to an increase in the bias of the distributed statistic. For the more common  case that the bias of $T_N$ is at the order of $O(N^{-1})$, the bias of $T_{N,K}$ is of order $O(KN^{-1})$.

For the variance,  there is an increase by a factor of $K$ in the term that involves $\sigma_{\beta}^2N^{-2}$ due to the data blocking.
For the covariance term in $\Var(T_N)$ and $\Var(T_{N,K})$, under the conditions of Theorem \ref{theo:mse1}, as $\Var\big\{N^{-1}\sum_{i=1}^N\alpha(X_i;F)\big\}=O(N^{-1})$ and $\Var(R_N)=o(N^{-2})$,  
$\Cov\big\{N^{-1}\sum_{i=1}^N\alpha(X_i;F),R_N\big\}=o(N^{-3/2})$ by Cauchy-Schwarz inequality.
Similarly, $N^{-2}\sum_{k=1}^{K}n_k\sum_{i=1}^{n_k}\Cov\big\{\alpha(X_{k.i};F),R_{N,K}^{(k)}\big\}=o(K^{1/2}N^{-3/2})$. 
Thus, an increase by a factor of $K^{1/2}$ occurs in the covariance term. Moreover, the covariance term in $\Var(T_{N,K})$ is always  a smaller order of $N^{-1}$ as $K=o(N)$. However, it may be a larger order of $KN^{-2}$. For the case of $\sigma_{\alpha}^2>0$, the variance inflation happens in the second order term without altering the leading order term. For the degenerate case of $\sigma_{\alpha}^2=0$,  but $\sigma_{\beta}^2>0$, the variance inflation appears in the leading order via a factor of $K$. 

Combining the bias and variance from Theorem \ref{theo:mse1}, we have 
\begin{align}\label{eq:mse_T_N_K}
	\notag \MSE(T_{N,K})= & \sigma_{\alpha}^2N^{-1}+2^{-1}\sigma_{\beta}^2KN^{-2}+N^{-2}\left(\sum_{k=1}^{K}b_{1,k}n_k^{1-\tau_1}\right)^2 \\
	& +N^{-2}\sum_{k=1}^{K}n_k\sum_{i=1}^{n_k}\Cov\left\{\alpha(X_{k.i};F)R_{N,K}^{(k)}\right\} +o(KN^{-2}+K^{2\tau_1}N^{-2\tau_1}).
\end{align}
\indent To compare the MSEs of $T_N$ and $T_{N, K}$, we first consider the non-degenerate case of $\sigma_{\alpha}^2>0$. By comparing \eqref{eq:mse_T_N} with \eqref{eq:mse_T_N_K}, if
\begin{align}\label{eq:K_condition1}
	K=o\big(N^{1-1/(2\tau_1)}\big),
\end{align}
then $K^{2\tau_1}N^{-2\tau_1}=o(N^{-1})$, and $\MSE(T_N)$ and $\MSE(T_{N,K})$ share the same leading order term $\sigma^2_{\alpha}N^{-1}$. 
We note that $1-1/(2\tau_1)$ is an increasing function of $\tau_1$, ranging from $1/2$ ($\tau_1=1$) to $1$ ($\tau_1\rightarrow\infty$). For the case of $\tau_1=1$ when $T_N$ has a bias of order $N^{-1}$, the number of blocks $K$ is required to be at a smaller order of $N^{1/2}$ such that $T_{N,K}$ maintains the same leading MSE as $T_N$. The increase in the variance due to the data blocking is reflected in the MSE, although it is of the second order in the non-degenerate case. 

If $T_N$ is degenerate ($\sigma_{\alpha}^2=0$ but $\sigma_{\beta}^2>0$), besides the increase in the bias, 
an increase by a factor of $K$ occurs in the variance of $T_{N,K}$. This indicates that the distributed statistic can not maintain the same efficiency as $T_N$ in the degenerate case.

In summary, we see that the data blocking increases the bias and variance of the distributed statistic $T_{N,K}$. This may be viewed as a price paid to gain computational scalability for massive data. Despite this, there is room to select $K$ properly to minimize the increases. For instance, by choosing $K$ to satisfy \eqref{eq:K_condition1}, the increase in the bias and variance are  confined in the second order for the non-degenerate case. 

\section{Computing issues and selection of $K$}\label{sec:select_K}

The rationale for carrying out the distributed formulation is in reducing time for computing and data communication in the context of massive data while maintaining certain level of estimation accuracy given a computation  budget.  
Suppose that the computational complexity of calculating $T_N$ is at the order of $O(N^{a})$ with $a>1$, then computing $T_{N,K}$ only needs $K\times O((N/K)^a)$, that is $O(K^{1-a}N^a)$ steps. So calculating $T_{N,K}$ is a factor of $O(K^{a-1})$ faster than computing $T_N$ directly. Similar statement can be made from Theorem 5 in \citet{ChenXie2014}. 

In addition to saving computing cost, the distributed approach requires less memory space. Suppose the memory needed for computing $T_N$ is $O(N^b)$ for some $b\geq 1$. Then for $T_{N,K}$, the requirement on the size of memory is at the order of $O((N/K)^b+K)$. This means a memory saving by a factor of $1-K^{-b}$ when $K=O(N^{b/(1+b)})$, or $1-KN^{-b}$ when $KN^{-b/(1+b)}\to \infty$.

Now we use the U-statistics as an example to demonstrate the computational advantages of the distributed approach. Consider a U-statistic of degree $m \geq 2$ with a symmetric kernel function $h$ such that
\begin{align}\label{eq:U_stat_Hoeffding}
	\notag U_N = & {N\choose m}^{-1}\sum_{1\leq i_1<\ldots<i_m\leq N}h(X_{i_1},\ldots,X_{i_m}) \\
	= & \theta_U+N^{-1}\sum_{i=1}^{N}\alpha_U(X_i;F)+N^{-2}\sum_{1\leq i<j\leq N}\beta_U(X_i,X_j;F)+R_{UN},
\end{align}
where $\theta_U=\E(U_N)$, $\alpha_U(x; F)=m\left[\E\left\{h(x,X_2,\ldots,X_m)\right\}-\theta_U\right]$, 
$$ \beta_U(x,y;F) = (m-1)\left[m\E\left\{h(x,y,X_3,\ldots,X_m)\right\}-\alpha_U(x;F)-\alpha_U(y;F)-m\theta_U\right], $$
and $R_{UN}$ is the remainder term satisfying $\E(R_{UN})=0$ and $\Var(R_{UN})=O(N^{-3})$ under the condition that $\E\{h(X_1, \ldots, X_m)\}^2<\infty$. 
The representation in \eqref{eq:U_stat_Hoeffding} is the Hoeffding's decomposition of U-statistics \citep{Hoeffding1948, Hoeffding1961}.

If  $s$ steps are needed to calculate the $h$ function, then the computing cost of $U_N$ is $(s+1){N\choose m}$, which is at the order of $O(N^{m})$. Assume that the entire dataset of size $N$ is divided evenly into $K$ subsets with each of size $n=[N/K]$. 
Denote $U_{N,K}^{(k)}$ as the U-statistic using data in the $k$-th block. 
Notice that the number of computing steps for $U_{N,K}$ is $K  (s+1){n\choose m}$, which is of order $O(K^{1-m}N^m)$. Thus, the computing steps of the distributed U-statistic is a factor $K^{m-1}$ less than that of the full sample based $U_N$. 

Now we consider the problem of selecting $K$. As discussed earlier, using a larger $K$  reduces the computing burden and makes computation more feasible. However, the mean square error of $T_{N,K}$ is an increasing function of $K$, which indicates 
 a loss in the estimation accuracy for using a large $K$. Hence, in practice, the selection of $K$ should be chosen as a compromise between the statistical accuracy and the computational cost and feasibility.

Suppose that $K\geq K_0$ is a hard requirement to meet the memory and storage  capacity. As the mean square error of $T_{N,K}$ is increasing with respect  to $K$, so the minimum mean square error we can get is $\MSE(T_{N,K_0})$, which is achieved at $K=K_0$.
In practice, we may have a fixed time budget. In this situation, denote the computing cost as $C(K)$ and the fixed time budget as $C_0$.
It is reasonable to assume that the computing cost $C(K)$ is a decreasing function of $K$. Denote $K_0'=\min\{K|C(K)\leq C_0\}$, then the optimal $K$ is $\max(K_0,K_0')$. This suggests selecting the smallest $K$ that meet the limits of memory, storage and computation budget, to attain the best statistical efficiency possible.

\section{Asymptotic distribution of $T_{N,K}$}

We investigate the asymptotic distribution of $T_{N,K}$ under different assumptions on $\alpha$ and $R_N$, 
and compare it with that of $T_N$. 
First, we have the following theorem on the asymptotic distribution of $T_N$ when $\sigma_{\alpha}^2>0$.

\begin{theorem}\label{theo:T_N_asym_normal_dist}
	Under Condition \ref{condition:alpha} and $\sigma_{\alpha}^2>0$, if $R_N=o_p(N^{-1/2})$, then as $N\rightarrow\infty$,
	\begin{align*}
		N^{1/2}\sigma_{\alpha}^{-1}(T_N-\theta) \xrightarrow{d} \mathcal{N}(0,1).
	\end{align*}
\end{theorem}

It is noted that the requirement that $R_N=o_p(N^{-1/2})$ can be met by many statistics and one sufficient assumption for it is $\E(R_N^2)=o(N^{-1})$.

When discussing the statistical efficiency of $T_{N,K}$ in Section \ref{sec:SACE}, the conditions on the first two moments of the remainder term $R_N$ are assumed in Theorem \ref{theo:mse1} in order to get a better insight into the MSE of $T_{N,K}$. However, for some statistics in the form of $T_N$, the orders of $R_N$'s first two moments can be hard to achieve. Instead, $R_N$ is known to be at a certain stochastic order. 
Thus, we have an alternative assumption to Condition \ref{condition:R_N} on $R_N$ and $\{R_{N,K}^{(k)}\}_{k=1}^K$.

\begingroup
\def\thecondition{C3$'$}
\addtocounter{condition}{-1}
\begin{condition}\label{condition:R_N'}
	If $R_N=O_p(N^{-\tau_1})$ for some $\tau_1>1/2$, then $R_{N, K}^{(k)}=O_p(n_k^{-\tau_1})$ for $k=1,\ldots,K$.
\end{condition}
\endgroup

The following theorem concerns the asymptotic distribution of $T_{N,K}$.

\begin{theorem}\label{theo:T_N_K_asym_normal_dist}
	Under Conditions \ref{condition:alpha}, \ref{condition:K} and $\sigma_{\alpha}^2>0$, assume $K=o\left(N^{1-1/(2\tau_1)}\right)$, then
	
	(i) If $\E(R_N)=b_1N^{-\tau_1}+o(N^{-\tau_1})$ and $\Var(R_N)=o(N^{-\tau_2})$ for some $\tau_1\geq1$ and $\tau_2\geq1$, then under Condition \ref{condition:R_N}, as $N\rightarrow\infty$,
		$N^{1/2}\sigma_{\alpha}^{-1}(T_{N,K}-\theta) \xrightarrow{d} \mathcal{N}(0,1)$.
	
	(ii) If $R_N=O_p(N^{-\tau_1})$ for some $\tau_1>1/2$, then under Condition \ref{condition:R_N'}, as $N\rightarrow\infty$,
	$		N^{1/2}\sigma_{\alpha}^{-1}(T_{N,K}-\theta) \xrightarrow{d} \mathcal{N}(0,1)$.
\end{theorem}

The condition $K=o\big(N^{1-1/(2\tau_1)}\big)$ indicates that, under Conditions \ref{condition:alpha}, \ref{condition:K}, \ref{condition:R_N} (or \ref{condition:R_N'}) and $\sigma_{\alpha}^2>0$, the smaller order $R_N$ is, the higher $K$ can be, so that $T_{N,K}$ can attain the same asymptotic normal distribution as $T_N$. Thus, the requirement on $K$ is determined by the order of $R_N$. It is noted that the above condition on $K$ coincides with the requirement on $K$ in \eqref{eq:K_condition1} in order that $\MSE(T_N)$ and $\MSE(T_{N,K})$ have the same leading order term. 

We need the following assumption on $\alpha$  to attain the uniform convergence of the distribution of $T_{N,K}$ to that of $T_N$, which is commonly used in obtaining the uniform convergence, see for instance in \citet{Petrov1998}.

\begin{condition}\label{condition:alpha2}
	There exists a positive constant $\delta\leq 1$ such that $\E|\alpha(X_1;F)|^{2+\delta}<\infty$.
\end{condition}

\begin{theorem}\label{theo:uniform_convergence}
	Under Conditions \ref{condition:alpha}, \ref{condition:K}, \ref{condition:R_N}, \ref{condition:alpha2} and $\sigma_{\alpha}^2>0$, assume that $\E(R_N)=b_1N^{-\tau_1}+o(N^{-\tau_1})$ and $\Var(R_N)=o(N^{-\tau_2})$ for some $\tau_1\geq 1$ and $\tau_2> 1$, and $K=O(N^{\tau'})$ for a positive constant $\tau'$ such that $\tau'<1-1/(2\tau_1)$, then as $N\rightarrow\infty$,
	\begin{align*}
		\sup\limits_{x \in \mathbf{R}}\left|\P\left\{N^{1/2}\sigma_{\alpha}^{-1}(T_{N, K}-\theta) \leq x\right\} - \P\left\{N^{1/2}\sigma_{\alpha}^{-1}(T_N-\theta) \leq x\right\}\right| = o(1).
	\end{align*}
\end{theorem} 

It is noted that the uniform convergence requires $K=O(N^{\tau'})$ for $0<\tau'<1-1/(2\tau_1)$, which is slightly stronger than $K=o\big(N^{1-1/(2\tau_1)}\big)$ assumed in Theorem \ref{theo:T_N_K_asym_normal_dist}.

To sum up the results for the case of $\sigma_{\alpha}^2>0$, we note that as long as the number of data blocks $K$ does not diverges too fast, $T_{N,K}$ can maintain the same estimation efficiency as $T_N$ and  the same asymptotic normal distribution. Thus, $T_{N,K}$ is a good substitute of $T_N$ as an estimator of $\theta$ for the non-degenerate case. 

For the degenerate case of $\sigma_{\alpha}^2=0$, an increase by a factor of $K$ in the variance has been shown in Theorem \ref{theo:mse1}. This means that $T_{N,K}$ can not achieve the same efficiency as $T_N$. 
It is noticed that when $\sigma_{\alpha}^2 =0$ and $\sigma^2_{\beta}>0$, 
\begin{align}\label{eq:T_N_K_degenerate}
	\notag T_N-\theta & = N^{-2}\sum_{1\leq i<j\leq N}\beta(X_i,X_j;F)+R_N ~~~~~~\text{and} \\
	T_{N,K}-\theta & = N^{-1}\sum_{k=1}^Kn_k^{-1}\sum_{1\leq i<j\leq n_k}\beta(X_{k,i},X_{k,j};F)+R_{N, K}.
\end{align}
Before studying the asymptotic distributions of $T_N$ and $T_{N,K}$, we introduce notations used in the spectral decomposition of $\beta$. As $\beta$ is symmetric in its two arguments, when it has finite second moment, there exist a sequence of eigenvalues $\{\lambda_{\ell}\}_{{\ell}=1}^{\infty}$ and  eigenfunctions $\{\beta_{\ell}\}_{{\ell}=1}^{\infty}$ 
such that $\beta$ admits expansion \citep{Dunford1963, Serfling1980}:
\begin{align}\label{eq:beta_expan}
	\beta(x,y;F) = \sum_{{\ell}=1}^{\infty}\lambda_{\ell}\beta_{\ell}(x;F)\beta_{\ell}(y;F)
\end{align}
in the sense that $\lim\limits_{L\rightarrow\infty}\E\big\{\big|\beta(X_1,X_2;F)-\sum_{{\ell}=1}^{L}\lambda_{\ell}\beta_{\ell}(X_1;F)\beta_{\ell}(X_2;F)\big|^2\big\}=0$.
Here $\{\beta_{\ell}\}_{\ell=1}^{\infty}$ satisfy   $\sum_{{\ell}=1}^{\infty}\lambda_{\ell}^2=\sigma_{\beta}^2<\infty$,  $\E\left\{\beta_{\ell}(X_1;F)\right\}=0$ for all ${\ell}\in\textbf{Z}$ and 
$$ \E\left\{\beta_{{\ell}_1}(X_1; F)\beta_{{\ell}_2}(X_1; F)\right\}=\left\{ \begin{array}{rcl}
1 & \ \  {\ell}_1={\ell}_2, \\ 
0 &  \ \ {\ell}_1\neq {\ell}_2. 
\end{array}\right. $$


\begin{theorem}\label{theo:T_N_degenerate}
	Under Condition \ref{condition:alpha}, $\sigma_{\alpha}^2=0$ and $\sigma_{\beta}^2>0$, if $R_N=o_p(N^{-1})$, then as $N\rightarrow\infty$,
		$ 2N(T_N-\theta) \xrightarrow{d} \sum_{{\ell}=1}^{\infty}\lambda_{\ell}(\chi_{1{\ell}}^2-1),$ 
	where $\{\chi_{1\ell}^2\}_{\ell=1}^{\infty}$ are independent $\chi_1^2$ random variables.
\end{theorem}


From \eqref{eq:T_N_K_degenerate}, $T_{N,K}$ can be viewed as a weighted average of $K$ independent random variables, the asymptotic behavior of degenerate $T_{N, K}$ is given in the following theorem.

\begin{theorem}\label{theo:T_N_K_degenerate}
	Under Conditions \ref{condition:alpha}, \ref{condition:K}, $\sigma_{\alpha}^2=0$ and $\sigma_{\beta}^2>0$, then
	
	(i) if $K$ is finite and $R_{N,K}=o_p(N^{-1})$, then 
		$2N(T_{N,K}-\theta) \xrightarrow{d} \sum_{{\ell}=1}^{\infty}\lambda_{\ell}(\chi_{K{\ell}}^2-K)$ as $N\rightarrow\infty$, 
	where $\{\chi_{K\ell}^2\}_{\ell=1}^{\infty}$ are independent $\chi_K^2$  random variables;
	
	(ii) if $K\rightarrow\infty$, there exists a constant $\delta'>0$ such that $\E\left|\beta(X_1,X_2;F)\right|^{2+\delta'}<\infty$, and $R_{N, K}=o_p(K^{1/2}N^{-1})$, then 
		$2^{1/2}K^{-1/2}N\sigma_{\beta}^{-1}(T_{N,K}-\theta) \xrightarrow{d} \mathcal{N}(0,1)$ as $N\rightarrow\infty$.
\end{theorem}

Here, we get two different limiting distributions for $T_{N,K}$ depending on whether $K$ is finite or diverging. This is easy to interpret as $T_{N,K}$ is an average of $K$ independent random variables.  Thus, with the data divided into increasingly many data blocks, the distribution of $T_{N, K}$ is asymptotically normal.  When $K$ is finite, the condition $R_{N, K}=o_p(N^{-1})$ is a directly result of $R_{N,K}^{(k)}=o_p(n_k^{-1})$, $k=1,\dots,K$. When $K$ diverges to infinity, if $\E(R_N)=b_1N^{-\tau_1}+o(N^{-\tau_1})$ for some $b_1\neq0$, $\tau_1\geq1$ and $\Var(R_N)=o(N^{-2})$, then under Condition \ref{condition:R_N} which prescribes the divisibility of $R_N$ with respect to the sample size $N$, $\E(R_{N,K})=N^{-1}\sum_{k=1}^{K}b_{1,k}n_k^{1-\tau_1}+o(K^{\tau_1}N^{-\tau_1})$ and $\Var(R_{N,K})=o(KN^{-2})$. Thus, in order for $R_{N, K}=o_p(K^{1/2}N^{-1})$, it requires $K^{\tau_1}N^{-\tau_1}=o(K^{1/2}N^{-1})$ that is 
$ K=o\big(N^{1-1/(2\tau_1-1)}\big). $
This is a slower growth rate for $K$ comparing to the non-degenerate case. If $\tau_1=1$ with $b_1\neq0$, $\E(R_{N,K})=O(KN^{-1})$ which means that $R_{N, K}=o_p(K^{1/2}N^{-1})$ is not achievable in this case. Similarly, if $R_N=O_p(N^{-\tau_1})$ for some $\tau_1>1$, then under Condition \ref{condition:R_N'}, it is necessary $K=o\left(N^{1-1/(2\tau_1-1)}\right)$ such that $R_{N, K}=o_p(K^{1/2}N^{-1})$ is satisfied.

It is of interest to mention that when $K\to \infty$, the normalizing factor is of order $K^{-1/2}N$, which is of smaller order of $N$, the normalizing factor for $T_N$. This 
reflects the variance increase of $T_{N,K}$ as revealed in Theorem \ref{theo:mse1}.

\section{Bootstrap procedures}

An important remaining issue is how to approximate the distribution of $T_{N,K}$ under different conditions on $\sigma^2_{\alpha}$ and $K$.  
 This motivates our consideration of the bootstrap method, which has been a powerful tool for statistical analyzes, especially in approximating distributions, 
see \citet{Efron1979,Hall1992,Shao1995}.  
In this section, we will first investigate the properties of the conventional bootstrap when applied to the distributed statistics, followed by reviews 
of the bag of little bootstrap (BLB) and the subsampled double bootstrap (SDB). Then, we propose two distributed bootstrap algorithms which are particularly designed for 
the distributed statistics 
 under the massive data scenario.


We first outline the entire sample bootstrap for the distributed statistics to offer a benchmark for our discussions in this section, despite the method is not computationally feasible for massive data.
To approximate the distribution of $T_{N,K}$, 
this version of the bootstrap generates resampled data blocks from the entire sample for constructing resampled distributed statistics. 
By recognizing the distributed character of $T_{N,K}$, 
the entire-sample bootstrap has the following resampling procedure.

\textbf{Step 1}. For $k = 1, \ldots, K$, randomly sample a data subset $\mathfrak{X}_{N,K}^{*(k)}=\big\{X^*_{k,1}, \ldots, X^*_{k,n_k}\big\}$ from $\mathfrak{X}_N$ with replacement and compute the corresponding statistic $T^{*(k)}_{N, K}=T\big(\mathfrak{X}_{N,K}^{*(k)}\big)$.

\textbf{Step 2}. Obtain the distributed statistic on the resampled data subsets $\mathfrak{X}_N^{*}=:\big\{\mathfrak{X}_{N,K}^{*(k)}\big\}_{k=1}^K$
\begin{align*}
	T^*_{N, K} = N^{-1}\sum_{k=1}^{K}n_kT^{*(k)}_{N, K}.
\end{align*}

This two-step 
procedure ensures that the resampled subsets $\mathfrak{X}_N^{*}$ are conditionally independently and identically distributed from $F_N$, the empirical distribution of $\mathfrak{X}_N$. 
Repeat Steps 1 to 2 $B$ times,  $B$ bootstrap distributed statistics, denoted as $T_{N, K}^{*1}, \ldots, T_{N, K}^{*B}$, are obtained. The empirical distribution of $\big\{N^{1/2}\big(T_{N, K}^{*b}-\hat{\theta}_N\big)\big\}_{b=1}^B$ is used to approximate that of $N^{1/2}(T_{N, K}-\theta)$, where  $\hat{\theta}_N=\theta(F_N)$ is the analog of $\theta$ under $F_N$.

The drawback of the entire-sample bootstrap is in requiring a rather large data communication expenditure should the data are stored in multiple locations. 
Furthermore, $\hat{\theta}_N=\theta(F_N)$ has the same computational complexity as $T_N$. 
Thus, we do not pursue this version of the bootstrap algorithm further in this paper. 
	

\def\fn{\footnote}

\subsection{BLB and SDB}

The bag of little bootstrap (BLB) \citep{Kleiner2014} and the subsampled double bootstrap (SDB) \citep{Sengupta2015} are two existing resampling-based methods for massive data inference. 
Suppose $\hat{\theta}_N=\hat{\theta}_N(F_N)$ is an estimator of $\theta=\theta(F)$ based on the sample $\mathfrak{X}_N=\{X_1,\ldots,X_N\}$. Let $u(\hat{\theta}_N,\theta)$ be an inferential quantity concerning $\hat{\theta}_N$, and $Q_N(F)$ be its sampling distribution, which is unknown due to its dependence on the underlying distribution $F$. One is interested in a quantity $\xi\left\{Q_N(F)\right\}$ 
on certain aspect of $Q_N(F)$. For example, if $u(\hat{\theta}_N, \theta)=\hat{\theta}_N-\theta$, 
$\xi\left\{Q_N(F)\right\}$ may be the expectation of $u(\hat{\theta}_N, \theta)$ representing the bias of $\hat{\theta}_N$. Similarly, $\xi\left\{Q_N(F)\right\}$ can be the mean square error of $\hat{\theta}_N$ or a confidence interval of $\theta$.

For massive data, calculating $\hat{\theta}_N$ and its resampled version $\hat{\theta}_N^{*(b)}$ with the entire-sample bootstrap can be computationally expensive. \citet{Kleiner2014} proposed 
BLB to obtain computationally ``cheaper" estimates of $\xi\left\{Q_N(F)\right\}$. 
To avoid calculating the full sample 
$\hat{\theta}_N$ and $\hat{\theta}_N^{*(b)}$,  BLB first  generates data subsets of smaller sizes,  
say $S$ subsets of size $n$ (often $n=N^\varsigma$ for some $0<\varsigma<1$) by sampling  from the original dataset $\mathfrak{X}_N$ without replacement (these subsets can be predefined disjoint subsets of $\mathfrak{X}_N$).  
Denote the $s$-th data subset as $\mathfrak{X}_{s, n}= \{X_{s,1}, \ldots, X_{s,n}\}$,  
 its empirical distribution as $F_{s, n}$,  and $\hat{\theta}_{s,n}=\hat{\theta}_n(F_{s, n})$ as the estimator of $\theta$, which is relatively cheap to compute due to its smaller size. 

The second step of BLB is in constructing inflated resample $\mathfrak{X}_{s,N}^{*(b)}$ of full size $N$  by repeated sampling with replacement from the smaller data subset $\mathfrak{X}_{s, n}$. The inflated resample $\mathfrak{X}_{s,N}^{*(b)}$ can be 
efficiently produced via generating resampling weights for elements of $\mathfrak{X}_{s, n}$  from the multinomial distribution $\text{Multinomial}(N, n^{-1}\mathbf{1}_{n})$ as elaborated in  \citet{Kleiner2014}. 

Let $F_{s, N}^{*(b)}$ 
be the empirical distribution of the inflated resample $\mathfrak{X}_{s,N}^{*(b)}$. 
BLB  calculates the estimator $\hat{\theta}_{s,N}^{*(b)}=\hat{\theta}_N\big(F_{s, N}^{*(b)}\big)$ 
and $u\big(\hat{\theta}_{s,N}^{*(b)},\hat{\theta}_{s,n}\big)$ for B inflated  resamples of $\mathfrak{X}_{s,n}$, namely for $b=1, \cdots, B$.  
Then, the empirical distribution of $\big\{u\big(\hat{\theta}_{s,N}^{*(b)},\hat{\theta}_{s,n}\big)\big\}_{b=1}^B$, denoted as $Q^*_N\big(F_{s,n}\big)$, 
is used to estimate $Q_N(F_{s,n})$ and to obtain $\xi\big\{Q^*_N\left(F_{s, n}\right)\big\}$ for the $s$-th data subset $\mathfrak{X}_{s, n}$. 
Finally, the BLB estimator of $\xi\left\{Q_N(F)\right\}$ is 
 $S^{-1}\sum_{s=1}^{S}\xi\left\{Q^*_N\left(F_{s, n}\right)\right\}$.  
\citet{Kleiner2014} proved that under some conditions, $S^{-1}\sum_{s=1}^{S}\xi\left\{Q^*_N\left(F_{s, n}\right)\right\}$ converges to $\xi\left\{Q_N(F)\right\}$ in probability as $n \to \infty$.

\citet{Sengupta2015} proposed the SDB that combines the idea of the BLB and a fast double bootstrap \citep{Davidson2002,Chang2015}. 
SDB first generates a large number $(S)$ of random subsets of size $n$ from the original full sample, which is similar to BLB's first step.  However,  the SDB selects subsets from the entire dataset by sampling with replacement,  while the BLB can sample subsets without replacement in the creation of the small subsets $\mathfrak{X}_{s, n}$. 
	Hence   
in the SDB,  the sampled subsets are denoted as $\mathfrak{X}_{s,n}^{*}=\big\{X_{s,1}^{*},\ldots,X_{s,n}^{*}\big\}$ with ${\ast}$ to highlight the resampling features, whereas in BLB the symbol is not used in $\mathfrak{X}_{s, n}$ as it may be viewed as true subsets of the original sample.  
For each $\mathfrak{X}_{s,n}^{*}$, SDB  generates only one inflated resample $\mathfrak{X}_{s, N}^{*(1)}=\big\{X_{s,1}^{*(1)}, \ldots, X_{s,N}^{*(1)}\big\}$ of the full size $N$
using the same approach as in the BLB. So for each $\mathfrak{X}_{s,n}^{*}$, only one $u(\hat{\theta}_{s,N}^{*(1)}, \hat{\theta}_{s,n}^{*})$ is obtained, where $\hat{\theta}_{s,n}^{*}=\hat{\theta}_N(F_{s, n}^{*})$ and $\hat{\theta}_{s,N}^{*(1)}=\hat{\theta}_N(F_{s, N}^{*(1)})$, $F_{s, n}^{*}$ and $F_{s, N}^{*(1)}$ are the empirical distributions of $\mathfrak{X}_{s, n}^{*}$ and $\mathfrak{X}_{s, N}^{*(1)}$, respectively. The empirical distribution of $\big\{u(\hat{\theta}_{s,N}^{*(1)},\hat{\theta}_{s,n}^{*})\big\}_{s=1}^S$, denoted as $Q^*_{N,n,S}$, is used to estimate $Q_N(F)$ and obtain $\xi\big\{Q^*_{N,n,S}\big\}$,  the SDB's estimate of $\xi\left\{Q_N(F)\right\}$.

It is noted that SDB only generates one double bootstrap resample for each 
$\mathfrak{X}_{s,n}^{*}$ and the estimate of $\xi\{Q_N(F)\}$ is based on the $u(\hat{\theta}_{s,N}^{*(1)}, \hat{\theta}_{s,n}^{*})$ over different subsets  $\{ \mathfrak{X}^{\ast}_{s,n}\}_{s=1}^S$. Thus,  SDB requires $S \to \infty$. In contrast, the number of subsets {$S$}  in the BLB can be either finite or diverging. 
In addition, SDB requires resampling from the entire dataset $\mathfrak{X}_N$ to get the random subsets $\mathfrak{X}_{s,n}^{*}$, while the BLB can choose predefined disjoint subsets $\{ \mathfrak{X}_{s,n}\}_{s=1}^S$. Thus, for massive datasets stored in different locations, SDB requires more data communications as a trade-off for having only one double bootstrap resample. 

The attraction  of BLB and SDB is {in  avoiding resampling from the original full data sample but rather } 
 from  smaller size subsamples, which is facilitated by obtaining the sampling weights from the multinomial distribution. 
Indeed,  both BLB and SDB have computational advantages when the underlying  estimator can be calculated via  a weighted empirical function representation, 
which is the case for the general M-estimators or other liner-type statistics. In those cases, the amount of computation for $\hat{\theta}_{s,N}^{*(b)}$ is scaled in $n$ rather than $N$. 
However, for non-linear estimators like the  symmetric statistics considered in this paper, the computational advantages of BLB and SDB via a weighted empirical function representation may not be available.



As noted   in \citet{Sengupta2015},  BLB uses  a small number of subsets (i.e. small $S$) but a large number of resamples for each subset (i.e. large $B$), 
 which may lead to only a small portion of the full dataset being covered. 
Another issue is that  using too many resamples for each subset can increase the computational burden too.
For the SDB, it can not be implemented distributively as it requires to conduct the first level resampling from the entire data. 
We propose two versions of a distributed bootstrap method that overcomes these issues. 

\subsection{Distributed bootstrap}

Our aim is in approximating the distribution of $N^{1/2}(T_{N, K}-\theta)$ 
by avoiding resampling from the entire dataset. By recognizing the distributive nature of the statistics, we consider resampling within each data subset.

Suppose the entire data are divided into $K$ subsets: $\mathfrak{X}_{N,K}^{(1)}, \ldots, \mathfrak{X}_{N,K}^{(K)}$. For $k=1,\ldots,K$, let $F^{(k)}_{N,K}$ be the empirical distribution of $\mathfrak{X}_{N,K}^{(k)}$, and $\mathfrak{X}_{N,K}^{*(k)}=\big\{X^*_{k,1},\ldots,X^*_{k,n_k}\big\}$ be i.i.d.~samples drawn from $F^{(k)}_{N, K}$.  Repeat this procedure $B$ times, one obtains $B$ resampled data subsets $\mathfrak{X}_{N,K}^{*1(k)}, \ldots, \mathfrak{X}_{N,K}^{*B(k)}$. 
Compute the corresponding statistic $T_{N,K}^{*b(k)} = T\big(\mathfrak{X}_{N,K}^{*b(k)}\big)$ for each resampled subset. Averaging them over the $K$ subsets leads to
\begin{align*}
	T_{N,K}^{*b} = N^{-1}\sum_{k=1}^{K}n_kT_{N,K}^{*b(k)}
\end{align*}
for $b = 1, \ldots, B$. Let 
$\hat{\theta}^{(k)}_{N, K}=\theta\big(F_{N,K}^{(k)}\big)$ 
be the analogy of $\theta=\theta(F)$ under $\mathfrak{X}_{N,K}^{(k)}$. Define $\hat{\theta}_{N,K}=N^{-1}\sum_{k=1}^{K}n_k\hat{\theta}^{(k)}_{N,K}$, 
then the empirical distribution of $\{N^{1/2}(T_{N,K}^{*b}-\hat{\theta}_{N, K})\}_{b=1}^B$ is used to approximate the distribution of $N^{1/2}(T_{N, K}-\theta)$. The algorithm is outlined in Table \ref{tab:01} in the supplementary material.
 
We call the above  the distributed bootstrap (DB) because the resampling is fulfilled within each data subset without the sample inflation as BLB. For each data subset in each iteration, we can calculate $T_{N,K}^{*b(k)}$ and $\hat{\theta}^{(k)}_{N, K}$ locally  avoiding data communication between different data subsets. These make the approach suitable for the massive data.

To explore the theoretical properties of the distributed bootstrap, 
we need some notations and assumptions. 
Specifically, 
 $T^{*(k)}_{N,K}=T(\mathfrak{X}_{N,K}^{*(k)})$ is assumed to admit 
\begin{align}\label{eq:T_N_K_boot1}
	T^{*(k)}_{N,K} = \hat{\theta}^{(k)}_{N,K}+n_k^{-1}\sum_{i=1}^{n_k}\hat{\alpha}\big(X^*_{k,i}; F^{(k)}_{N,K}\big)+n_k^{-2}\sum_{1\leq i<j\leq n_k}\hat{\beta}\big(X^*_{k,i},X^*_{k,j};F^{(k)}_{N,K}\big)+R^{*(k)}_{N,K},
\end{align}
where $\hat{\theta}^{(k)}_{N, K}=\theta\big(F^{(k)}_{N, K}\big)$ and $R^{*(k)}_{N, K}=R\big(\mathfrak{X}_{N,K}^{*(k)}; F^{(k)}_{N, K}\big)$ are  analogues of $\theta$ and $R^{(k)}_{N,K}$ under $F^{(k)}_{N, K}$. Then $T^*_{N, K}=N^{-1}\sum_{k=1}^{K}n_kT^{*(k)}_{N,K}$ can be written as
\be
	T^*_{N, K} = \hat{\theta}_{N,K}+N^{-1}\sum_{k=1}^{K}\sum_{i=1}^{n_k}\hat{\alpha}\big(X^*_{k,i};F_{N,K}^{(k)}\big)+N^{-1}\sum_{k=1}^{K}n_k^{-1}\sum_{1\leq i<j\leq n_k}\hat{\beta}\big(X^*_{k,i},X^*_{k,j};F_{N,K}^{(k)}\big)+R^{*}_{N,K}, \label{eq:exp1}
\ee
where $\hat{\theta}_{N,K}=N^{-1}\sum_{k=1}^{K}n_k\hat{\theta}^{(k)}_{N,K}$ and $R^*_{N,K}=N^{-1}\sum_{k=1}^{K}n_kR^{*(k)}_{N,K}$. Here, $\{\hat{\alpha}(x; F_{N,K}^{(k)})\}_{k=1}^K$ and $\{\hat{\beta}(x, y; F_{N,K}^{(k)})\}_{k=1}^K$ are the empirical versions of $\alpha(x; F)$ and $\beta(x, y; F)$, which we regulate in the following condition. Let first denote 
$S(F)$ as the support of $F$.  
\begin{condition}\label{condition:boot1}
	For 
	$k=1,\ldots,K$, $\hat{\alpha}(x; F_{N,K}^{(k)})$ and $\hat{\beta}(x, y; F_{N,K}^{(k)})$ satisfy $\sum_{i=1}^{n_k}\hat{\alpha}(X_{k,i}; F_{N,K}^{(k)})=0$, $\hat{\beta}(x, y; F_{N,K}^{(k)})$ is symmetric in $x$ and $y$, $\sum_{i=1}^{n_k}\hat{\beta}(X_{k,i}, y; F_{N,K}^{(k)})=0$ for any $y\in S(F)$. In addition, as $n_k\to\infty$,  $\sup\limits_{x \in S(F)}|\hat{\alpha}(x;F_{N,K}^{(k)})-\alpha(x;F)|=o_p(1)$ and
	$ \sup\limits_{x,y\in S(F)}|\hat{\beta}(x, y; F_{N,K}^{(k)})-\beta(x, y; F)|=o_p(1).$
\end{condition}

Condition \ref{condition:boot1} indicates that for $X_{k,i}^*$ with distribution $F_{N,K}^{(k)}$, $\E\big\{\hat{\alpha}\big(X_{k,i}^*;F_{N,K}^{(k)}\big)|F_{N,K}^{(k)}\big\}=0$ and $\E\big\{\hat{\beta}\big(X_{k,i}^*,y;F_{N,K}^{(k)}\big)|F_{N,K}^{(k)}\big\}=0$ for any $y\in S(F)$. Moreover, it requires that $\{\hat{\alpha}(x; F_{N,K}^{(k)})\}_{k=1}^K$ and $\{\hat{\beta}(x, y; F_{N,K}^{(k)})\}_{k=1}^K$ are uniformally consistent to $\alpha(x;F)$ and $\beta(x,y;F)$. 
The next theorem, inspired by \citet{Hall1986,Hall1988} and \citet{Lai1993}, establishes the consistency of the distributed bootstrap for the non-degenerated case of $\sigma_a^2 >0$. 

\begin{theorem}\label{theo:edg_T_SC_N_boot1}
	Under the conditions of Theorem \ref{theo:uniform_convergence}, $\E\left|\beta(X_1,X_1;F)\right|^2<\infty$,  $T^*_{N, K}$ admits expansion (\ref{eq:exp1}) with $R^{*}_{N,K}$ satisfying $\P\big\{\big|R^*_{N, K}\big|\geq N^{-1/2}(lnN)^{-1}\big|F^{(1)}_{N,K},\ldots,F^{(K)}_{N,K}\big\}=o_p(1)$. If $\big\{\hat{\alpha}\big(x;F^{(k)}_{N,K}\big)\big\}_{k=1}^K$ and $\big\{\hat{\beta}\big(x,y;F^{(k)}_{N,K}\big)\big\}_{k=1}^K$ satisfy Condition \ref{condition:boot1}, then, as $N\rightarrow\infty$,
	\begin{align*}
		 \sup\limits_{x\in\mathbf{R}}\bigg|\P\bigg\{N^{1/2}\hat{\sigma}^{-1}_{\alpha,N,K}(T^*_{N,K}-\hat{\theta}_{N,K}) \leq x\big| & F^{(1)}_{N,K},\ldots,F^{(K)}_{N,K}\bigg\} \\
		 & -\P\left\{N^{1/2}\sigma_{\alpha}^{-1}(T_{N,K}-\theta) \leq x\right\}\bigg|=o_p(1),
	\end{align*}
	and $\hat{\sigma}^{2}_{\alpha,N,K}-\sigma_{\alpha}^2=o_p(1)$ where $\hat{\sigma}^{2}_{\alpha, N, K} = N^{-1}\sum_{k=1}^{K}n_k\E\big\{\hat{\alpha}^2\big(X^*_{k,1}; F^{(k)}_{N, K}\big)\big|F^{(k)}_{N, K}\big\}$.
\end{theorem}

	Theorem \ref{theo:edg_T_SC_N_boot1} gives theoretical support for the distributed bootstrap for the non-degenerate case. For specific statistics, we need to check if 
	$T_{N,K}^*$ satisfies the conditions in Theorem \ref{theo:edg_T_SC_N_boot1}.  
	Section \ref{sec:U_stat_exp} in the supplementary material provides details checking on the conditions for the U-statistics.  
	The distributed bootstrap works for both finite and infinite $K$, as long as $K$ is of an appropriate order of $N$. The requirement on $K$ is hidden in the assumption that 
$$\P\big\{\big|R^*_{N, K}\big|\geq N^{-1/2}(lnN)^{-1}\big|F^{(1)}_{N,K},\ldots,F^{(K)}_{N,K}\big\}=o_p(1).$$ 
 For a distributed U-statistic of degree two as discussed in the supplementary material, that is equivalent to  $R_{N,K}^{*(k)}$ satisfying $\E\big(R^{*(k)}_{N,K}\big|F_{N,K}^{(k)}\big)=0$ and $\E\big(R^{*(k)2}_{N,K}\big|F_{N,K}^{(k)}\big)=O_p(n_k^{-4})$, for $k=1,\ldots,K$. Thus, the consistency of the distributed bootstrap for the distributed U-statistics 
 is ensured as long as $K=O(N^{\tau'})$ for a positive $\tau'<1$.

Theorem \ref{theo:edg_T_SC_N_boot1} ensures that the distributed bootstrap can be utilized in a broad range by combining it with the continuous mapping theorem and delta method, for instance in the variance estimation and confidence intervals.
Specifically, let $\hat{\sigma}^2_{DB} = B^{-1}\sum_{b=1}^{B}\big(T_{N,K}^{*b} - B^{-1}\sum_{\ell=1}^{B}T_{N,K}^{*\ell}\big)^2$
be the sample variance of $\{T_{N,K}^{*b},~b=1,\ldots,B\}$. Then $\hat{\sigma}^2_{DB}$ is a consistent estimator of $N^{-1}\sigma^2_{\alpha}$, which is a consistent estimator of $\Var(T_{N, K})$.  In addition, denote $u^*_{\tau}$ as the lower $\tau$ quantile of the empirical distribution of  $\big\{N^{1/2}\big(T_{N,K}^{*b}-\hat{\theta}_{N,K}\big),~b=1,\ldots,B\big\}$, then an equal-tail two-sided confidence interval for $\theta$ with level $1-\tau$ can be constructed as
$$ \left(T_{N, K}-N^{-1/2}u^*_{1-\tau/2},~T_{N, K}-N^{-1/2}u^*_{\tau/2}\right). $$
	
For degenerate statistics,  
the distributed bootstrap can be adapted by replacing $\big(T_{N,K}^{*(k)}-\hat{\theta}_{N,K}^{(k)}\big)$ with $n_k^{-2}\sum_{1\leq i<j\leq n_k}\hat{\beta}(X_{k,i}^*,X_{k,j}^*;F_{N,K}^{(k)})$ for each data block. Similar to Theorem \ref{theo:T_N_K_degenerate}, the consistency of the distributed bootstrap for the degenerate case needs to be discussed separately for $K$ being finite or $K\to\infty$. 
When $K$ is fixed, the conditional distribution of $2\sum_{k=1}^{K}n_k^{-1}\sum_{1\leq i<j\leq n_k}\hat{\beta}\big(X^*_{k,i},X^*_{k,j};F_{N,K}^{(k)}\big)$ is consistent to the distribution of $2N(T_{N,K}-\theta)$ if $\P\big\{\big|R^*_{N, K}\big|\geq N^{-1}(lnN)^{-1}\big|F^{(1)}_{N,K},\ldots,F^{(K)}_{N,K}\big\}=o_p(1)$. When $K\to\infty$ and $\P\big\{\big|R^*_{N, K}\big|\geq K^{1/2}N^{-1}(lnN)^{-1}\big|F^{(1)}_{N,K},\ldots,F^{(K)}_{N,K}\big\}=o_p(1)$,  that  conditional distribution of  $2^{1/2}K^{-1/2}\sum_{k=1}^{K}n_k^{-1}\sum_{1\leq i<j\leq n_k}\hat{\beta}\big(X^*_{k,i},X^*_{k,j};F_{N,K}^{(k)}\big)$ can approximate the distribution of $2^{1/2}K^{-1/2}N(T_{N,K}-\theta)$ can be established. 
In the next subsection, we will propose a pseudo-distributed bootstrap that works for both degenerate and non-degenerate cases.

\subsection{Pseudo-distributed bootstrap}

Although the distributed bootstrap lead to substantial  computational saving due to its distributive nature, it is still computationally involved  when the size of the data is massive since the distributed statistics need to be re-calculated for each bootstrap replication. 
 To further reduce the computational burden, we consider another way to approximate the distribution of $T_{N,K}$ under the case of diverging $K$. 

The idea comes from the expression
\begin{align*}
	T_{N,K}=N^{-1}\sum_{k=1}^{K}n_kT_{N,K}^{(k)}=K^{-1}\sum_{k=1}^K(Kn_k/N)T_{N,K}^{(k)},
\end{align*}
which indicates that $T_{N,K}$ is the average of $K$ independent random variables. Hence,  when $K$ is large, approximating the distribution of $T_{N,K}$ is similar to that  of the sample mean of independent but not necessary  identically distributed samples. 
This leads us to propose  directly resampling $\big\{T_{N,K}^{(k)}\big\}_{k=1}^K$ rather than the original data.

We first consider the non-degenerate case of $\sigma_{\alpha}^2>0$. As different subsets of data may have different sample sizes,  we need to scale $\big\{T_{N,K}^{(k)}\big\}_{k=1}^K$ before the resampling. Let 
$$\mathcal{T}_{N,K}^{(k)}=N^{-1/2}K^{1/2}n_kT_{N,K}^{(k)} \quad \hbox{for} \quad k=1,\ldots,K$$ 
 and $F_{K, \mathcal{T}}$ be the empirical distribution of $\{\mathcal{T}_{N, K}^{(k)}\}_{k=1}^K$. To estimate the distribution of $N^{1/2}(T_{N, K}-\theta)$, we  draw $B$ independent resamples $\big\{\mathcal{T}_{N, K}^{*b(1)},\ldots,\mathcal{T}_{N, K}^{*b(K)}\big\}$ for $b=1,\ldots,B$ from $F_{K, \mathcal{T}}$. Let $\mathcal{T}_{N,K}^{*b}=K^{-1}\sum_{k=1}^K\mathcal{T}_{N, K}^{*b(k)}$,  
 then the empirical distribution of $\big\{K^{1/2}\big(\mathcal{T}_{N,K}^{*b}-N^{1/2}K^{-1/2}T_{N,K}\big)\big\}_{b=1}^B$ is used to estimate the distribution of $N^{1/2}(T_{N, K}-\theta)$. We call this algorithm the pseudo-distributed bootstrap (PDB) and its procedure is summarized in Table \ref{tab:02} in the supplementary material.

The PDB is the conventional bootstrap carried out on the scaled subset-wise statistics $\big\{\mathcal{T}_{N, K}^{(k)}\big\}_{k=1}^K$, which are independent but not necessary identically distributed. The bootstrap procedures under non-i.i.d.~models has been studied in \citet{Liu1988}. 
The following theorem establishes the asymptotic properties of the PDB for the non-degenerate case. 

\begin{theorem}\label{corr:pdb1}
	Under the conditions of Theorem \ref{theo:uniform_convergence}.
	Suppose that $\mathcal{T}_{N, K}^{*(1)},\ldots, \mathcal{T}_{N, K}^{*(K)}$ is an i.i.d.~sample from $F_{K, \mathcal{T}}$ and $\mathcal{T}_{N,K}^{*}=K^{-1}\sum_{k=1}^K\mathcal{T}_{N, K}^{*(k)}$. If $\E\left|\beta(X_1, X_2; F)\right|^{2+\delta}<\infty$, $\sup_k\E\big|n_k^{1/2}R_{N,K}^{(k)}\big|^{2+\delta}<\infty$, and $\sup_kN^{-1/2}K^{1/2}|n_k-NK^{-1}|\to0$ as $N\to\infty$, then with probability $1$, as $K\rightarrow\infty$,
	\begin{align*}
		\sup\limits_{x \in \mathbf{R}}\left|\P\left\{K^{1/2}(\mathcal{T}_{N,K}^{*}-N^{1/2}K^{-1/2}T_{N,K}) \leq x\big|F_{K, \mathcal{T}}\right\} - \P\left\{N^{1/2}(T_{N,K}-\theta) \leq  x\right\}\right|= o(1).
	\end{align*}	
\end{theorem}	

Theorem \ref{corr:pdb1} shows that under moderate conditions on $n_k$, $K$ and the moments of $T_{N,K}^{(k)}$, the PDB is asymptotic consistent when $\sigma^2_{\alpha}>0$. As we have mentioned earlier, the subset based statistics $\big\{T_{N,K}^{(k)}\big\}_{k=1}^K$ need to be re-scaled first by their respective sample sizes. The condition $\sup_kN^{-1/2}K^{1/2}|n_k-NK^{-1}|\to0$ ensures that the size of each subset should not be too different from each other.

	Compared to the distributed bootstrap, besides computing saving, an appealing property of the PDB is that it does not require the bootstrap distributed statistics admitting form \eqref{eq:exp1} 
	as assumed in Theorem \ref{theo:edg_T_SC_N_boot1}. This makes the PDB more versatile and easier to verify, although it does requires that $K$  diverges.

Another advantage of the PDB is that it can be applied for the degenerate  case. Indeed, when $\sigma_{\alpha}^2=0$ but $\sigma_{\beta}^2>0$, $T_{N,K}$ is still an average of $K$ independent random variables. 
One only needs to rescaled $\big\{T_{N,K}^{(k)}\big\}_{k=1}^K$ with a different scaling factor such that $\mathcal{T}_{N,K}^{(k)}=n_kT_{N,K}^{(k)}$ for $k=1,\ldots,K$, followed by the same procedures  as those for the non-degenerate case.  The next theorem gives theoretical support for the PDB 
under degeneracy.

\begin{theorem}\label{corr:pdb2}
	For $T_{N,K}$ in (\ref{eq:T_N_K}) with $\sigma^2_{\alpha}=0$ but $\sigma^2_{\beta}>0$, under Conditions \ref{condition:alpha} and \ref{condition:R_N} such that  $\E(R_N)=b_1N^{-\tau_1}+o(N^{-\tau_1})$ and $\Var(R_N)=o(N^{-\tau_2})$ for some $\tau_1>1$, $\tau_2\geq2$, assume $n_k=NK^{-1}$ for $k=1,\ldots,K$ and $K=O(N^{\tau'})$ for some positive constant $\tau'<1-1/(2\tau_1-1)$. 
	Suppose that $\mathcal{T}_{N, K}^{*(1)},\ldots, \mathcal{T}_{N, K}^{*(K)}$ is an i.i.d.~sample from $F_{K, \mathcal{T}}$ and denote $\mathcal{T}_{N,K}^{*}=K^{-1}\sum_{k=1}^K\mathcal{T}_{N, K}^{*(k)}$. Assume $\E\left|\beta(X_1, X_2; F)\right|^{2+\delta'}<\infty$ and $\sup_k\E\big|n_kR_{N,K}^{(k)}\big|^{2+\delta'}<\infty$ for some positive constant $\delta'$, then with probability $1$, as $K\rightarrow\infty$,
	\begin{align*}
		\sup\limits_{x \in \mathbf{R}}\left|\P\left\{K^{1/2}(\mathcal{T}_{N,K}^{*}-NK^{-1}T_{N,K}) \leq x\bigg|F_{K, \mathcal{T}}\right\} - \P\left\{NK^{-1/2}(T_{N,K}-\theta) \leq  x\right\}\right|= o(1).
	\end{align*}
\end{theorem}

Compared to the case of $\sigma^2_{\alpha}>0$, stronger conditions are needed for degenerate case. This is reflected firstly in that $\tau_1$ needs to be  strictly larger than $1$ such that $R_{N,K}$ can be dominated by the quadratic term involving $\beta$. Secondly, $n_k$ is assumed to be the same for all subsets and $K$ is required to have a slower growth rate to $N$. Lastly,  a stronger moment condition is needed for $\big\{R_{N,K}^{(k)}\big\}_{k=1}^K$. 
That its applicable for both non-degenerate and degenerate statistics reveals the versatility of the PDB. 

Before  concluding this section,  we compare the distributed bootstrap (DB), the PDB, BLB and SDB in approximating the distribution of $T_{N,K}$. First, we compare them regarding their computational complexity. To make our discussion more straightforward, we assume that the entire data are divided evenly into $K$ blocks such that $n_1=\cdots=n_K=K^{-1}N$.

\renewcommand{\arraystretch}{1}
\begin{table}[h]
	\centering
	\caption{ Computing complexity of the distributed bootstrap (DB), the pseudo-distributed bootstrap (PDB), the BLB and the SDB.}
	\begin{tabular}{|c|c|c|c|}
		\hline
		Methods &  Computing Complexity & Distributed or Not & Requirement on $K$ \\\hline
		DB & $(B+1)K\times t(K^{-1}N)$ & distributed & finite or diverging  \\
		PDB & $K\times t(K^{-1}N)+BK$ & distributed & diverging \\
		BLB & $S\left\{BK\times t(K^{-1}N)+t(n)\right\}$ & distributed & finite or diverging \\
		SDB & $S\left\{K\times t(K^{-1}N)+t(n)\right\}$ & not distributed & finite or diverging \\\hline
	\end{tabular}
	\label{tab:03}
\end{table}

Table 3 reports the computational complexity of the four methods together with specification on their distributive nature and the requirement on the total number of subsets $K$.  
Suppose that the time cost for calculating the statistic in \eqref{eq:T_N} on a sample of size $m$ is $t(m)$, then the computing time of BLB with $B$ resamples and $S$ subsets is $S\left\{BK\times t(K^{-1}N)+t(n)\right\}$, where $n$ is the size of the smaller subset $\mathfrak{X}_{s, n}$. For SDB with $S$ random subsets of size $n$, its computational complexity is $S\left\{K\times t(K^{-1}N)+t(n)\right\}$.  
If we generate $B$ resamples in the distributed bootstrap, its computational cost would be $(B+1)K\times t(K^{-1}N)$ for the distributed bootstrap. In addition, the PDB 
requires $K\times t(K^{-1}N)+BK$ time to fulfill its implementation with $B$ pseudo resamples. Thus, the PDB  is the fastest among the four methods. 
The distributed bootstrap and the SDB have similar computational cost if $S$ is close to $B$. However, the SDB can not be implemented distributively, while the other three methods can. 
For the BLB, its computational complexity depends on both $S$ and $B$. 
Furthermore, only the PDB requires $K$ diverges, while the other three  work for both  $K$ being finite or diverging. 

\section{Simulation studies}\label{sec:simulation}

In this section, we report simulation studies to evaluate the performance of the proposed distributed approaches and the associated bootstrap algorithms.  In particular, we compare the proposed bootstrap methods to the BLB and  SDB  approaches. The statistic we considered was the Gini's mean difference
\begin{align*}
	U_N = 2\{N(N-1)\}^{-1}\sum_{1\leq i<j\leq N}|X_i - X_j|.
\end{align*}
It is an unbiased estimator of a dispersion parameter $\theta_U=\E|X_i-X_j|$ which measures the variability of a distribution. In addition, $U_N$ is a U-statistic of degree two and can be expanded to  the form \eqref{eq:T_N}. 
Suppose the entire data are divided into $K$ data blocks with the $k$-th having size $n_k$. For each subset, denote $U_{N,K}^{(k)}$ as the Gini's mean difference obtained from the $k$-th data block, then the distributed Gini's mean difference estimator is $U_{N,K}=N^{-1}\sum_{k=1}^{K}n_kU_{N,K}^{(k)}$. 

When $\E X_i^4<\infty$, we can use the jackknife \citep{Efron1981} to estimate $\Var(U_N)$, denoted as $N^{-1}S_{U_N}^2$ where
	$S_{U_N}^2 = 4(N-1)(N-2)^2\sum_{i=1}^{N}\big\{(N-1)^{-1}\sum_{j\neq i}|X_i-X_j|-U_N\big\}^2$ 
which can be written in the form of (\ref{eq:T_N}) with the quadratic term absorbed in the remainder term \citep{Callaert1981}. {As an estimator of $\Var(U_N)$, $N^{-1}S_{U_N}^2$ has a bias of order $O(N^{-2})$. Because $\Var(N^{-1}S_{U_N}^2)=O(N^{-3})$, the MSE of $N^{-1}S_{U_N}^2$ is of order $O(N^{-3})$. 

Three distributions of $X_i$ were experimented: 
 (I) $\mathcal{N}(1,1)$;  
(II) $Gamma(3,1)$ 
and (III) $Poisson(4)$. 
The true values of $\theta$ and $\Var(U_{N,K})$ can be calculated algebraically for the first two distributions. For the Poisson distribution, they can be well approximated 
\citep{Lomnicki1952}. We chose the sample size $N=100,000$.   For the sake of convenience, we randomly divided the entire dataset into blocks of equal size. 
The number of blocks $K\in\mathcal{K}=\{5, 10, 20, 50, 100, 200, 500, 1000, 2000, 5000\}$. 
All the simulation results were based on $2000$ replications. All the simulation experiments  were conducted in R with a single Intel(R) Core(TM) i7 4790K @4.0 GHz processor.

Figure \ref{fig:mse} summarizes the empirical performance of $U_{N,K}$ and $N^{-1}S_{U_{N,K}}^2$ for the three distributions relative to those of  $U_N$ and $N^{-1}S_{U_N}^2$ based on the full sample, respectively. 
The figure presents the relative MSE of $U_{N,K}$, and the relative MSE, bias and variance of 
 $N^{-1}S_{U_{N,K}}^2$  with respective to their full sample counterparts. 

\begin{figure}[h]
	\centering
	\subfigure{%
		\includegraphics[width=0.4\linewidth]{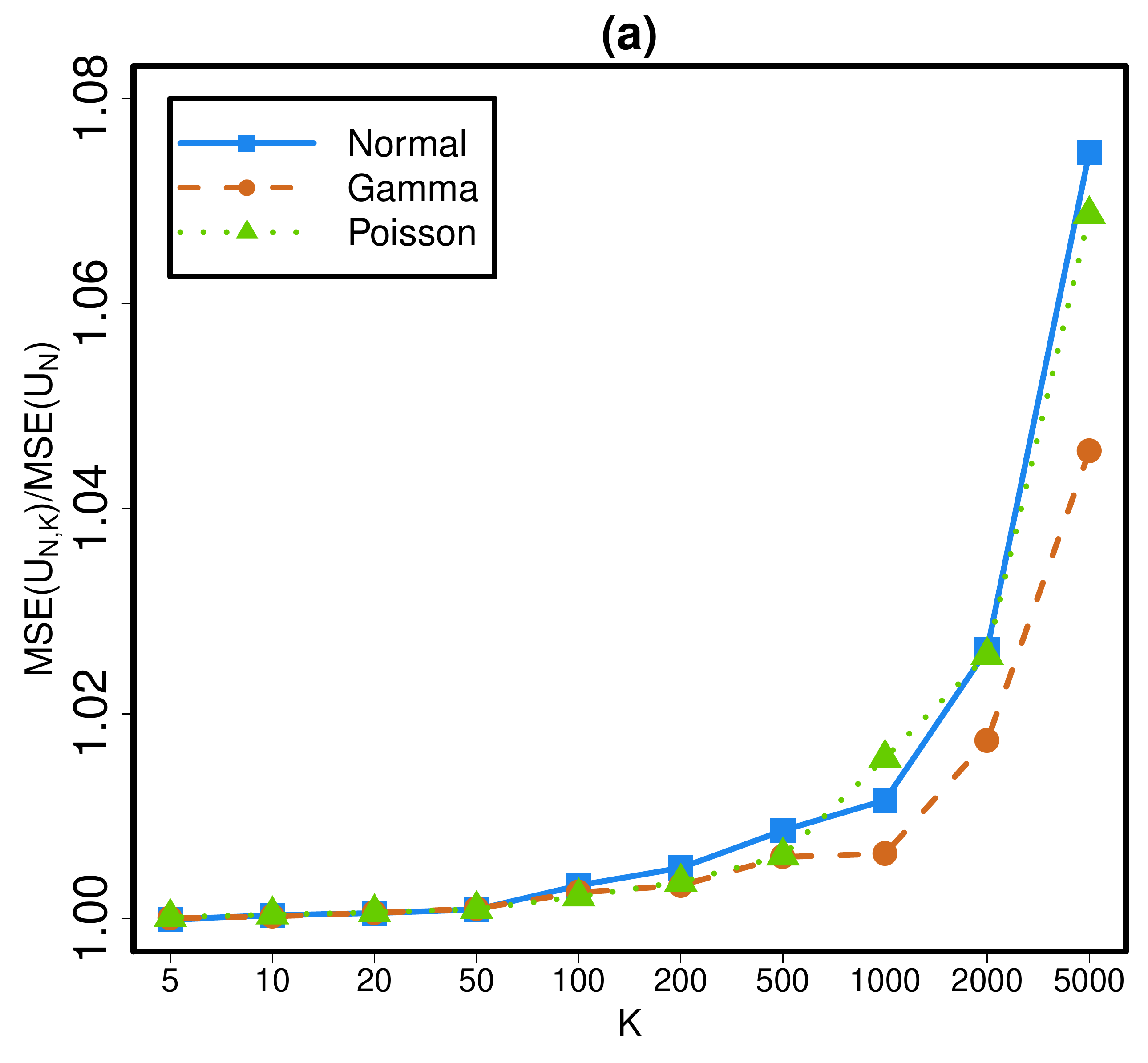}
		\label{fig:mse_mean}}
	\quad
	\subfigure{%
		\includegraphics[width=0.4\linewidth]{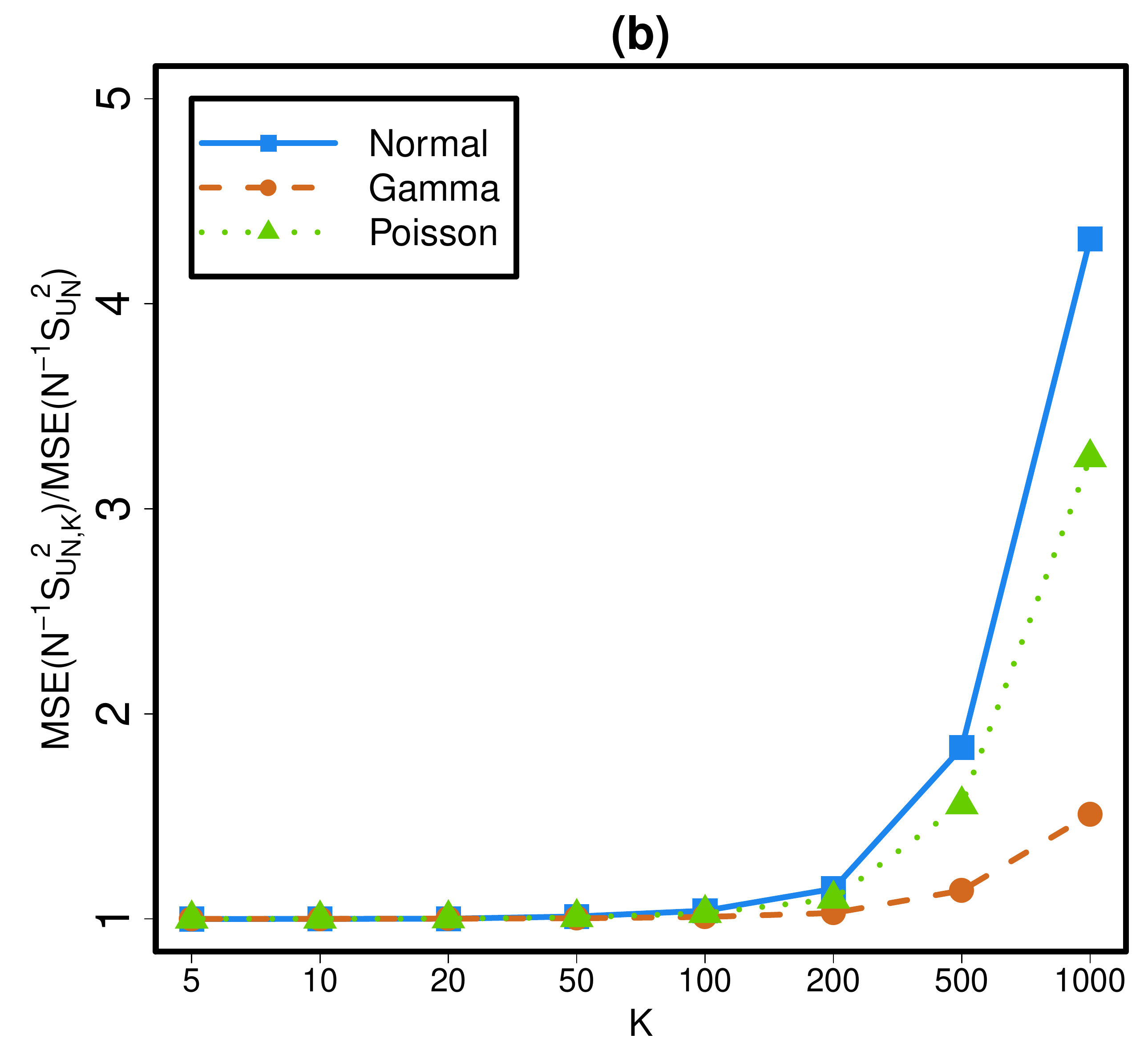}
		\label{fig:mse_var}}
	\quad
	\subfigure{%
		\includegraphics[width=0.4\linewidth]{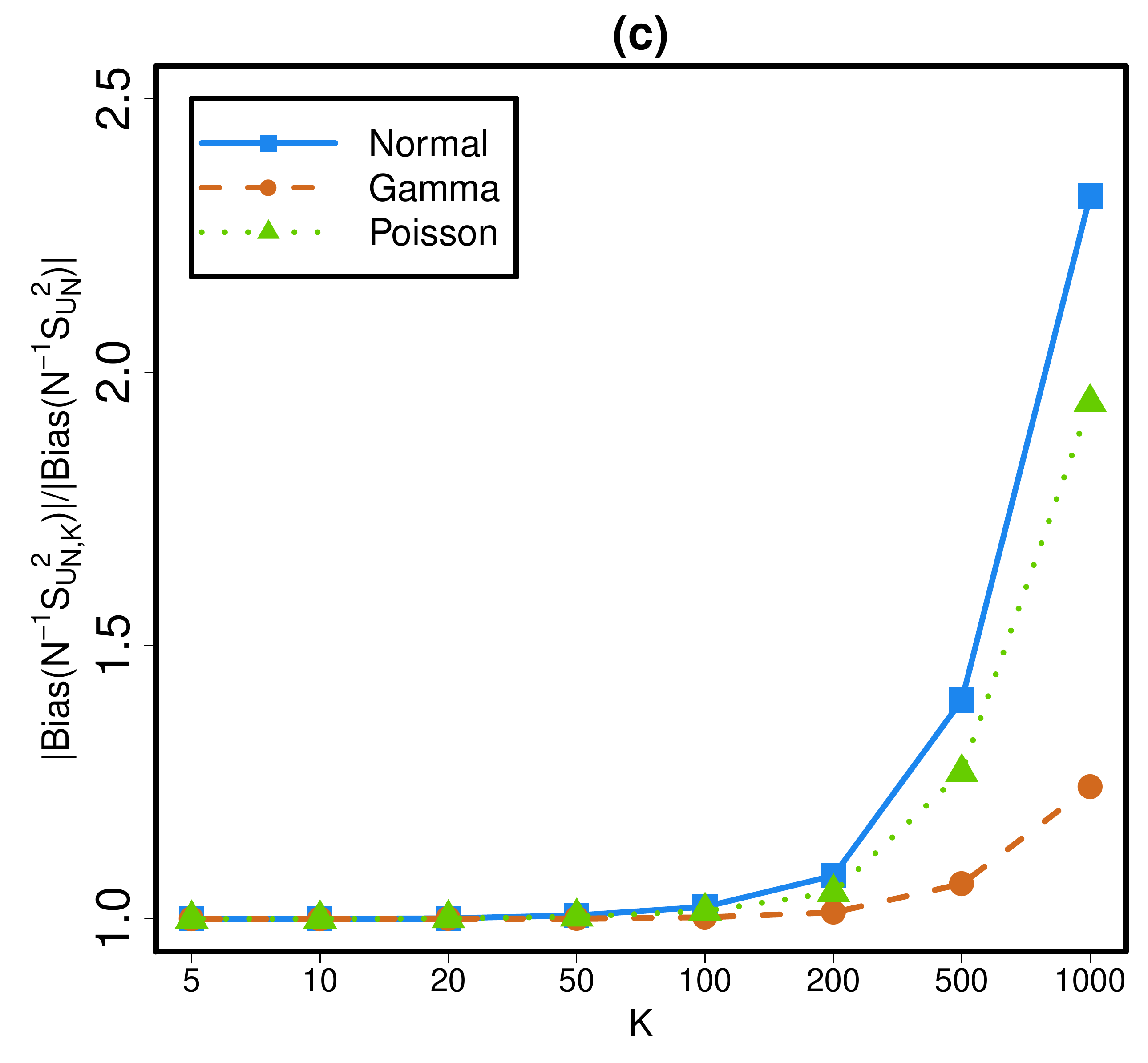}
		\label{fig:bias_var}}
	\quad
	\subfigure{%
		\includegraphics[width=0.4\linewidth]{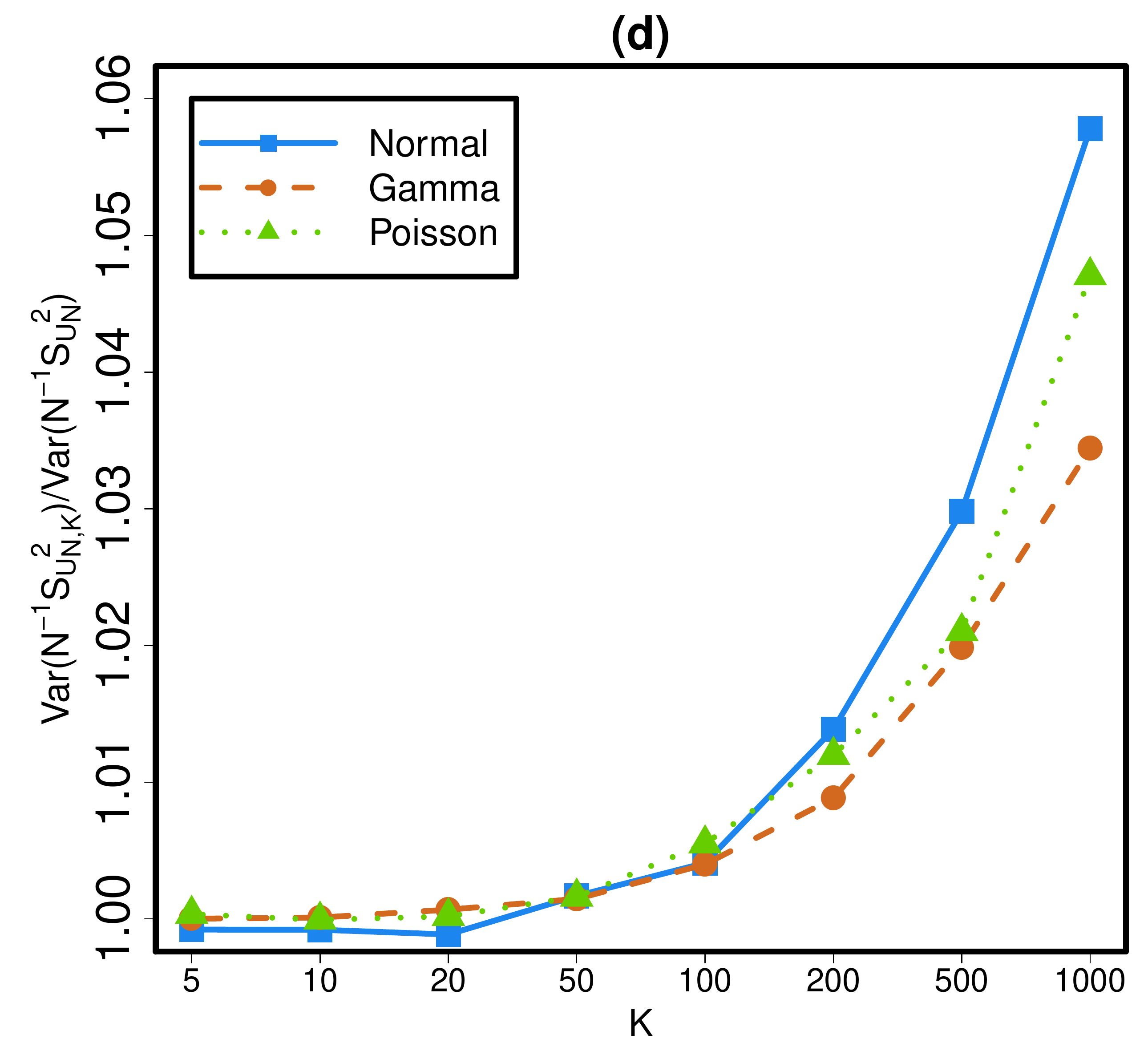}
		\label{fig:var_var}}
	\quad
	\caption{ Relative MSEs of $U_{N, K}$ (a) and  $S^2_{U_{N,K}}$  (b), absolute relative bias (c) and variance  (d) of $N^{-1}S^2_{U_{N,K}}$, to their full sample counterparts  $U_N$ and $S^2_{UN}$ with respect to $K$. 
	}
	\label{fig:mse}
\end{figure}

Figure \ref{fig:mse_mean} shows that the relative $\MSE(U_{N, K})$  increased quite slowly as $K$ was increased before being too large. This indicated that  the distributed estimator $U_{N, K}$ was almost as good as $U_N$ for relatively small $K$, which coincided with our theoretical results that $U_{N, K}$ is unbiased and the variance increase occurs only in the second order term. 
From the plot we can also  see that even $K$ was as large as $5000$, 
the relative increases in the MSEs were all less than $10\%$ for all three distributions. Thus,  $U_{N, K}$ was a reasonable substitute of $U_N$ as an estimator of $\theta$ in these situations.   
Figure \ref{fig:mse_var} shows that 
 the MSEs of 
$N^{-1}S_{U_{N,K}}^2$ was comparable to that of  $N^{-1}S_{U_N}^2$ for relatively small $K$. However, $\MSE(N^{-1}S_{U_{N,K}}^2)$ increased rapidly when $K$ was lager than $200$, 
which confirmed  the theoretical result that $\MSE(N^{-1}S_{U_{N,K}}^2)$ increases at the order of $K^2$ when $KN^{-1/2}$ diverges. In addition, Figures \ref{fig:bias_var} and \ref{fig:var_var} depict the variations of the absolute bias and variance of $N^{-1}S_{U_{N,K}}^2$, which both contribute to its mean square error. By comparing Figure \ref{fig:mse_var} to Figure \ref{fig:var_var}, we see that the increment in the mean square error of $N^{-1}S_{U_{N,K}}^2$ was mainly due to the increase in the bias when $K$ was relatively large. 

Now we turn to the performance of the distributed bootstrap (DB) and the pseudo-distributed bootstrap (PDB), and compare them with the BLB and SDB. In the simulations, we  constructed the $95\%$ equal-tailed  confidence intervals for $\theta$ based on $U_{N,K}$ with $K\in\mathcal{K}=\{20, 50, 100, 200, 500, 1000\}$. 
For each simulated dataset and $K\in\mathcal{K}$, 
each method was allowed to run for $10$ seconds in order to mimic the fixed time budget scenario. 
For the BLB, we fixed 
 $B=100$ as in \citet{Kleiner2014} and \citet{Sengupta2015}; and the $s$-th subset $\mathfrak{X}_{s,n}$ had size $N/K$. 
For the SDB, the size of 
  $\mathfrak{X}_{s,n}^*$ was also $N/K$. Table \ref{tab:simu_iter} summarized the number of iterations completed for each method within the $10$ second budget for different $K$.  The table shows that the pseudo-distributed bootstrap (PDB) was the fastest that had the most completed iterations among the four methods. The BLB was the slowest. For $K=20$, the BLB could not finish one iteration with the time budget of $10$ seconds. The distributed bootstrap (DP) and SDB had similar performance. However, it is worth mentioning that these results did  not take the time expenditure in data communication among different data blocks into account. 

\renewcommand{\arraystretch}{1}
\begin{table}[h]
	\centering
	\caption{Number of completed iterations with respect to the block size $K$ in $10$ seconds for four bootstrap methods: the distributed bootstrap (DB), the pseudo-distributed bootstrap (PDB), the bag of little bootstrap (BLB) and the subsampled double bootstrap (SDB).}
	\label{tab:simu_iter}
	\begin{tabular}{c|cccccc}
		\hline
		& \multicolumn{6}{c}{K} \\
		& $20$ & $50$ & $100$ & $200$ & $500$ & $1000$ \\\hline
		DB & $47$ & $131$ & $265$ & $513$ & $1067$ & $1569$ \\
		PDB & $5000+$ & $5000+$ & $5000+$ & $5000+$ & $5000+$ & $5000+$ \\
		BLB & $0$ & $1$ & $3$ & $5$ & $11$ & $16$ \\
		SDB & $51$ & $128$ & $249$ & $472$ & $1015$ & $1436$ \\\hline
	\end{tabular}
\end{table}

We next evaluate the inference offered by these four bootstrap methods under a  fixed time budget of 10 seconds.  We tried to run $500$ simulation replications which produced the coverage probabilities and widths of the $95\%$ confidence intervals for $\theta$ in Table \ref{tab:simu_cr_gau} for the Gaussian scenario, while the results for Gamma and Poisson distributions can be found in Tables \ref{tab:simu_cr_gam} and \ref{tab:simu_cr_pois} in the supplementary material. 
Table \ref{tab:simu_cr_gau} shows that there was some under-coverage for the confidence intervals by all four methods 
 for relatively small $K$. For $K=20$, the smallest number of data blocks considered in the simulation, the reason for  the DP and SDB not having good coverage performance  was that they completed too few bootstrap iterations 
  to ensure good accuracy. As the pseudo-distributed bootstrap (PDB) is based on the assumption that $K$ diverges, the coverage probability was not good when $K=20$ even it can complete enough iterations in $10$ seconds. As the number of subsets  $K$ increased, and the computational burden was alleviated, the performances of the distributed bootstrap (DP), PDB and SDB all got better with comparable to coverage and width. The best coverage appeared to be reached at $K=100$. BLB could  not produce any result for $K=20$ as it could not complete a single bootstrap run. For larger $K$, the coverage probabilities of the BLB were slightly less comparable to the other three methods. 
  It is interesting to see that the PDB was able to take advantage of its speedy computation to complete more runs of the bootstrap resampling that led to its producing very comparable coverage levels and width of the confidence intervals.  

\renewcommand{\arraystretch}{1} 
\begin{table}[h]
	\centering
	\caption{Coverage probabilities and widths (in parentheses) of the $95\%$ confidence intervals of the four bootstrap methods with $10$ seconds time budget for the Gaussian data.}
		\begin{tabular}{c|cccccc}
			\hline
			& \multicolumn{6}{c}{K} \\
			& $20$ & $50$ & $100$ & $200$ & $500$ & $1000$ \\\hline
			DB & $0.918$ & $0.934$ & $0.940$ & $0.926$ & $0.938$ & $0.936$ \\
			& $(0.00960)$ & $(0.00981)$ & $(0.00991)$ & $(0.00993)$ & $(0.00994)$ & $(0.00993)$ \\
			\hline
			PDB & $0.918$ & $0.940$ & $0.944$ & $0.936$ & $0.940$ & $0.938$ \\
			& $(0.00962)$ & $(0.00989)$ & $(0.00994)$ & $(0.00998)$ & $(0.01000)$ & $(0.01004)$ \\
			\hline
			BLB & NA & $0.928$ & $0.932$ & $0.934$ & $0.934$ & $0.930$ \\
			& (NA) & $(0.00958)$ & $(0.00967)$ & $(0.00973)$ & $(0.00982)$ & $(0.00978)$ \\
			\hline
			SDB & $0.916$ & $0.932$ & $0.942$ & $0.938$ & $0.940$ & $0.940$ \\
			& $(0.00965)$ & $(0.00984)$ & $(0.00990)$ & $(0.00996)$ & $(0.00998)$ & $(0.00998)$ \\\hline
		\end{tabular}
	\label{tab:simu_cr_gau}
\end{table}

Next we evaluate  the relative errors of the confidence interval widths for each method. Let  $d$ be the true width,  $\hat{d}$ be the width of the $95\%$ confidence interval by one of the four methods, and the relative error is $|\hat{d}-d|/d$. 
{ The exact width $d$ was obtained by $5000$ simulations for each distribution.} 
The relative errors were averaged over the $500$ simulation replications. 
Follow the strategy in \citet{Sengupta2015}, these four methods are compared with respect to the time evolution of relative errors {under the fixed budget of 10 seconds}. This was to explore which method can produce more precise result given a fixed time budget. 
For the BLB and SDB, one iteration means the completion of estimation procedure for one subset. 
 The relative error was initially  assigned to be $1$ before the first iteration was completed. 

Figure \ref{fig:gau_ciw} displays the evolution of the relative errors in the widths of the confidence intervals with respect to time for the four methods  
 with different number block size $K$ for Gaussian $\mathcal{N}(1,1)$ data. We note that the pseudo-distributed bootstrap (PDB) was the fastest method to converge for all $K$.  
BLB's relative error was severely affected by its rather limited completion rates of the bootstrap iteration  when $K =20, 50$ and $100$. 
When $K=20$, 
the distributed bootstrap (DB) and SDB had smaller relative errors than the PDB after $7$ seconds.
This is expected as the convergence rate of the PDB relies on a larger $K$. However, as $K$ was increased to larger than 200, the relative errors of the PDB quickly decreased to an acceptable rate and became comparable to those of the DB and SDB towards the 10 seconds. 
DB and SDB had similar performance and they could produce stable results with a sufficient time budget. 
Results for Gamma and Poisson scenarios were similar and are reported in Figures \ref{fig:gam_ciw} and \ref{fig:pois_ciw} in the supplementary material.

\begin{figure}[htp]
	\centering
	\subfigure{%
		\includegraphics[width=0.45\linewidth]{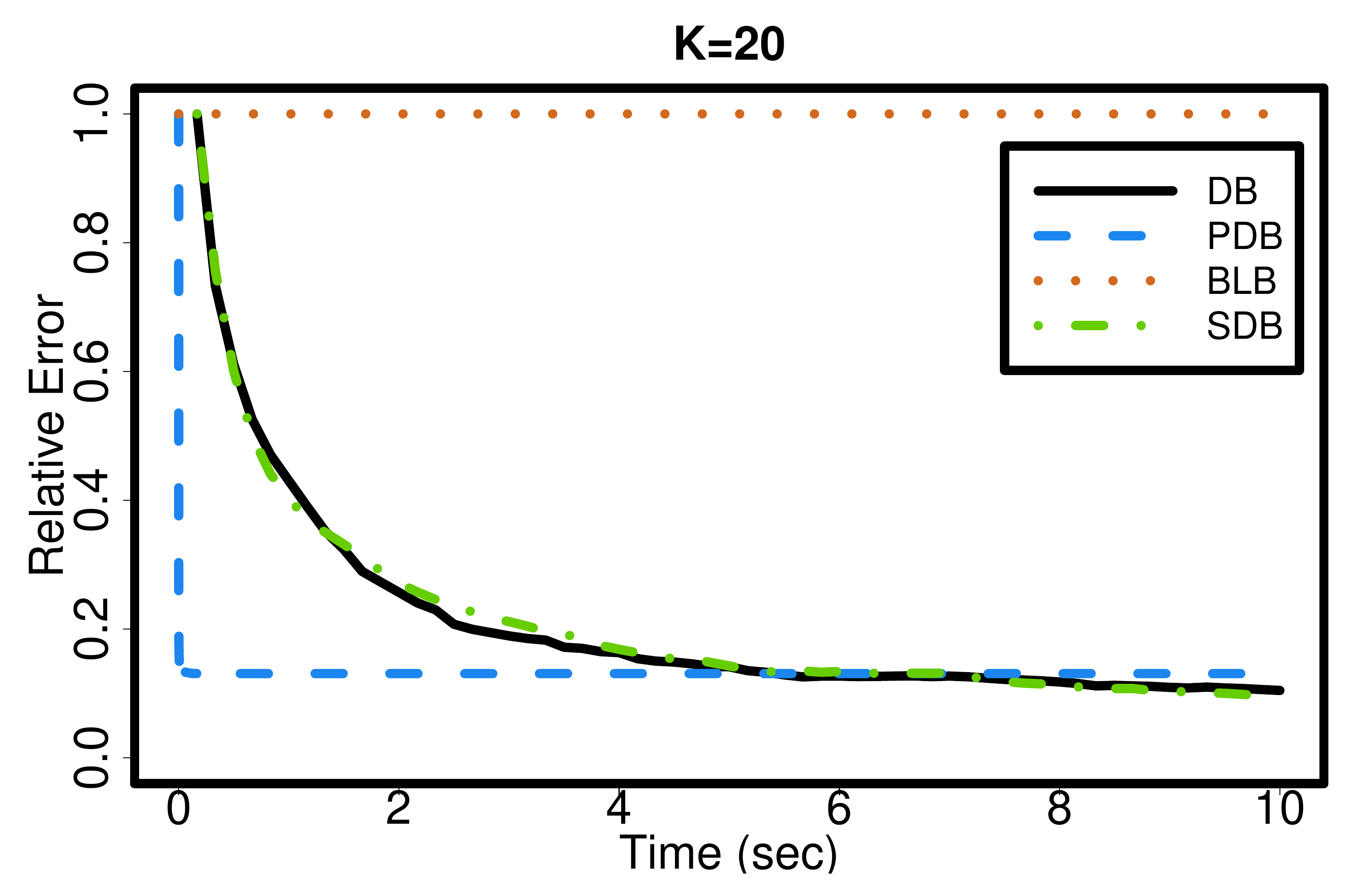}
		\label{fig:gau_K_20}}
	\quad
	\subfigure{%
		\includegraphics[width=0.45\linewidth]{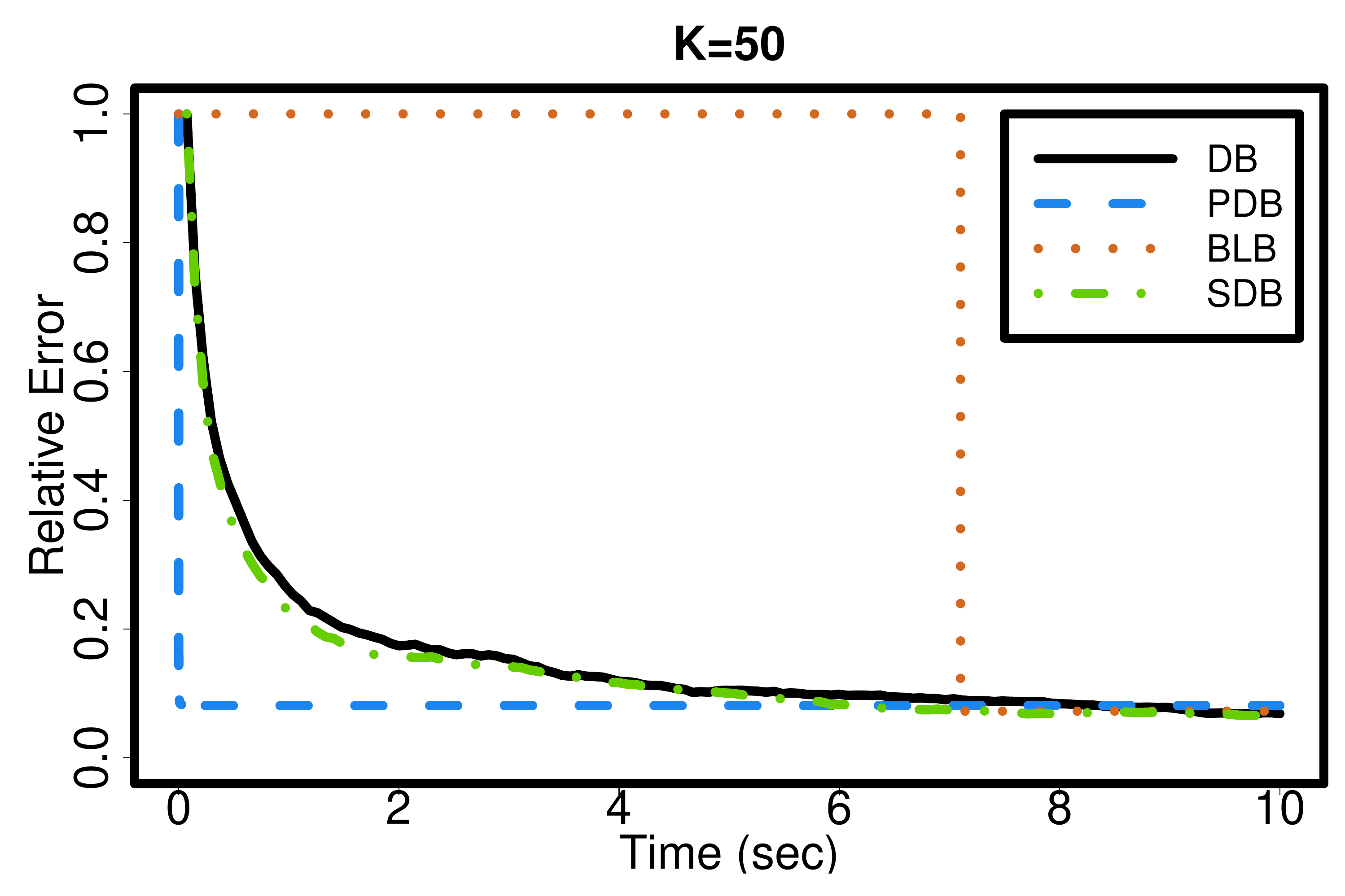}
		\label{fig:gau_K_50}}
	\quad
	\subfigure{%
		\includegraphics[width=0.45\linewidth]{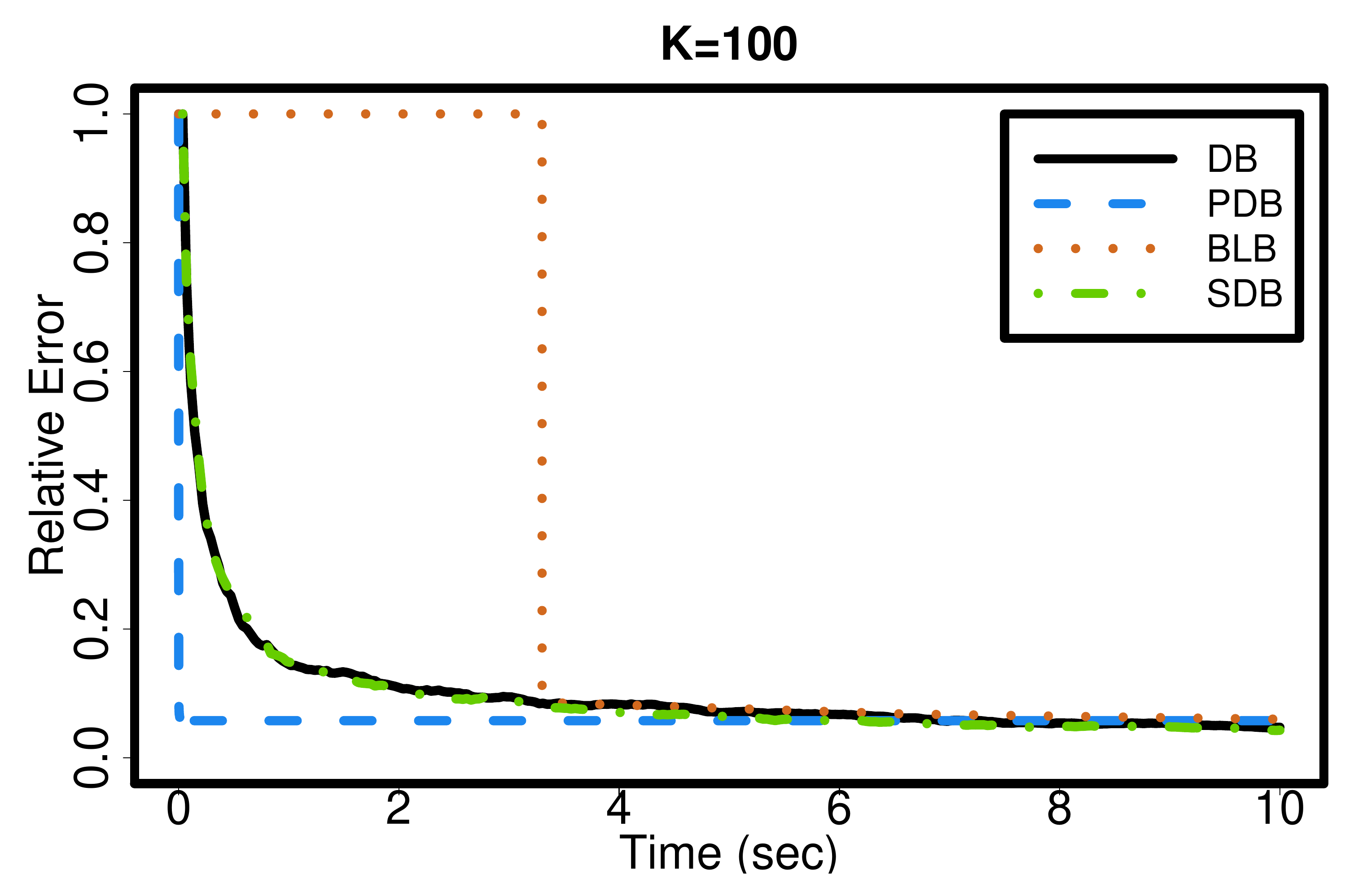}
		\label{fig:gau_K_100}}
	\quad
	\subfigure{%
		\includegraphics[width=0.45\linewidth]{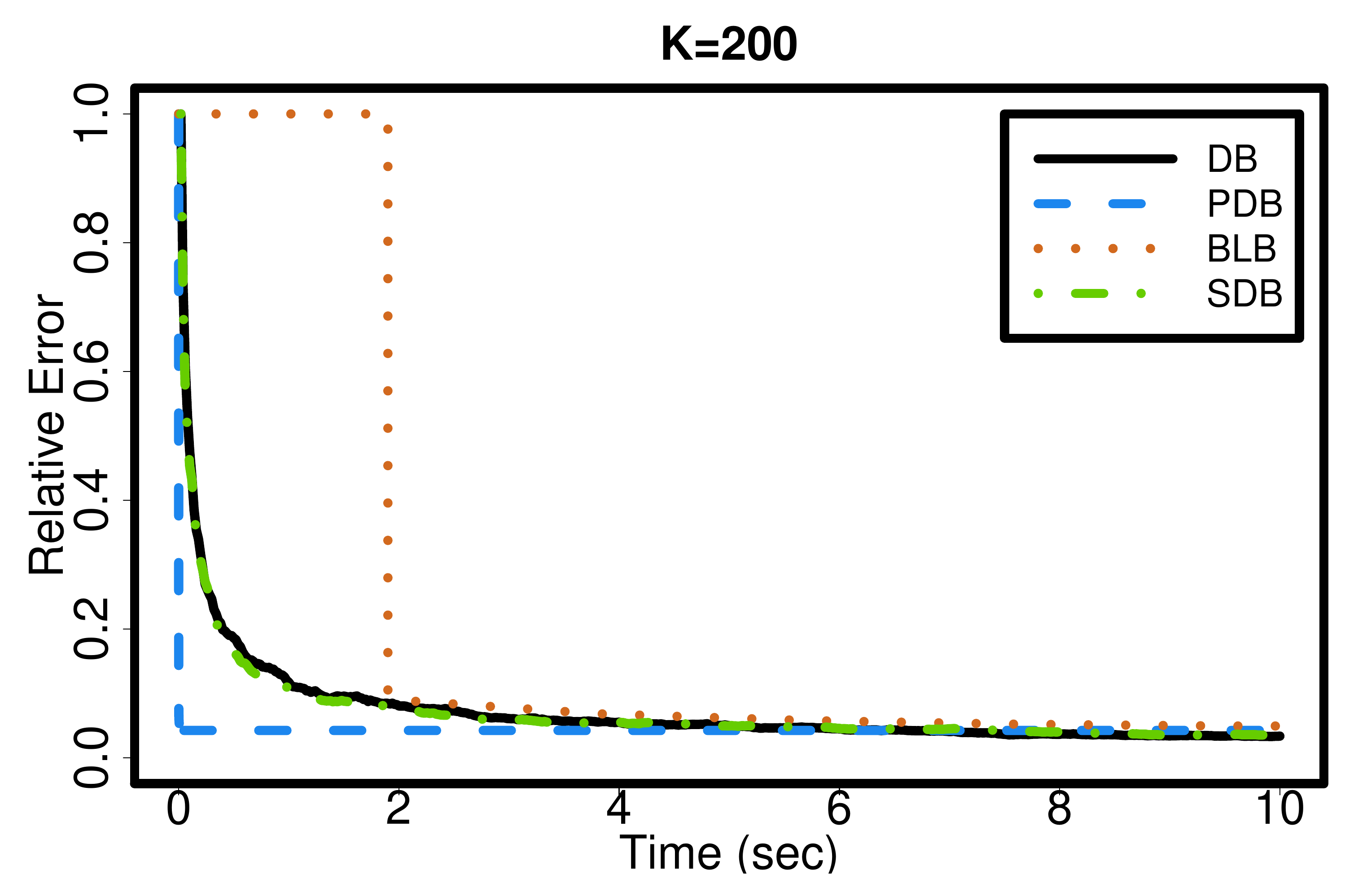}
		\label{fig:gau_K_200}}
	\quad
	\subfigure{%
		\includegraphics[width=0.45\linewidth]{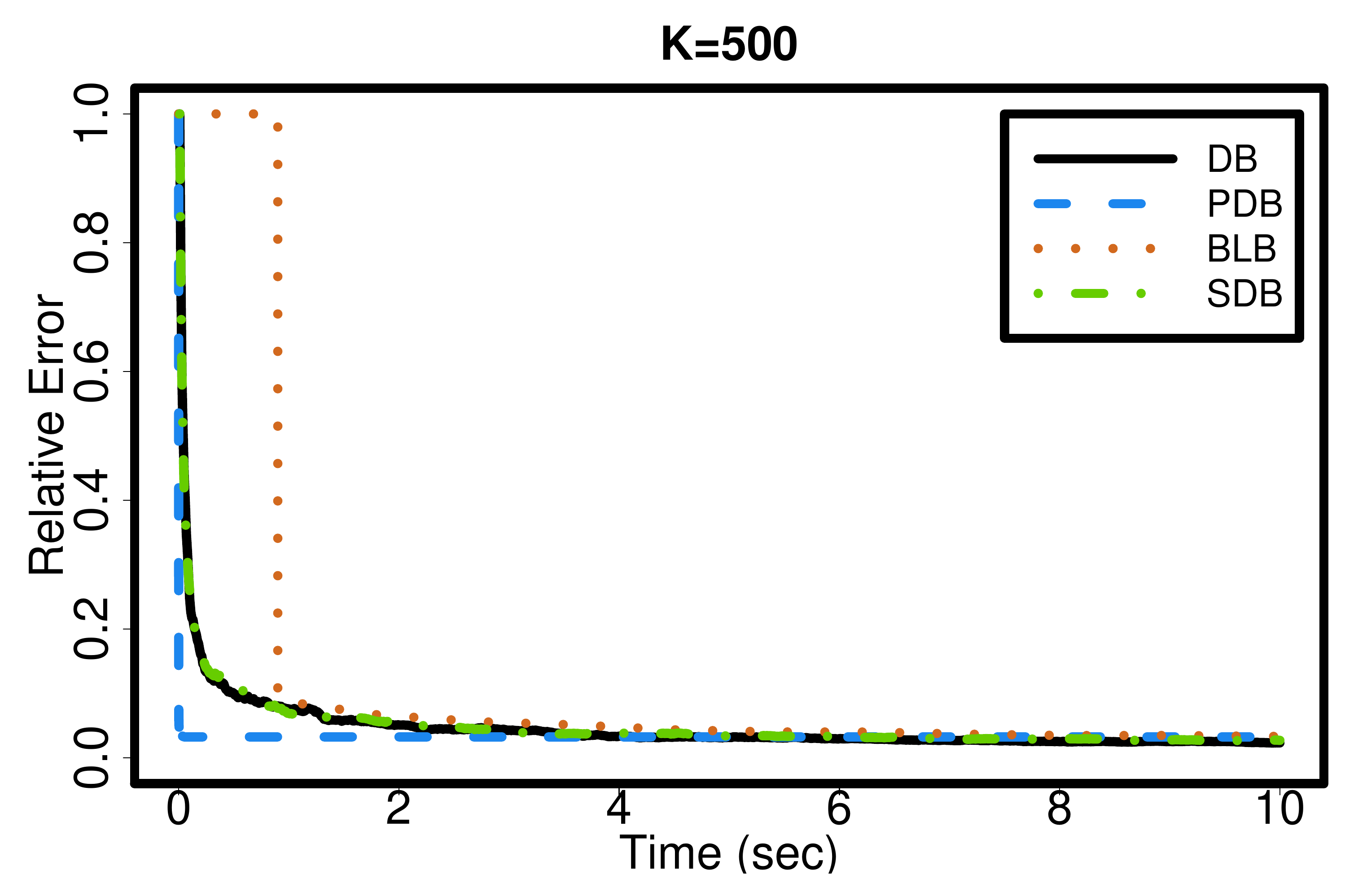}
		\label{fig:gau_K_500}}
	\quad
	\subfigure{%
		\includegraphics[width=0.45\linewidth]{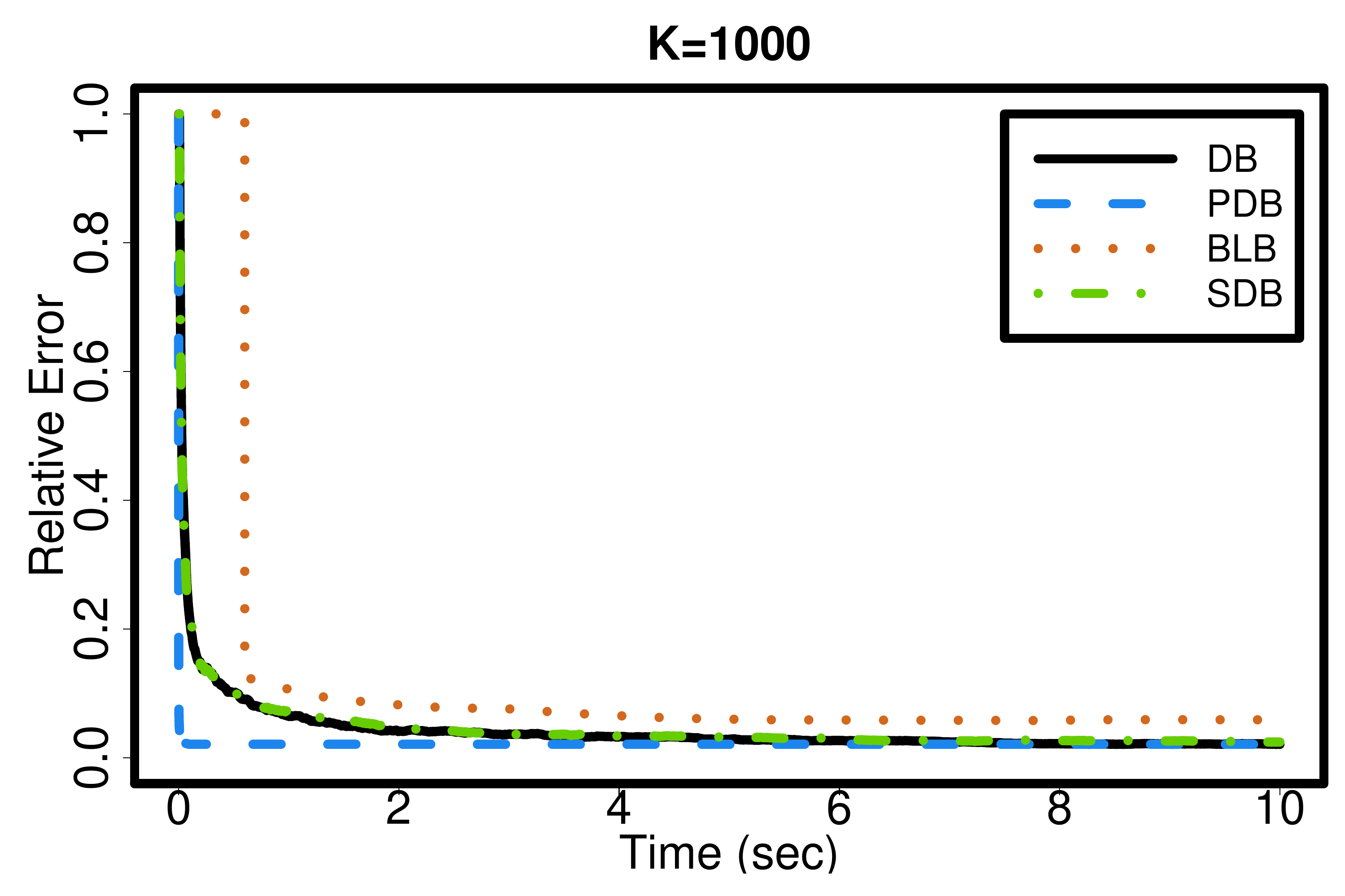}
		\label{fig:gau_K_1000}}
	\quad	
	\caption{Time evolution of the relative errors $|\hat{d}-d|/d$ in the width of the confidence interval under the Gaussian scenario with respect to different black size. DB: the distributed bootstrap (solid lines); PDB: the pseudo-distributed bootstrap (dashed lines); BLB: the bag of little bootstrap (dotted lines); SDB: the subsampled double bootstrap (dot-dashed lines).}
	\label{fig:gau_ciw}
\end{figure}

In conclusion, with a small time budget, the pseudo-distributed bootstrap (PDB) had better performance than the other three resampling strategies. When $K$ is relatively large, PDB has advantage in computing and can produce reasonable estimators. 
However, the convergence  of PDB is limited by the order of $K$. With a sufficient time budget, the distributed bootstrap (DB) and SDB can work for a wide range of $K$ and obtain better estimator than the pseudo-distributed bootstrap (PDB) in some situation.

\section{Real data analysis}

In this section, we analyze an airline on-time performance data to illustrate the proposed distributed inference for massive data. The data are available from the \textit{Bureau of Transportation Statistics} website (\url{https://www.bts.gov/}), which consist of the arrival and departure details for $5,617,658$ passenger flights within the USA in $2016$. We are interested in the arrival delay (ARR\_DELAY) variable which is the difference in minutes between the scheduled and the actual arrival times. 
After removing missing values, which corresponded to the canceled flights, there are $5,538,145$ flight entries in the dataset.

Each sample point 
 consists of $10$ features that are potentially related to the arrival delay. These features along with descriptions of their meanings are listed in Table \ref{tab:features} in the supplementary material. 
It is expected that departure delay (DEP\_DELAY) tends to result in late arrival. Hence, the variable DEP\_DELAY should be highly correlated with ARR\_DELAY. 
The variable DIF\_ELAPS is the difference between actual (ACT\_ELAPS) and scheduled elapsed time (CRS\_ELAPS) in the computer reservation system (CRS). Thus, DEP\_DELAY and DIF\_ELAPS together determine the ARR\_DELAY. 
One question is, how much of the arrival delay is due to the departure delay or the difference between actual and scheduled elapsed time? In addition, 
TAXI\_OUT records the time between departure from the gate and wheels off, while TAXI\_IN is the time differences between wheels down and arrival at the gate. DISTANCE provides the flight distance, while MONTH, DAY\_OF\_WEEK and DAY\_OF\_MONTH provide month and date information of each flight.

Our goal is to find out  which variables were more dependent with ARR\_DELAY 
and to characterize the dependence  with respect to different airports.
We only considered the airports with at least $200,000$ flight arrivals in $2016$. This led to four airports: ATL (Atlanta International Airport), ORD (Chicago O'Hare), DEN (Denver) and LAX (Los Angeles), which had $381,166$,  $238,994$, $222,121$, $210,706$ flight arrivals in 2016, respectively. 

We use the distance covariance introduced in \citet{SRB2007} to quantify the dependence. Suppose $\bY$ and $\bZ$ are two random vectors having finite first moments and taking values in $\mathbf{R}^p$ and $\mathbf{R}^q$, respectively. The population distance covariance between $\bY$ and $\bZ$ is
$$dcov^2(\bY,\bZ)=\int_{\mathbf{R}^{p+q}}\| \phi_{\bY,\bZ}(\bt,\bs)-\phi_{\bY}(\bt)\phi_{\bZ}(\bs) \|^2\omega(\bt,\bs)d\bt d\bs,$$
where $\phi_{\bY}(\bt)$, $\phi_{\bZ}(\bs)$ and $\phi_{\bY, \bZ}(\bt,\bs)$ are the characteristic functions of $\bY$, $\bZ$ and $\bX=(\bY^T, \bZ^T)^T$, respectively, and $\omega(\bt,\bs)=(c_pc_q\|\bt\|_p^{1+p}\|\bs\|_q^{1+q})^{-1}$ is a weight function with $c_d=\pi^{(1+d)/2}/\Gamma\{(1+d)/2\}$,  $\Gamma$ is the Gamma function. It is clear that $dcov^2(\bY,\bZ)$ equals zero if and only if $\bY$ and $\bZ$ are independent. Suppose $\mathfrak{X}_N=\{\bX_1, \ldots, \bX_N\}$ is a sample from $F$, where $\bX_i=(\bY_i^T, \bZ_i^T)^T$ for $i=1,\ldots,N$. Define $A_{ij}=\|\bY_i-\bY_j\|_p$ and $B_{ij}=\|\bZ_i-\bZ_j\|_q$ for $i,~j=1,\ldots,N$. Then, the 
empirical distance covariance \citep{SR2014} is
\begin{align*}
dcov_N^2(\bY,\bZ)= & \left\{N(N-1)\right\}^{-1}\sum_{i\neq j}A_{ij}B_{ij}+\left\{N(N-1)(N-2)\right\}^{-1}\sum_{i\neq j\neq l_1}A_{ij}B_{il_1} \\
& + \left\{N(N-1)(N-2)(N-3)\right\}^{-1}\sum_{i\neq j\neq l_1\neq l_2}A_{ij}B_{l_1l_2},
\end{align*}
which is a U-statistic of degree $4$ and an unbiased estimator of $dcov^2(\bY,\bZ)$.

Because of the sizes of these datasets, the distance covariance $dcov_N^2(\bY,\bZ)$ encounters heavy computational burden. We utilized the distributed version of the distance covariance, see Section \ref{sec:dist_covar} in the supplementary material for details. We divided the data of each airport randomly into $K$ subsets of equal size where $K$ was chosen in a set $\mathcal{K}=\{50, 100, 200, 500\}$. As discussed in the supplementary material, for two random vectors $\bY$ and $\bZ$, $\hat{\sigma}_{N,K}^{-1}dcov_{N,K}^{2}(\bY,\bZ)$ is asymptotic normal with unit variance regardless 
$\bY$ and $\bZ$ are independent or not. Here, $dcov_{N,K}^{2}(\bY,\bZ)$ is the distributed distance covariance based on $K$ subsets, and $\hat{\sigma}_{N,K}^{2}$ is a consistent estimator of $\Var\{dcov_{N,K}^{2}(\bY,\bZ)\}$. 
So, $\DM(\bY,\bZ)=: \hat{\sigma}_{N,K}^{-1}dcov_{N,K}^{2}(\bY,\bZ)$ can be used to measure the dependence between $\bY$ and $\bZ$. 
The larger $\DM(\bY,\bZ)$ is, the more dependence between $\bY$ and $\bZ$ is. Also we can compare $\DM(\bY,\bZ)$ with the critical values from the standard normal distribution to test for the significance of the dependence between $\bY$ and $\bZ$.

Figure \ref{fig:real_data_airport} reports the standardized dependence measures between the arrival delay and each of the $10$ feature variables for the four airports with the block size $K\in\mathcal{K}$.
It shows that all test statistics were above the $0.1\%$ critical value lines, which indicated significant dependence between the arrival delay and the ten features for the four airports.  Here, we deliberately use a rather small critical value to account for any multiplicity in the testing. 
Figure \ref{fig:real_data_airport} clearly shows  that the amount of dependence as quantified by $\DM(\mathrm{ARR\_DELAY}, \bZ)$ for $\bZ$ being the ten features were very consistent across different $K$ values. The largest variation with respect to $K$ was observed for the CRS\_ELP and Distance for Atlanta. 

This implied that the distributed inference would produce stable results with respect to different $K$. This is attractive as using a larger $K$ benefits the computing.

\begin{figure}[h]
	\centering
	\subfigure{%
		\includegraphics[width=0.47\linewidth]{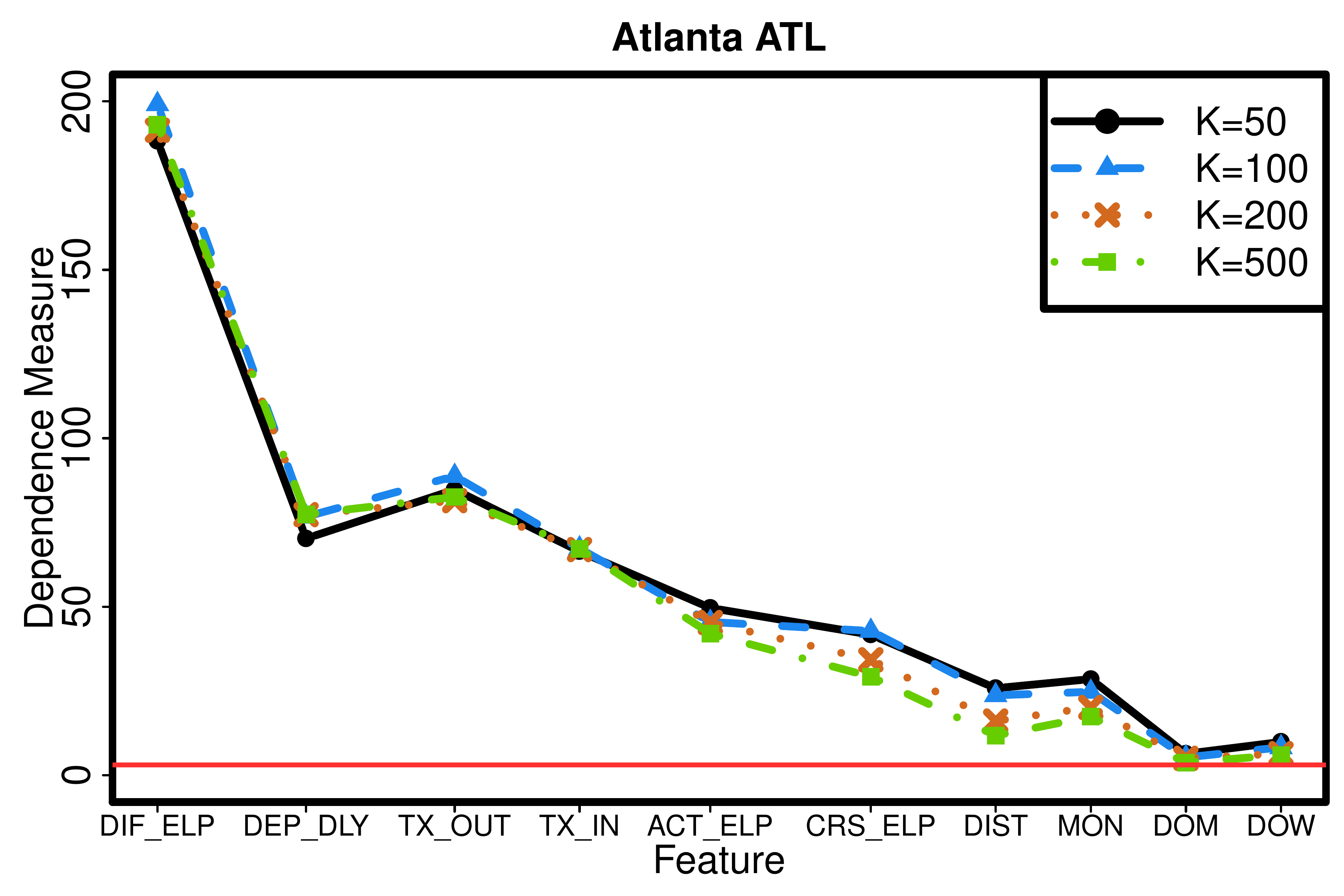}
		\label{fig:ATL}}
	\quad
	\subfigure{%
		\includegraphics[width=0.47\linewidth]{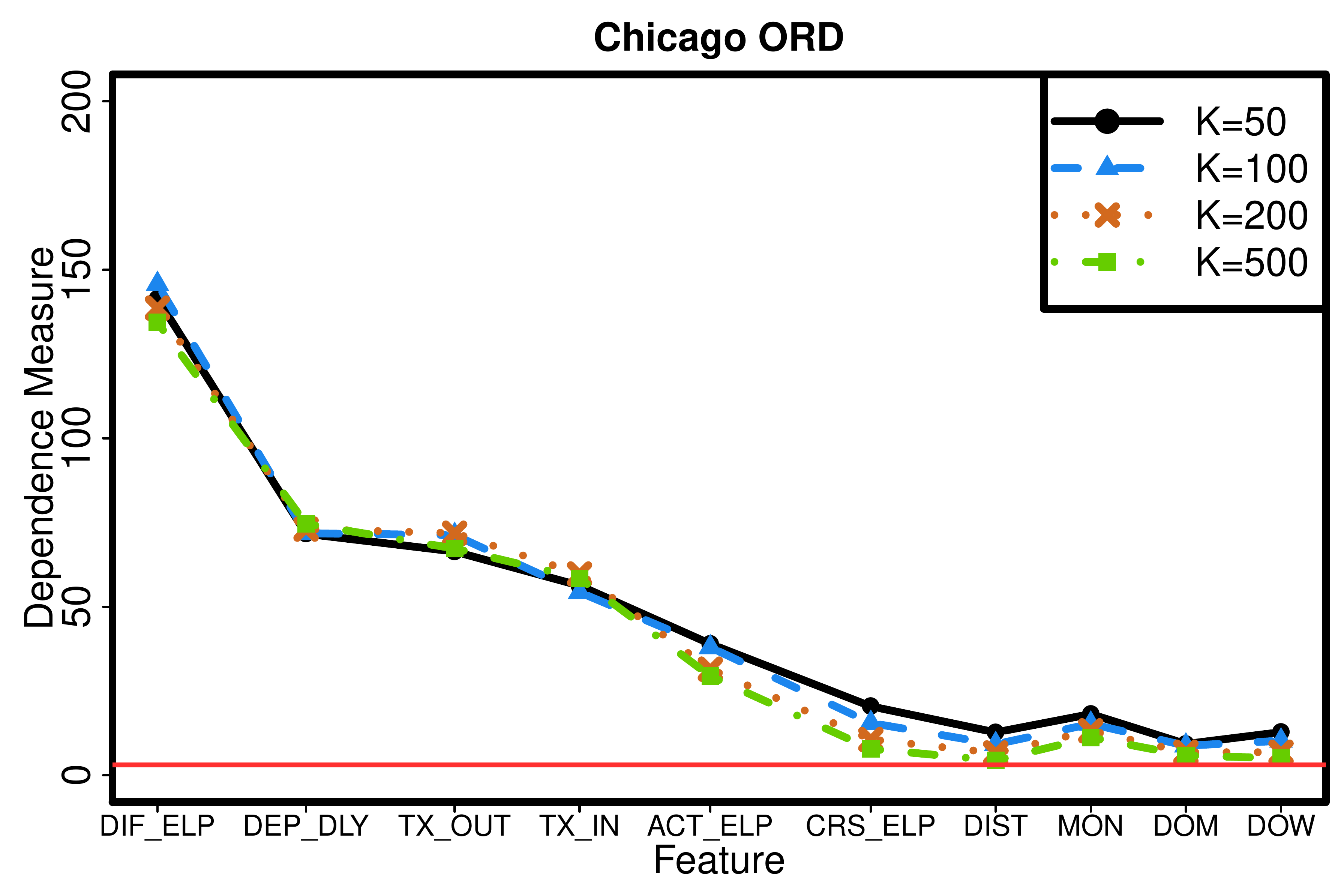}
		\label{fig:ORD}}
	\quad
	\subfigure{%
		\includegraphics[width=0.47\linewidth]{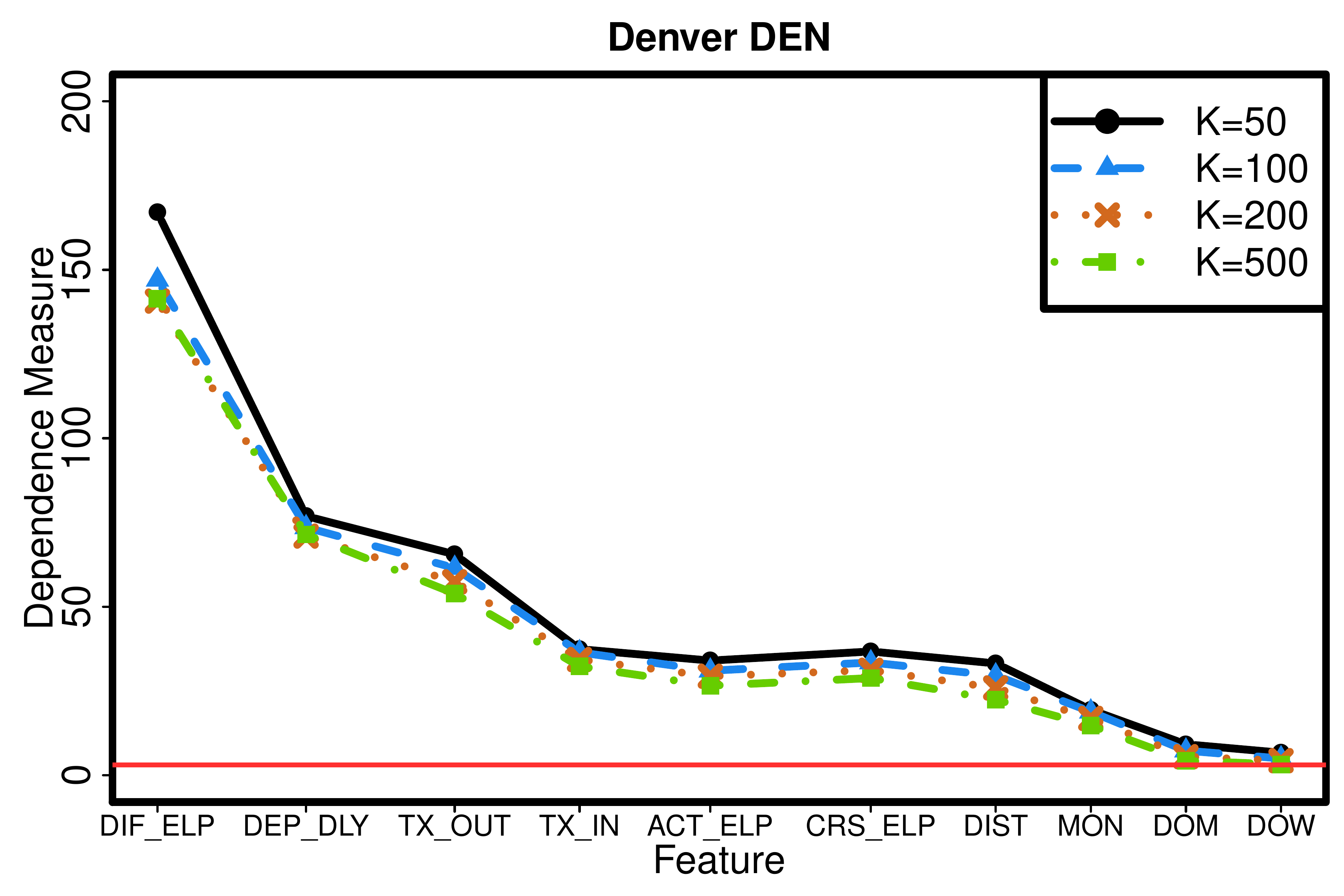}
		\label{fig:DEN}}
	\quad
	\subfigure{%
		\includegraphics[width=0.47\linewidth]{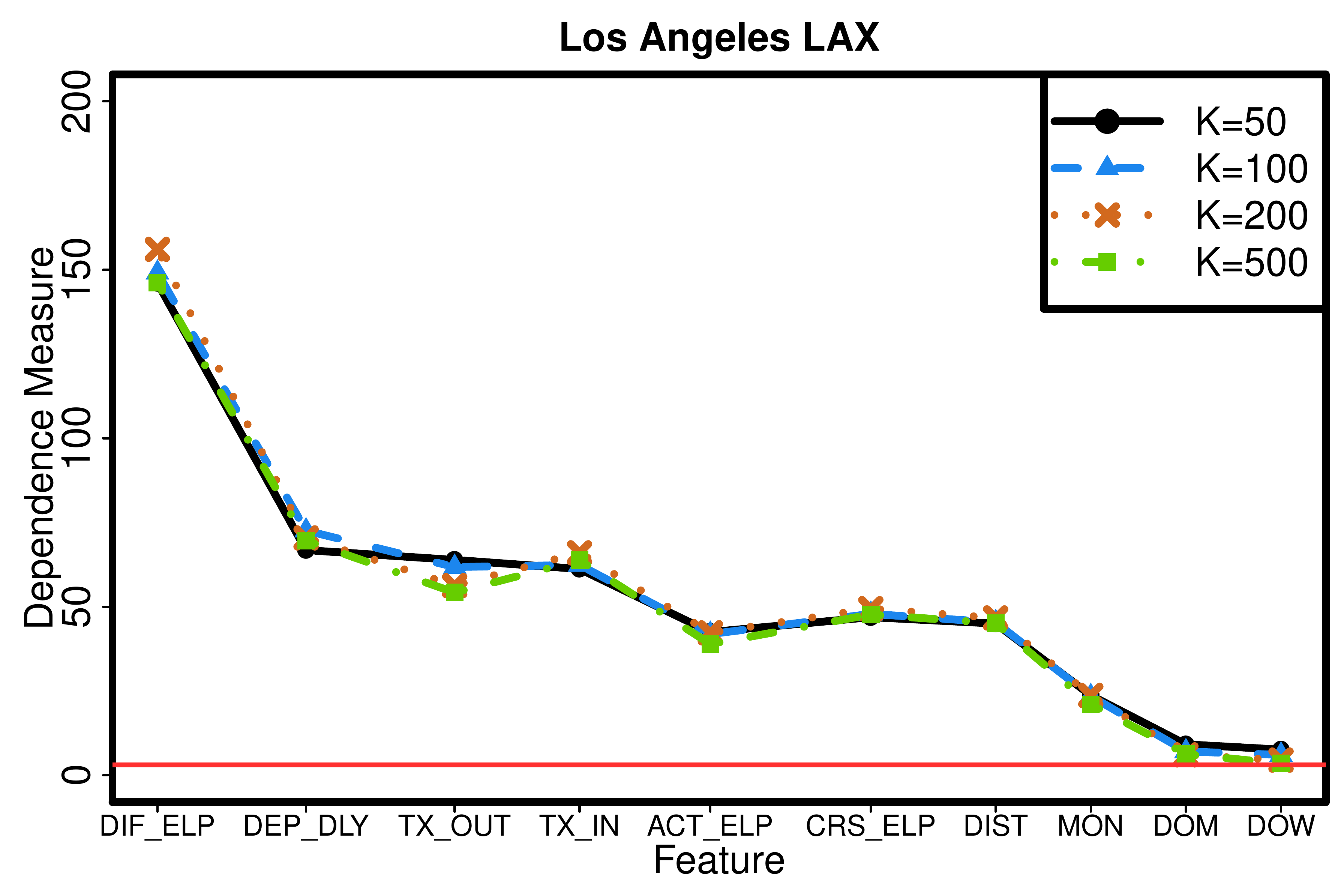}
		\label{fig:LAX}}
	\quad	
	\caption{Standardized dependence measures $\DM(\mathrm{ARR\_DELAY}, \bZ)$ for $\bZ$ being the $10$ feature variables respectively for the four airports with four block size $K$. The horizontal solid lines (in red) mark the $0.1\%$ critical value of the standard normal distribution.}
	\label{fig:real_data_airport}
\end{figure}

As expected, DIF\_ELAPS and DEP\_DELAY were the most dependent with the arrival delay for all airports. The only exception was Atlanta, where TAXI\_OUT was more dependent with the arrival delay than DEP\_DELAY. Besides DIF\_ELAPS and DEP\_DELAY, TAXI\_OUT and TAXI\_IN both had strong dependence on the arrival delay than the other six features for the four airports. 
The DAY\_OF\_MONTH and DAY\_OF\_WEEK were the least dependent with the arrival delay. 
For the feature DISTANCE, the result varied with different airports. For DEN and LAX, the dependence between DISTANCE and the arrival delay was stronger than those at ATL and ORD. It is observed that different data block size $K$ had less bearings on the 
distance measure for almost all the features at all four airports, except for the Distance and CRS\_ELP features at ATL, ORD and DEN. 

We carried out further analysis to give a clearer vision of the impact of DIF\_ELAPS and DEP\_DELAY on the arrival delay. Take Atlanta ATL as an illustration, we split the data of ATL into $10$ equal sized subsets according to the $\{0\%, 10\%,\ldots, 90\%, 100\%\}$ quantiles of its arrival delay. 
Figure \ref{fig:real_data_quan} plots  the standardized distance dependence between ARR\_DELAY and DIF\_ELAPS and DEP\_DELAY, respectively,  for different quantile ranges of ARR\_DELAY with the subset size $K\in\{50, 100, 200\}$. 
 For each subset with ARR\_DELAY within each quantile range, the standardized dependence measures were plotted with respect to $K$. 
Figure \ref{fig:real_data_quan} shows that the results were consistent for different $K$ in all cases, which was assuring for the distributed inference.

\begin{figure}[h]
	\centering
	\includegraphics[width=0.6\linewidth]{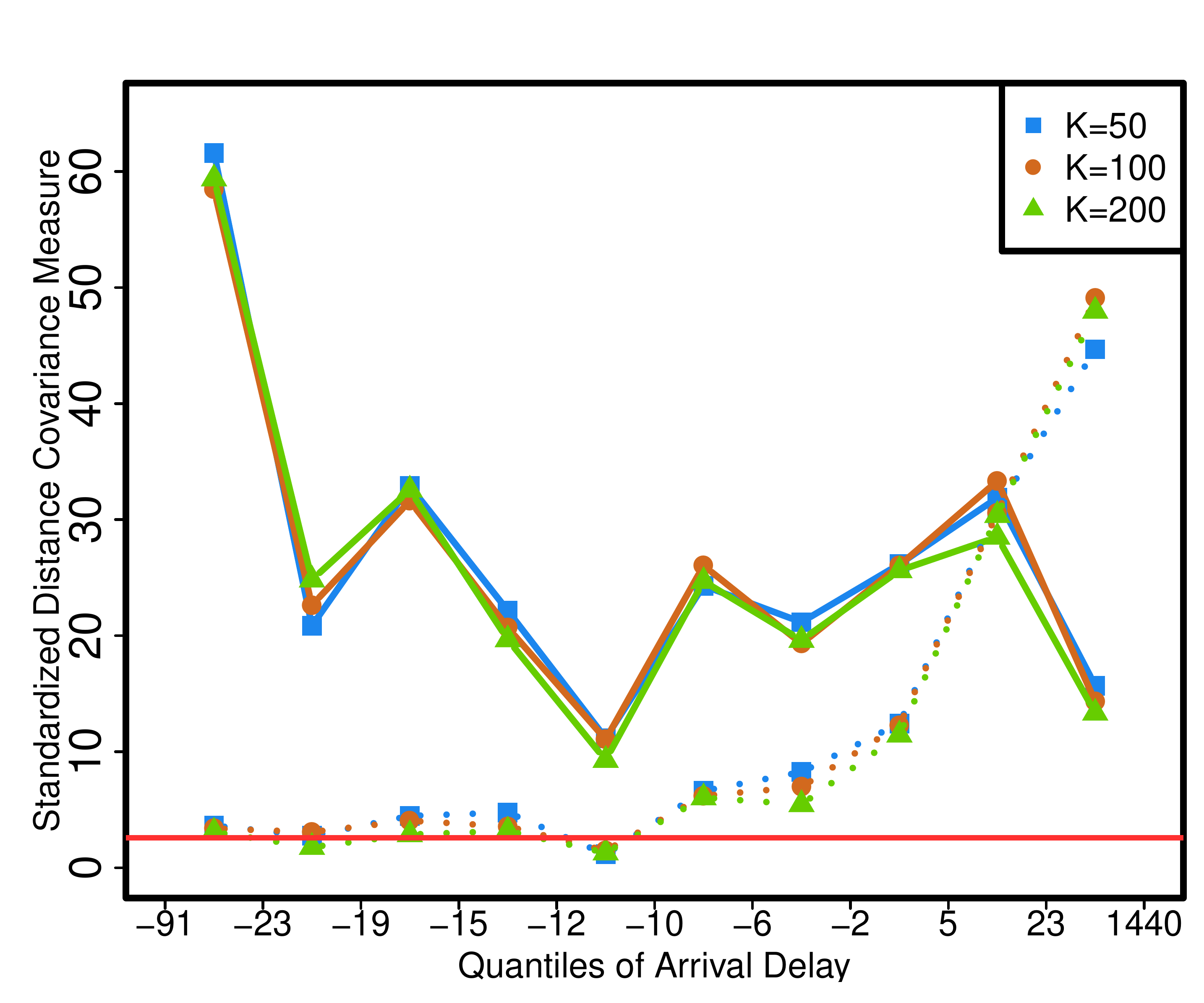}
	\caption{Standardized distance covariance measures between the arrival delay with DIF\_ELAPS (solid lines) and DEP\_DELAY (dotted lines) at Atlanta for different quantile ranges of ARR\_DELAY and blocking size $K$. 
	 The y-axis is in the $N(0,1)$ scale with 
		the horizontal red line marking the $0.995$ quantile.}
		\label{fig:real_data_quan}
\end{figure} 


\section{Conclusion}

The paper has investigated distributed inferences for massive data with a focus on a general  symmetric statistics $T_N$ given in (\ref{eq:T_N}), which is designed to make the computation scalable. 
We have provided detailed analyzes  on the statistical properties of the distributed statistics as well as their asymptotic distributions when  the number of subsets $K$ is either  fixed or diverging and the statistics is non-degenerate or otherwise. 
Two bootstrap methods are proposed and studied theoretically, which are showed to have aspects of computational advantages over the BLB and SDB in the context of massive data.

An important practical issue for the distributed inference is the choice of the number of blocks $K$.  As showed in our analysis,  an increasing $K$ would decrease the computational cost, but lead to a loss in statistical efficiency.  
This topic has been touched in Section \ref{sec:select_K}, where storage and memory requirement are considered. However, it is still an issue on how to select $K$ in practice that balances the computing time and statistical efficiency. Instead of considering a fixed time budget, another way of consideration is to minimize computing time subject to attaining certain statistical efficiency. This is similar as the sample size determination problem. We leave its analysis in future study.  
It  is also of interest to study higher order correctness and convergence rates of our proposed distributed approaches. 

\clearpage\pagebreak\newpage

\newpage
\setcounter{page}{1}

%
%
%
%
\renewcommand\thesection{\Alph{section}}
\renewcommand\thesubsection{\thesection.\arabic{subsection}}

\newcounter{sectionOld}
\setcounter{section}{0}
\setcounter{lemma}{0}

\setcounter{theorem}{0}
\makeatletter
\renewcommand{\thetheorem}{\thesection.\arabic{subsection}}
\makeatother

\setcounter{corollary}{0}
\makeatletter
\renewcommand{\thecorollary}{\thesection.\arabic{subsection}}
\makeatother

\setcounter{table}{0}
\makeatletter
\renewcommand{\thetable}{S\@arabic\c@table}
\makeatother

\setcounter{figure}{0} \makeatletter
\renewcommand{\thefigure}{S\@arabic\c@figure}
\makeatother

\section*{\centering Supplementary material for ``Distributed Statistical Inference for Massive Data"}

\begin{center}
	Song Xi Chen and Liuhua Peng
\end{center}

\vspace{0.5in}

The supplementary materials including technical details, proofs of main theorems and additional simulation studies, are organized as follows. Section \ref{sec:sup_app} presents important lemmas and proofs of theorems in the paper. Section  \ref{sec:sup_tech} introduces distributed distance covariance along with other technical discussions. More simulation results are reported in Section \ref{sec:sup_simu}.

\section{Proofs}\label{sec:sup_app}

In this section, we present the proofs of main theorems in this paper. Before that, we give several lemmas that will be frequently used. The first lemma is a generalization of Esseen's inequality which we may refer to Theorem 5.7 in \citet{Petrov1998}.

\begin{lemma}\label{lem:l01}	
	Let $Y_1, \ldots, Y_N$ be independent random variables. For $i=1,\ldots,N$, $\E(Y_i)=0$ and $\E|Y_i|^{2+\delta}<\infty$ for some positive $\delta\leq 1$. Then
	\begin{align*}
	\sup\limits_{x \in \mathbf{R}}\bigg|\P\bigg(B_N^{-1/2}\sum_{i=1}^{N}Y_i \leq x \bigg) - \Phi(x)\bigg| \leq C_1B_N^{-1-\delta/2}\sum_{i=1}^{N}\E|Y_i|^{2+\delta},
	\end{align*}
	where $B_N=\sum_{i=1}^{N}\Var(Y_i)$ and $C_1$ is a positive constant.	
\end{lemma}

The second lemma is a modified version of Marcinkiewicz-Zygmund strong law of larger numbers (SLLN) and its proof can be found in \citet{Liu1988}.

\begin{lemma}\label{lem:l02}	
	Let $Y_1, \ldots, Y_N$ be independent random variables with $\E|Y_i|^{\eta+\varepsilon}<\infty$ for some positive $\eta$, $\eta<2$, and $\varepsilon>0$. Then as $N\to\infty$,
	\begin{align*}
	N^{-1/\eta}\sum_{i=1}^{N}(Y_i-a_i) \to 0
	\end{align*}
	almost surely, where $a_i=\E Y_i$ if $\eta\geq 1$ and $a_i=0$ otherwise.	
\end{lemma}

\subsection{Proof of Theorem \ref{theo:uniform_convergence}}

\begin{proof}	
	According to Lemma \ref{lem:l01}, we know that
	\begin{align*}
	\sup\limits_{x \in \mathbf{R}}\left|\P\left(V_N \leq x \right) - \Phi(x)\right| \leq N^{-\delta/2}C_1\sigma_{\alpha}^{-2-\delta}\E|\alpha(X_1;F)|^{2+\delta},
	\end{align*}
	where $V_N = N^{-1/2}\sigma_{\alpha}^{-1}\sum_{i=1}^{N}\alpha(X_i;F)$. Denote $\Delta_N = N^{-3/2}\sigma_{\alpha}^{-1}\sum_{1\leq i < j \leq N}\beta(X_i,X_j;F)$ and $\Delta_{N,K} = N^{-1/2}\sigma_{\alpha}^{-1}\sum_{k=1}^{K}n_k^{-1}\sum_{1\leq i < j \leq n_k}\beta(X_{k,i}, X_{k,j}; F)$, then by Cauchy-Schwarz inequality, 
	\begin{align*}
	& \P\left(|\Delta_N| \geq (lnN)^{-1}\right)=O\big(N^{-1}(lnN)^2\big) \text{~~~and~~~} \\
	& \P\left(|\Delta_{N,K}| \geq (lnN)^{-1}\right)=O\big(KN^{-1}(lnN)^2\big).
	\end{align*}
	Under Condition \ref{condition:R_N} and the assumption that $K=O(N^{\tau'})$, $\E(R_{N,K})=O(N^{\tau_1\tau'-\tau_1})$ and $\Var(R_{N,K})=o(N^{\tau_2\tau'-\tau'-\tau_2})$. Again by Cauchy-Schwarz inequality, 
	\begin{align*}
	& \P\big(|R_N| \geq N^{-1/2}(lnN)^{-1}\big)=O\big(N^{1-2\tau_1}(lnN)^2\big)+o\big(N^{1-\tau_2}(lnN)^2\big)=o(1), \\
	& \P\big(|R_{N,K}| \geq N^{-1/2}(lnN)^{-1}\big)=O\big(N^{2\tau_1\tau'-2\tau_1+1}(lnN)^2\big)+o\big(N^{\tau_2\tau'-\tau'-\tau_2+1}(lnN)^2\big)=o(1),
	\end{align*}
	under the condition that $\tau'<1-1/(2\tau_1)$. Finally, by the fact that
	\begin{align*}
	\sup\limits_{x \in \mathbf{R},~|t|<(lnN)^{-1}}\left|\Phi(x+t) - \Phi(x)\right| = o(1).
	\end{align*}
	we complete the proof of Theorem \ref{theo:uniform_convergence}.
	
\end{proof}

\subsection{Proof of Theorem \ref{theo:T_N_K_degenerate}}

\begin{proof}
	
	(i) When $K$ is finite, the result can be easily obtained by the independence between each subset and the classic limit theorem for degenerate U-statistics \citep{Serfling1980}. 
	
	(ii) When $K\rightarrow\infty$, without loss of generality, assume $\theta=0$. Denote $T_{N,K}^{(k)}=\Delta_{N,K}^{(k)}+R_{N,K}^{(k)}$ where $\Delta_{N,K}^{(k)}=n_k^{-2}\sum_{1\leq i<j\leq n_k}\beta\left(X_{k,i}, X_{k,j}; F\right)$. Under the condition that $R_{N,K}=o_p(K^{1/2}N^{-1})$, it is sufficient to prove that $2^{1/2}K^{-1/2}N\sigma^{-1}_{\beta}\Delta_{N,K}$ converges to 
	$\mathcal{N}(0,1)$ as $K,N\to\infty$, here
	$ \Delta_{N,K}=N^{-1}\sum_{k=1}^{K}n_k\Delta_{N,K}^{(k)}=N^{-1}\sum_{k=1}^{K}n_k^{-1}\sum_{1\leq i<j\leq n_k}\beta\left(X_{k,i}, X_{k,j}; F\right)$.
	Rewrite
	\begin{align*}
	2^{1/2}K^{-1/2}N\sigma^{-1}_{\beta}\Delta_{N,K} = \sum_{k=1}^{K}2^{1/2}K^{-1/2}n_k\sigma^{-1}_{\beta}\Delta_{N,K}^{(k)} \equiv \sum_{k=1}^{K}\mathbf{\Delta}_{N,K}^{(k)},
	\end{align*}
	where $\mathbf{\Delta}_{N,K}^{(k)}=2^{1/2}K^{-1/2}n_k\sigma^{-1}_{\beta}\Delta_{N,K}^{(k)}$ satisfies $\E\big(\mathbf{\Delta}_{N,K}^{(k)}\big)=0$ and $\Var\big(\mathbf{\Delta}_{N,K}^{(k)}\big)=K^{-1}(1-n_k^{-1})$ for $k=1,\ldots,K$. As $\big\{\mathbf{\Delta}_{N,K}^{(k)}\big\}_{k=1}^K$ are independent, it is enough to check the Lindeberg condition for $\mathbf{\Delta}_{N,K}^{(k)}$. Note that $s_{N,K}^2\equiv\sum_{k=1}^{K}\Var\big(\mathbf{\Delta}_{N,K}^{(k)}\big)=1-K^{-1}\sum_{k=1}^{K}n_k^{-1}$, then for every $\varepsilon>0$,
	\begin{align*}
	& s_{N,K}^{-2}\sum_{k=1}^{K}\E\left[\left|\mathbf{\Delta}_{N,K}^{(k)}\right|^2\mathbf{1}\left\{\left|\mathbf{\Delta}_{N,K}^{(k)}\right|>\varepsilon s_{N,K}\right\}\right] \\
	= & s_{N,K}^{-2}\sum_{k=1}^{K}K^{-1}\E\left[\left|K^{1/2}\mathbf{\Delta}_{N,K}^{(k)}\right|^2\mathbf{1}\left\{\left|K^{1/2}\mathbf{\Delta}_{N,K}^{(k)}\right|>\varepsilon s_{N,K}K^{1/2}\right\}\right] \\
	\leq & s_{N,K}^{-2}\sum_{k=1}^{K}K^{-1}\left(\varepsilon s_{N,K}K^{1/2}\right)^{-\delta'}\E\left|K^{1/2}\mathbf{\Delta}_{N,K}^{(k)}\right|^{2+\delta'} \\
	\leq & C_2s_{N,K}^{-2}\sum_{k=1}^{K}K^{-1}\left(\varepsilon s_{N,K}K^{1/2}\right)^{-\delta'} \\
	\to & ~0,~~~~~as~~K\to\infty.
	\end{align*}
	The next-to-last line is from the moment inequalities of U-statistics \citep{Koroljuk1994} and $C_2$ is a positive constant. Thus we finish the proof of Theorem \ref{theo:T_N_K_degenerate}.	
\end{proof}

\subsection{Proof of Theorem \ref{theo:edg_T_SC_N_boot1}}

\begin{proof}
	
	According to the proof of Theorem \ref{theo:uniform_convergence}, under the conditions in Theorem \ref{theo:uniform_convergence}, we have
	\begin{align*}
	\sup\limits_{x \in \mathbf{R}}\left|\P\left\{N^{1/2}\sigma_{\alpha}^{-1}(T_{N, K}-\theta) \leq x\right\}  - \Phi(x)\right| = o(1),
	\end{align*}
	thus, it is sufficient to show that under the conditions assumed in the theorem
	\begin{align}\label{eq:proof_10}
	\sup\limits_{x \in \mathbf{R}}\left|\P\left\{N^{1/2}\hat{\sigma}^{-1}_{\alpha,N,K}(T^*_{N, K}-\hat{\theta}_{N,K}) \leq x\big| F_{N,K}^{(1)},\ldots,F_{N,K}^{(K)}\right\} - \Phi(x)\right| = o_p(1).
	\end{align}	
	We proceed to show that $\E\{|\hat{\alpha}(X_{k,1}^*; F_{N,K}^{(k)})|^{2+\delta}|F_{N,K}^{(k)}\}$ is bounded in probability for $k=1,\ldots,K$. Note that
	\begin{align*}
	& \E\{|\hat{\alpha}(X_{k,1}^*; F_{N,K}^{(k)})|^{2+\delta}|F_{N,K}^{(k)}\} \\
	= & n_k^{-1}\sum_{i=1}^{n_k}|\hat{\alpha}(X_{k,i}; F_{N,K}^{(k)})|^{2+\delta} \\
	\leq & n_k^{-1}\sum_{i=1}^{n_k}C_{\delta}\big\{|\hat{\alpha}(X_{k,i}; F_{N,K}^{(k)})-\alpha(X_{k,i};F)|^{2+\delta}+|\alpha(X_{k,i};F)|^{2+\delta}\big\},
	\end{align*}
	for some positive constant $C_\delta$. By the condition that $\sup\limits_{x \in S(F)}|\hat{\alpha}(x; F_{N,K}^{(k)})-\alpha(x; F)|=o_p(1)$ and the strong law of large numbers, we have $\E\{|\hat{\alpha}(X_{k,1}^*; F_{N,K}^{(k)})|^{2+\delta}|F_{N,K}^{(k)}\}=O_p(1)$ for $k=1,\ldots,K$. Similarly, we can show that $\E\{|\hat{\beta}(X_{k,1}^*, X_{k,2}^*; F_{N,K}^{(k)})|^2|F_{N,K}^{(k)}\}<\infty$ in probability.
	
	Thus, under the condition $\P\big\{|R^*_{N, K}| \geq N^{-1/2}(ln N)^{-1}\big|F_{N,K}^{(1)},\ldots,F_{N,K}^{(K)}\big\}=o_p(1)$, by carrying out similar procedures in the proof of Theorem \ref{theo:uniform_convergence}, we can prove (\ref{eq:proof_10}) which leads to the result that
	\begin{align*}
	\sup\limits_{x \in \mathbf{R}}\left|\P\left\{N^{1/2}\hat{\sigma}^{-1}_{\alpha,N,K}(T^*_{N, K}-\hat{\theta}_{N,K}) \leq x\big| F_{N,K}^{(1)},\ldots,F_{N,K}^{(K)}\right\} - \P\left\{N^{1/2}\sigma_{\alpha}^{-1}(T_{N, K}-\theta) \leq x\right\}\right| = o_p(1).
	\end{align*}
	
	In addition,
	\begin{align*}
	\hat{\sigma}^{2}_{\alpha,N,K} - \sigma_{\alpha}^2 = & N^{-1}\sum_{k=1}^{K}n_k\left[\E\left\{\hat{\alpha}^{2}(X_{k,1}^*; F_{N,K}^{(k)})\big|F_{N,K}^{(k)}\right\} - \sigma_{\alpha}^2\right] \\
	= & N^{-1}\sum_{k=1}^{K}n_k\left\{n_k^{-1}\sum_{i=1}^{n_k}\hat{\alpha}^{2}(X_{k,i}; F_{N,K}^{(k)}) - \sigma_{\alpha}^2\right\} \\
	= & N^{-1}\sum_{k=1}^{K}n_k\bigg[n_k^{-1}\sum_{i=1}^{n_k}\left\{\hat{\alpha}(X_{k,i}; F_{N,K}^{(k)})-\alpha(X_{k,i};F)\right\}^2 \\
	& + n_k^{-1}\sum_{i=1}^{n_k}2\alpha(X_{k,i};F)\left\{\hat{\alpha}(X_{k,i}; F_{N,K}^{(k)})-\alpha(X_{k,i};F)\right\} \\
	& +  n_k^{-1}\sum_{i=1}^{n_k}\left\{\alpha^2(X_{k,i}; F) - \sigma_{\alpha}^2\right\}\bigg] \\
	= & o_p(1).
	\end{align*}
	
\end{proof}

\subsection{Proof of Theorem \ref{corr:pdb1}}

\begin{proof}
	Let $\mathfrak{T}_{N,K}^{(k)}=\mathcal{T}_{N,K}^{(k)}-N^{1/2}K^{-1/2}\theta$. Denote $$\theta_k=\E\left(\mathfrak{T}_{N, K}^{(k)}\right)=N^{-1/2}K^{1/2}(n_k-NK^{-1})\theta+N^{-1/2}K^{1/2}n_k\E\left(R_{N,K}^{(k)}\right)$$ and $\sigma_{k}^2=\Var\left(\mathfrak{T}_{N, K}^{(k)}\right)=N^{-1}Kn_k\sigma_{\alpha}^2\left\{1+o(1)\right\}$ for $k=1,\ldots,K$, then
	\begin{align*}
	\bar{\theta}=K^{-1}\sum_{k=1}^{K}\theta_k=N^{-1/2}K^{-1/2}\sum_{k=1}^{K}n_k\E\left(R_{N,K}^{(k)}\right).
	\end{align*}
	
	Denote $\mathfrak{T}_{N,K}^{*(k)}=\mathcal{T}_{N,K}^{*(k)}-N^{1/2}K^{-1/2}\theta$, then $\mathfrak{T}_{N,K}^{*(k)}$, $k=1,\ldots,K$ are i.i.d.~from $F_{K, \mathcal{T}}(x+N^{1/2}K^{-1/2}\theta)$. Since $\E|\alpha(X_1;F)|^{2+\delta} < \infty$, $\E\left|\beta(X_1, X_2; F)\right|^{2+\delta}<\infty$ and $\sup_k\E\left|n_k^{1/2}R_{N,K}^{(k)}\right|^{2+\delta}<\infty$, we have $\sup_k\E\left(\left|\mathfrak{T}_{N,K}^{(k)}\right|^{2+\delta}\right)<\infty$. This immediately results in
	\begin{align*}
	\E\left\{\left|\mathfrak{T}_{N,K}^{*(k)}\right|^{2+\delta_1}\big| F_{K, \mathcal{T}}\right\} = K^{-1}\sum_{k=1}^{K}\left|\mathfrak{T}_{N,K}^{(k)}\right|^{2+\delta_1}<\infty
	\end{align*}
	with probability $1$ for some $\delta_1$ satisfies $0<\delta_1<\delta$ (Lemma \ref{lem:l02}). Define $$\mathfrak{T}_{N,K}^{*}=K^{-1}\sum_{k=1}^{K}\mathfrak{T}_{N,K}^{*(k)}=\mathcal{T}_{N,K}^{*}-N^{1/2}K^{-1/2}\theta,$$
	then $\mathfrak{T}_{N,K}^{*}-\mathfrak{T}_{N,K}=\mathcal{T}_{N,K}^{*}-N^{1/2}K^{-1/2}T_{N,K}$ where $\mathfrak{T}_{N,K}=K^{-1}\sum_{k=1}^{K}\mathfrak{T}_{N,K}^{(k)}$. Denote $\bar{\sigma}^2=K^{-1}\sum_{k=1}^{K}\sigma_k^2$ and $$\bar{s}^2=K^{-1}\sum_{k=1}^{K}\left(\mathfrak{T}_{N,K}^{(k)}-\mathfrak{T}_{N,K}\right)^2,$$
	then $\sup_kN^{-1/2}K^{1/2}|n_k-NK^{-1}|\to0$ implies that
	\begin{align*}
	\bar{s}^2-\bar{\sigma}^2 =  K^{-1}\sum_{k=1}^{K}\left\{\left(\mathfrak{T}_{N,K}^{(k)}-\theta_k+\theta_k-\bar{\theta}+\bar{\theta}-\mathfrak{T}_{N,K}\right)^2-\sigma_k^2\right\} \to 0
	\end{align*}
	almost surely.
	According to Lemma \ref{lem:l01},	
	\begin{align*}
	& \sup\limits_{x \in \mathbf{R}}\left|\P\left\{K^{1/2}\bar{s}^{-1}(\mathcal{T}_{N,K}^{*}-N^{1/2}K^{-1/2}T_{N,K}) \leq x\bigg|F_{K, \mathcal{T}}\right\} - \Phi(x)\right| \\
	\leq & C_1\bar{s}^{-2-\delta_1}K^{-\delta_1/2}\E\left\{\left|\mathfrak{T}_{N,K}^{*(k)}-\mathfrak{T}_{N,K}\right|^{2+\delta_1}\big| F_{K, \mathcal{T}}\right\} \\
	\leq & C_2\bar{s}^{-2-\delta_1}K^{-\delta_1/2}\left[\E\left\{\left|\mathfrak{T}_{N,K}^{*(k)}\right|^{2+\delta_1}\big| F_{K, \mathcal{T}}\right\}+\left|\mathfrak{T}_{N,K}\right|^{2+\delta_1}\right] 
	\to  0
	\end{align*}
	almost surely as $K\to\infty$.
	It has been shown in the proof of Theorem \ref{theo:uniform_convergence} that
	\begin{align*}
	\sup\limits_{x \in \mathbf{R}}\left|\P\left\{N^{1/2}\sigma_{\alpha}^{-1}(T_{N,K}-\theta) \leq x \right\} - \Phi(x)\right| = o(1).
	\end{align*}
	Combine this with the fact that $\bar{s}^2-\sigma_{\alpha}^2\to0$ with probability 1, we complete the proof of this corollary.

\end{proof}

\subsection{Proof of Theorem \ref{corr:pdb2}}

\begin{proof}
	Use similar techniques as in the proof of Theorem \ref{corr:pdb1}, let $\mathfrak{T}_{N,K}^{(k)}=\mathcal{T}_{N,K}^{(k)}-NK^{-1}\theta=n_k(T_{N,K}^{(k)}-\theta)$, denote $\theta_k=\E\big(\mathfrak{T}_{N, K}^{(k)}\big)=n_k\E\big(R_{N,K}^{(k)}\big)$ and $\sigma_{k}^2=\Var\big(\mathfrak{T}_{N, K}^{(k)}\big)=2^{-1}\sigma_{\beta}^2\left\{1+o(1)\right\}$ for $k=1,\ldots,K$, then $\mathfrak{T}_{N,K}=K^{-1}\sum_{k=1}^{N}\mathfrak{T}_{N,K}^{(k)}=NK^{-1}(T_{N,K}-\theta)$ and $\bar{\theta}=K^{-1}\sum_{k=1}^{K}\theta_k=K^{-1}\sum_{k=1}^{K}n_k\E\big(R_{N,K}^{(k)}\big)$.
	
	Under the conditions that $\E\left|\beta(X_1, X_2; F)\right|^{2+\delta'}<\infty$ and $\sup_k\E\big|n_kR_{N,K}^{(k)}\big|^{2+\delta'}<\infty$, we obtain $\sup_k\E\big|n_k(T_{N,K}^{(k)}-\theta)\big|^{2+\delta'}<\infty$, which in turn implies that
	\begin{align*}
	\sup\limits_{x \in \mathbf{R}}\left|\P\left\{NK^{-1/2}\bar{\sigma}^{-1}(T_{N,K}-\theta) \leq x \right\} - \Phi(x)\right| = o(1),
	\end{align*}
	where $\bar{\sigma}^2=K^{-1}\sum_{k=1}^{K}\sigma_k^2$.
	
	Denote $\mathfrak{T}_{N,K}^{*(k)}=\mathcal{T}_{N,K}^{*(k)}-NK^{-1}\theta$, then $\mathfrak{T}_{N,K}^{*(k)}$, $k=1,\ldots,K$ are i.i.d. from $F_{K, \mathcal{T}}(x+NK^{-1}\theta)$. As $\sup_k\E\big(\big|\mathfrak{T}_{N,K}^{(k)}\big|^{2+\delta'}\big)<\infty$, this immediately results in
	\begin{align*}
	\E\left\{\left|\mathfrak{T}_{N,K}^{*(k)}\right|^{2+\delta_2}| F_{K, \mathcal{T}}\right\} = K^{-1}\sum_{k=1}^{K}\left|\mathfrak{T}_{N,K}^{(k)}\right|^{2+\delta_2}<\infty
	\end{align*}
	with probability $1$ and $\delta_2$ satisfies $0<\delta_2<\delta'$ (Lemma \ref{lem:l02}). Define $$\mathfrak{T}_{N,K}^{*}=K^{-1}\sum_{k=1}^{K}\mathfrak{T}_{N,K}^{*(k)}=\mathcal{T}_{N,K}^{*}-NK^{-1}\theta,$$
	then $\mathfrak{T}_{N,K}^{*}-\mathfrak{T}_{N,K}=\mathcal{T}_{N,K}^{*}-NK^{-1}T_{N,K}$. Denote $\bar{s}^2=K^{-1}\sum_{k=1}^{K}\left(\mathfrak{T}_{N,K}^{(k)}-\mathfrak{T}_{N,K}\right)^2$, then
	\begin{align*}
	\bar{s}^2-\bar{\sigma}^2 =  K^{-1}\sum_{k=1}^{K}\left\{\left(\mathfrak{T}_{N,K}^{(k)}-\theta_k+\theta_k-\bar{\theta}+\bar{\theta}-\mathfrak{T}_{N,K}\right)^2-\sigma_k^2\right\},
	\end{align*}
	which convergence to $0$ almost surely.
	According to Lemma \ref{lem:l01},	
	\begin{align*}
	& \sup\limits_{x \in \mathbf{R}}\left|\P\left\{K^{1/2}\bar{s}^{-1}(\mathcal{T}_{N,K}^{*}-NK^{-1}T_{N,K}) \leq x\bigg|F_{K, \mathcal{T}}\right\} - \Phi(x)\right| \\
	\leq & C_1\bar{s}^{-2-\delta_2}K^{-\delta_2/2}\E\left\{\left|\mathfrak{T}_{N,K}^{*(k)}-\mathfrak{T}_{N,K}\right|^{2+\delta_2}| F_{K, \mathcal{T}}\right\} 
	\to  0
	\end{align*}
	almost surely as $K\to\infty$, thus we finish the proof of Theorem \ref{corr:pdb2}.
	
\end{proof}

\section{Technical details}\label{sec:sup_tech}

\subsection{U-statistics as an example}\label{sec:U_stat_exp}

We show that the conditions in Theorem \ref{theo:edg_T_SC_N_boot1} are satisfied by U-statistics. Suppose $U_N=U(\mathfrak{X}_N)$ is a U-statistic of degree $2$ with a symmetric kernel function $h$, that is,
\begin{align}\label{eq:U_exp}
\notag	U_N = & 2\{N(N-1)\}^{-1}\sum_{1\leq i<j\leq N}h(X_{i}, X_{j}) \\
= & \theta_U + N^{-1}\sum_{i=1}^{N}\alpha_U(X_i; F) + N^{-2}\sum_{1\leq i < j \leq N}\beta_U(X_i, X_j; F) + R_{UN},
\end{align}
where $\theta_U=\E(U_N)$, $\alpha_U(x; F)=2\left[\E\left\{h(x, X_2)\right\}-\theta_U\right]$, $\beta_U(x, y; F)=2h(x, y)-\alpha_U(x; F)-\alpha_U(y; F)-2\theta_U$, and $R_{UN}$ is the remainder term. Obviously, $U_N$ is in the form of \eqref{eq:T_N}. Under the condition that $\E\{h(X_1, X_2)\}^2<\infty$, it can be shown that $\alpha_U$ and $\beta_U$ meet Condition \ref{condition:alpha} with $\sigma_{\alpha,U}^2=\Var\{\alpha_U(X_1;F)\}$ and $R_{UN}$ satisfies $\E(R_{UN})=0$ and $\E(R_{UN}^2)=O(N^{-4})$ \citep{Hoeffding1948,Serfling1980}. Let $U_{N,K}^{(k)}=U\big(\mathfrak{X}_{N,K}^{(k)}\big)$ be the corresponding U-statistic obtained from the $k$-th data block, $k=1,\ldots,K$, then the remainder terms of $\{U_{N,K}^{(k)}\}_{k=1}^K$ satisfy Condition \ref{condition:R_N} with $\tau_1=\infty$. Thus the distributed U-statistic is
\begin{align*}
U_{N,K}= & N^{-1}\sum_{k=1}^{K}n_k^{-1}U_{N,K}^{(k)} \\
= & \theta_U + N^{-1}\sum_{i=1}^{N}\alpha_U(X_i; F) + N^{-1}\sum_{k=1}^{K}n_k^{-1}\sum_{1\leq i < j \leq n_k}\beta_U(X_{k,i}, X_{k,j}; F) + R_{UN,K},
\end{align*}
where $R_{UN,K}$ satisfies $\E(R_{UN,K})=0$ and $\E\left(R_{UN,K}^2\right)=O(K^3N^{-4})$ under Condition \ref{condition:K}. 
If $\E|\alpha_U(X_1; F)|^{2+\delta}<\infty$ and $K=O(N^{\tau'})$ for some positive $\delta$ and $\tau'<1$, then as $N\to\infty$,
\begin{align*}
\sup\limits_{x \in \mathbf{R}}\big|\P\big\{N^{1/2}\sigma_{\alpha, U}^{-1}(U_{N, K}-\theta_U) \leq x\big\} - \P\big\{N^{1/2}\sigma_{\alpha, U}^{-1}(U_N-\theta_U) \leq x\big\}\big| = o(1).
\end{align*}

Now we consider the distributed bootstrap version of the distributed U-statistic. For $k=1,\ldots,K$, let $\mathfrak{X}_{N,K}^{*(k)}=\{X_{k,1}^*,\ldots,X_{k,n_k}^*\}$ be an i.i.d.~sample from $F_{N,K}^{(k)}$ and $U_{N,K}^{*(k)}=U\big(\mathfrak{X}_{N,K}^{*(k)}\big)$, 
then
\begin{align*}
U_{N,K}^{*(k)}= & 2\{n_k(n_k-1)\}^{-1}\sum_{1\leq i<j\leq n_k}h(X_{k,i}^*, X_{k,j}^*) \\
= & \hat{\theta}_{UN,K}^{{(k)}} + n_k^{-1}\sum_{i=1}^{n_k}\hat{\alpha}_U(X^*_{k,i}; F_{N,K}^{(k)}) + n_k^{-2}\sum_{1\leq i < j \leq n_k}\hat{\beta}_U(X^*_{k,i}, X^*_{k,j}; F_{N,K}^{(k)}) + R^{*(k)}_{UN, K},
\end{align*}
where $\hat{\theta}_{UN,K}^{{(k)}}=\E\big(U_{N,K}^{*(k)}|F_{N,K}^{(k)}\big)=n_k^{-2}\sum_{i=1}^{n_k}\sum_{j=1}^{n_k}h(X_{k,i}, X_{k,j})$, $$\hat{\alpha}_U(x; F_{N,K}^{(k)})=2\big[\E\big\{h(x, X_{k,1}^*)|F_{N,K}^{(k)}\big\}-\hat{\theta}_{UN,K}^{{(k)}} \big]$$, and $\hat{\beta}_U(x, y; F_{N,K}^{(k)})=2h(x, y)-\hat{\alpha}_U(x; F_{N,K}^{(k)})-\hat{\alpha}_U(x; F_{N,K}^{(k)})-2\hat{\theta}_{UN,K}^{{(k)}}$.
Thus, $U_{N,K}^{*}=N^{-1}\sum_{k=1}^{K}n_kU_{N,K}^{*(k)}$ is in the form of \eqref{eq:exp1} with $\hat{\theta}_{UN,K}=N^{-1}\sum_{k=1}^{K}n_k\hat{\theta}_{UN,K}^{{(k)}}$ and remainder term 
$R^*_{UN, K}=N^{-1}\sum_{k=1}^{K}n_kR^{*(k)}_{UN, K}$. Represent $\hat{\alpha}_U(x; F_{N,K}^{(k)})$ and $\hat{\beta}_U(x, y; F_{N,K}^{(k)})$ in term of $\alpha_U(x; F)$ and $\beta_U(x, y; F)$, we have
\begin{align*}
\hat{\alpha}_U(x; F_{N,K}^{(k)}) = & \alpha_U(x; F)-n_k^{-1}\sum_{i=1}^{n_k}\alpha_U(X_{k,i}; F)+n_k^{-1}\sum_{i=1}^{n_k}\beta_U(x, X_{k,i}; F) \\
& -n_k^{-2}\sum_{i=1}^{n_k}\sum_{j=1}^{n_k}\beta_U(X_{k,i}, X_{k,j}; F), \\
\hat{\beta}_U(x, y; F_{N,K}^{(k)}) = & \beta_U(x, y; F) - n_k^{-1}\sum_{i=1}^{n_k}\beta_U(x, X_{k,i}; F)-n_k^{-1}\sum_{i=1}^{n_k}\beta_U(X_{k,i}, y; F) \\
& +n_k^{-2}\sum_{i=1}^{n_k}\sum_{j=1}^{n_k}\beta_U(X_{k,i}, X_{k,j}; F).
\end{align*}
By simple algebra, it is checked that $\sum_{i=1}^{n_k}\hat{\alpha}_U(X_{k,i};F_{N,K}^{(k)})=0$, $\hat{\beta}_U(x, y; F_{N,K}^{(k)})$ is symmetric in $x$ and $y$, and $\sum_{i=1}^{n_k}\{\hat{\beta}_U(X_{k,i}, y; F_{N,K}^{(k)})\}=0$ for any $y\in S(F)$. Furthermore, according to the condition $\E\{h(X_1,X_2)\}^2<\infty$ and $\E\{h(X_1,X_1)\}^2<\infty$, $\sup\limits_{x\in S(F)}|\hat{\alpha}_U(x; F_{N,K}^{(k)})-\alpha_U(x;F)|=o_p(1)$ and $\sup\limits_{x,y\in S(F)}|\hat{\beta}_U(x, y; F_{N,K}^{(k)})-\beta_U(x,y;F)|=o_p(1)$ are easily verifiable by law of large numbers.

Note that $R^{*(k)}_{UN, K}=2n_k^{-2}(n_k-1)^{-1}\sum_{1\leq i<j\leq n_k}\hat{\beta}_U(X_{k,i}^*, X_{k,j}^*; F_{N,K}^{(k)})$, $\E\big(R^{*(k)}_{UN, K}\big|F_{N,K}^{(k)}\big)=0$ and $\E\big(R^{*(k)2}_{UN, K}\big|F_{N,K}^{(k)}\big)=O_p(n_k^{-4})$, from which $$\P\left\{|R^*_{UN, K}| \geq N^{-1/2}(ln N)^{-1}\big|F_{N,K}^{(1)},\ldots,F_{N,K}^{(K)}\right\} = o_p(1)$$ can be verified by Cauchy-Schwarz inequality and the condition that $K=O(N^{\tau'})$ for a positive $\tau'<1$. Thus the conditions in Theorem \ref{theo:edg_T_SC_N_boot1} are all satisfied by $U_{N,K}^{*}$. Finally, let $\hat{\sigma}^{2}_{\alpha,UN,K} =N^{-1}\sum_{k=1}^{K}n_k\E\{\hat{\alpha}_U^2(X^*_{k,1}; F_{N,K}^{(k)})|F_{N,K}^{(k)}\}$, Theorem \ref{theo:edg_T_SC_N_boot1} implies that $\hat{\sigma}^{2}_{\alpha,UN,K}-\sigma_{\alpha,U}^2=o_p(1)$ and
\begin{align*}
\sup\limits_{x \in \mathbf{R}}\big|\P\big\{N^{1/2}\hat{\sigma}_{\alpha,UN,K}^{-1}(U^*_{N, K}-\hat{\theta}_{UN,K}) \leq x\big| & F_{N,K}^{(1)},\ldots,F_{N,K}^{(K)}\big\}  \\
& - \P\big\{N^{1/2}\sigma_{\alpha,U}^{-1}(U_{N, K}-\theta_U) \leq x\big\}\big| = o_p(1).
\end{align*}

\subsection{Discussions on the convergence rate of the pseudo-distributed bootstrap}

Under the conditions in Theorem \ref{corr:pdb1}, we assume that 
$\{R_{N,K}^{(k)}\}_{k=1}^K$ are of small  order such that they can be 
omitted in the following discussion. 
Assume $\alpha(X_1;F)$ is non-lattice, $\E\left|\alpha(X_1;F)\right|^4<\infty$ and $\E\left|\beta(X_1, X_2; F)\right|^4<\infty$, by Theorem 5.18 of \citet{Petrov1998},
\begin{align*}
& \P\big\{K^{1/2}\bar{s}^{-1}(\mathcal{T}_{N,K}^{*}-N^{1/2}K^{-1/2}T_{N,K}) \leq x\big|F_{K, \mathcal{T}}\big\} \\
= & \Phi(x) -6^{-1                                                                                                              } K^{-1/2}\bar{s}^{-3}\E\big\{\big(\mathcal{T}_{N,K}^{*(1)}-N^{1/2}K^{-1/2}T_{N,K}\big)^3\big|F_{K,\mathcal{T}}\big\}(x^2-1)\phi(x) + O_p(K^{-1})
\end{align*}
uniformly in $x \in \mathbf{R}$ where $\bar{s}^2=K^{-1}\sum_{k=1}^{K}\big(\mathcal{T}_{N,K}^{(k)}-N^{1/2}K^{-1/2}T_{N,K}\big)^2$. Mentioned that $$\E\big\{\big(\mathcal{T}_{N,K}^{*(1)}-N^{1/2}K^{-1/2}T_{N,K}\big)^3\big| F_{K, \mathcal{T}}\big\}=O_p(n_k^{-1/2}+K^{-1/2}),$$ 
this clearly implies $$\sup\limits_{x \in \mathbf{R}}\left|\P\left\{K^{1/2}\bar{s}^{-1}(\mathcal{T}_{N,K}^{*}-N^{1/2}K^{-1/2}T_{N,K}) \leq x\big|F_{K, \mathcal{T}}\right\}-\Phi(x)\right|=O_p(N^{-1/2}+K^{-1}).$$
According to the proof of Theorem \ref{theo:uniform_convergence},
\begin{align*}
\sup\limits_{x \in \mathbf{R}}\left|\P\left\{N^{1/2}\sigma_{\alpha}^{-1}(T_{N,K}-\theta) \leq x \right\} - \Phi(x)\right| = O(N^{-1/2}),
\end{align*}
it follows that
\begin{align*}
& \sup\limits_{x \in \mathbf{R}}\left|\P\left\{K^{1/2}\bar{s}^{-1}(\mathcal{T}_{N,K}^{*}-N^{1/2}K^{-1/2}T_{N,K}) \leq x\big|F_{K, \mathcal{T}}\right\} - \P\left\{N^{1/2}\sigma_{\alpha}^{-1}(T_{N,K}-\theta) \leq  x\right\}\right| \\
= & O_p(N^{-1/2}+K^{-1}).
\end{align*}
This result indicates that the convergence rate of the pseudo-distributed bootstrap is at the order of $O_p(N^{-1/2}+K^{-1})$.

\subsection{Distributed distance covariance}\label{sec:dist_covar}

Distance covariance, introduced in \citet{SRB2007}, is a method that measures and tests dependence between two random vectors. In this section, we introduce the distributed distance covariance and its usage in testing independence for massive data.

Suppose $\bY$ and $\bZ$ are two random vectors having finite first moments and taking values in $\mathbf{R}^p$ and $\mathbf{R}^q$, respectively. 
The population distance covariance between $\bY$ and $\bZ$ is
$$dcov^2(\bY,\bZ)=\int_{\mathbf{R}^{p+q}}\| \phi_{\bY,\bZ}(\bt,\bs)-\phi_{\bY}(\bt)\phi_{\bZ}(\bs) \|^2\omega(\bt,\bs)d\bt d\bs,$$
where $\phi_{\bY}(\bt)$, $\phi_{\bZ}(\bs)$ and $\phi_{\bY, \bZ}(\bt,\bs)$ are the characteristic functions of $\bY$, $\bZ$ and $\bX=(\bY^T, \bZ^T)^T$, respectively, and $\omega(\bt,\bs)=(c_pc_q\|\bt\|_p^{1+p}\|\bs\|_q^{1+q})^{-1}$ is a weight function with $c_d=\pi^{(1+d)/2}/\Gamma\{(1+d)/2\}$. 
It is clear that $dcov^2(\bY,\bZ)$ equals zero if and only if $\bY$ and $\bZ$ are independent. This property makes the distance variance more versatile in detecting dependence than the classic Pearson's correlation. If $\E\|\bY\|_p+\E\|\bZ\|_q<\infty$,
\begin{align*}
	dcov^2(\bY,\bZ)=\E\|\bY-\bY'\|_p\|\bZ-\bZ'\|_q-2\E\|\bY-\bY'\|_p\|\bZ-\bZ''\|_q+\E\|\bY-\bY'\|_p\E\|\bZ-\bZ'\|_q,
\end{align*}
where $\bX'=(\bY'^{T}, \bZ'^{T})^T$ and $\bX''=(\bY''^{T}, \bZ''^{T})^T$ are independent copies of $\bX$.

Suppose $\mathfrak{X}_N=\{\bX_1, \ldots, \bX_N\}$ is a sample from $F$, where $\bX_i=(\bY_i^T, \bZ_i^T)^T$ for $i=1,\ldots,N$. Define $A_{ij}=\|\bY_i-\bY_j\|_p$ and $B_{ij}=\|\bZ_i-\bZ_j\|_q$ for $i,j=1,\ldots,N$. Then the $\mathcal{U}$-centered empirical distance covariance \citep{SR2014,YZS2016} is
\begin{align*}
	dcov_N^2(\bY,\bZ)= & \left\{N(N-1)\right\}^{-1}\sum_{i\neq j}A_{ij}B_{ij}+\left\{N(N-1)(N-2)\right\}^{-1}\sum_{i\neq j\neq l_1}A_{ij}B_{il_1} \\
					   & + \left\{N(N-1)(N-2)(N-3)\right\}^{-1}\sum_{i\neq j\neq l_1\neq l_2}A_{ij}B_{l_1l_2},
\end{align*}
which is a U-statistic of degree $4$ and an unbiased estimator of $dcov^2(\bY,\bZ)$.

The computing complexity and memory requirement of the distance covariance are both at the order of $O(N^2)$, limit its usage when $N$ is large.
The proposed distributed version of the distance covariance has much to offer under the massive data scenario. Suppose the entire data $\mathfrak{X}_N$ are divided into $K$ sub-samples with the $k$-th subset $\mathfrak{X}_{N,K}^{(k)}=\{\bX_{k,1},\ldots,\bX_{k,n_k}\}$ of size $n_k$ for $k=1,\ldots,K$. Denote $dcov_{N,K}^{2(k)}(\bY,\bZ)$ as the empirical distance covariance based on $\mathfrak{X}_{N,K}^{(k)}$. Then the distributed distance covariance is defined as
\begin{align*}
	dcov_{N,K}^{2}(\bY,\bZ)=N^{-1}\sum_{k=1}^{K}n_kdcov_{N,K}^{2(k)}(\bY,\bZ).
\end{align*}

It is easy to see that $dcov_{N,K}^{2}(\bY,\bZ)$ is also an unbiased estimator of 
$dcov^2(\bY,\bZ)$, and it enjoys computational advantage over $dcov_N^2(\bY,\bZ)$. Specially, the memory requirement of computing $dcov_{N,K}^{2}(\bY,\bZ)$ is only $O(K^{-2}N^2)$, which makes $dcov_{N,K}^{2}(\bY,\bZ)$ more feasible when the size of the data is huge.
Now we consider testing independence between $\bY$ and $\bZ$ using the distance covariance. That is to test the hypothesis
\begin{align}\label{ind_hypo}
	H_0:\bY \text{~and~} \bZ \text{~are independent~~~~~~~~versus~~~~~~~~}H_1:\bY \text{~and~} \bZ \text{~are dependent}.
\end{align}
\indent Under the null hypothesis that $\bY$ and $\bZ$ are independent, if $\E\|\bY\|^2_p+\E\|\bZ\|^2_q<\infty$, the empirical distance covariance $dcov_N^2(\bY,\bZ)$ is a degenerate U-statistic that has the following representation \citep{YZS2016}:
\begin{align*}
	dcov_N^2(\bY,\bZ) = N^{-2}\sum_{1\leq i<j\leq N}\beta(\bX_i, \bX_j; F) + R_N,
\end{align*}
where $\beta(\bX_i, \bX_j; F)=2\mathbf{U}(\bY_i,\bY_j)\mathbf{V}(\bZ_i,\bZ_j)$ with $\mathbf{U}(y,y')=\|y-y'\|_p-\E\|y-\bY'\|_p-\E\|\bY-y'\|_p+\E\|\bY-\bY'\|_p$ and $\mathbf{V}(z,z')=\|z-z'\|_q-\E\|z-\bZ'\|_q-\E\|\bZ-z'\|_q+\E\|\bZ-\bZ'\|_q$, and $R_N=R(\mathfrak{X}_N; F)$ is the remainder term satisfies $\E(R_N)=0$ and $\Var(R_N)=O(N^{-3})$. It is easy to verify that under the null hypothesis, $\E\{\beta(\bX_1, \bX_2; F)|\bX_1\}=0$ and $\E\big\{\beta(\bX_1, \bX_2; F)\big\}^2=4\E\big\{\mathbf{U}(\bY_1,\bY_2)\big\}^2\E\big\{\mathbf{V}(\bZ_1,\bZ_2)\big\}^2\equiv\sigma_{\beta}^2<\infty$.

If $dcov_N^2(\bY,\bZ)$ is used to test \eqref{ind_hypo}, we need to obtain a reference distribution for $dcov_N^2(\bY,\bZ)$ using the bootstrap \citep{Arcones1992} or random permutation \citep{SRB2007} on the entire dataset. 
For massive dataset, computing $dcov_N^2(\bY,\bZ)$ itself is a big issue. Moreover, the bootstrap and random permutation on the entire dataset are both computationally expensive as we need to re-calculate the statistic for each resample or permuted sample. 
Thus we consider using the distributed distance covariance $dcov_{N,K}^{2}(\bY,\bZ)$. 
According to Theorem \ref{theo:T_N_K_degenerate}, we have the following asymptotic result concerning $dcov_{N,K}^{2}(\bY,\bZ)$.

\begin{corollary}\label{corr:dvoc_asy_nor}
	Under Condition \ref{condition:K}, assume there exists a constant $\delta'>0$ such that $\E\|\bY\|^{2+\delta'}_p+\E\|\bZ\|^{2+\delta'}_q<\infty$,
	then under the null hypothesis that $\bY$ and $\bZ$ are independent,
	\begin{align*}
		2^{1/2}K^{-1/2}N\sigma_{\beta}^{-1}dcov_{N,K}^2(\bY,\bZ) \xrightarrow{d} \mathcal{N}(0,1)~~~~\text{as}~N, K\to\infty.
	\end{align*}
\end{corollary}

We need to estimate $\sigma_{\beta}^2$ in order to implement the test. Mentioned that $dcov^2(\bY, \bZ)=\E\big\{\mathbf{U}(\bY_1,\bY_2)\mathbf{V}(\bZ_1,\bZ_2)\big\}$, $dcov^2(\bY, \bY)=\E\big\{\mathbf{U}(\bY_1,\bY_2)\big\}^2$ and $dcov^2(\bZ, \bZ)=\E\big\{\mathbf{V}(\bZ_1,\bZ_2)\big\}^2 $, an distributed version of unbiased estimator for $\sigma_{\beta}^2$ under the null is
\begin{align*}
	\hat{\sigma}_{\beta, N, K}^2 = 4\left\{N^{-1}\sum_{k=1}^{K}n_kdcov_{N,K}^{2(k)}(\bY,\bY)\right\}\left\{N^{-1}\sum_{k=1}^{K}n_kdcov_{N,K}^{2(k)}(\bZ,\bZ)\right\}.
\end{align*}
The consistency of the variance estimator $\hat{\sigma}_{\beta, N, K}^2$ is established in the next theorem.

\begin{theorem}\label{theo:dcov_consist}
	Under the conditions in Corollary \ref{corr:dvoc_asy_nor}, if $\E\|\bY\|^4_p+\E\|\bZ\|^4_q<\infty$, then,
	\begin{align*}
		\hat{\sigma}_{\beta, N, K}^2/\sigma_{\beta}^2 \xrightarrow{p} 1 ~~~~\text{as}~~ N\to\infty.
	\end{align*}
\end{theorem}

\begin{proof}
	
	Under the null hypothesis, $\E\big(\hat{\sigma}_{\beta,N,K}^2\big)=\sigma_{\beta}^2$, thus it is sufficient to show that $\Var\big(\hat{\sigma}_{\beta,N,K}^2\big)=o(1)$ as $N\to\infty$. Mentioned that
	\begin{align*}
	\Var\big(\hat{\sigma}_{\beta,N,K}^2\big) = & \Var\left[4\left\{N^{-1}\sum_{k=1}^{K}n_kdcov_{N,K}^{2(k)}(\bY,\bY)\right\}\left\{N^{-1}\sum_{k=1}^{K}n_kdcov_{N,K}^{2(k)}(\bZ,\bZ)\right\}\right] \\
	= & 16\E\left\{N^{-1}\sum_{k=1}^{K}n_kdcov_{N,K}^{2(k)}(\bY,\bY)\right\}^2\E\left\{N^{-1}\sum_{k=1}^{K}n_kdcov_{N,K}^{2(k)}(\bZ,\bZ)\right\}^2 \\
	& - 16dcov^4(\bY,\bY)dcov^4(\bZ,\bZ) \\
	= & 16N^{-2}\sum_{k=1}^{K}n_k^2\Var\left\{dcov_{N,K}^{2(k)}(\bY,\bY)\right\}dcov^4(\bZ,\bZ) \\
	& + 16N^{-2}\sum_{k=1}^{K}n_k^2\Var\left\{dcov_{N,K}^{2(k)}(\bZ,\bZ)\right\}dcov^4(\bY,\bY) \\
	& + 16N^{-4}\left[\sum_{k=1}^{K}n_k^2\Var\left\{dcov_{N,K}^{2(k)}(\bY,\bY)\right\}\right]\left[\sum_{k=1}^{K}n_k^2\Var\left\{dcov_{N,K}^{2(k)}(\bZ,\bZ)\right\}\right].
	\end{align*}
	Since $dcov_{N,K}^{2(k)}(\bY,\bY)$ and $dcov_{N,K}^{2(k)}(\bZ,\bZ)$ are both U-statistics of degree $4$, under the condition that $\E\|\bY\|^4_p+\E\|\bZ\|^4_q<\infty$, $\Var\big\{dcov_{N,K}^{2(k)}(\bY,\bY)\big\}=O(n_k^{-1})$ and $\Var\big\{dcov_{N,K}^{2(k)}(\bZ,\bZ)\big\}=O(n_k^{-1})$, these immediately lead to $\Var\big(\hat{\sigma}_{\beta,N,K}^2\big)=o(1)$ and the proof is complete.
	
\end{proof}

By Slutsky's theorem,  $2^{1/2}K^{-1/2}N\hat{\sigma}_{\beta, N, K}^{-1}dcov_{N,K}^2(\bY,\bZ)$ converges to standard normal distribution under the null. Thus, a consistent test with nominal significant level $\tau$ reject $H_0$ if
\begin{align}\label{eq:dist_covar_test1}
	2^{1/2}K^{-1/2}N\hat{\sigma}_{\beta, N, K}^{-1}dcov_{N,K}^2(\bY,\bZ) > z_{\tau},
\end{align}
where $z_{\tau}$ is the $\tau$-th upper-quantile of the standard normal distribution.

\begin{remark}
	When $\bY$ and $\bZ$ are dependent and $dcov_N^2(\bY,\bZ)$ is non-degenerate, $dcov_{N,K}^{2}(\bY,\bZ)$ still converges to a normal distribution under certain moment conditions (Theorem \ref{theo:T_N_K_asym_normal_dist}). That is,
	\begin{align*}
		[\Var\{dcov_{N,K}^{2}(\bY,\bZ)\}]^{-1/2}\{dcov_{N,K}^{2}(\bY,\bZ)-dcov^{2}(\bY,\bZ)\}\xrightarrow{d} \mathcal{N}(0,1)
	\end{align*}
	as $N\to\infty$. In addition, we can use the pseudo-distributed bootstrap to approximate the distribution of $dcov_{N,K}^{2}(\bY,\bZ)$ and get a consistent estimator of $\Var\{dcov_{N,K}^{2}(\bY,\bZ)\}$ when $K\to\infty$. 
\end{remark}

Next we consider the special case when each data block has the same size $NK^{-1}$. 
Define
\begin{align*}
	\hat{\sigma}_{N,K}^2 = K^{-2}\sum_{k=1}^{K}\left\{dcov_{N,K}^{2(k)}(\bY,\bZ)-dcov_{N,K}^{2}(\bY,\bZ)\right\}^2.
\end{align*}
According to the proof of Theorem \ref{corr:pdb2}, $\hat{\sigma}_{N,K}^2$ is a consistent estimator of $\Var\{dcov_{N,K}^{2}(\bY,\bZ)\}$ under the conditions in Corollary \ref{corr:dvoc_asy_nor}. Thus $\hat{\sigma}_{N,K}^{-1}dcov_{N,K}^{2}(\bY,\bZ)$ is asymptotic standard normal under the null. Therefore, a consistent test with nominal significant level $\tau$ reject $H_0$ if
\begin{align}\label{eq:dist_covar_test3}
	\hat{\sigma}_{N,K}^{-1}dcov_{N,K}^{2}(\bY,\bZ) > z_{\tau}.
\end{align}

\begin{remark}
	When each data block has the same size, even when $dcov_N^2(\bY,\bZ)$ is non-degenerate, $\hat{\sigma}_{N,K}^2$ is still a consistent estimator of $\Var\{dcov_{N,K}^{2}(\bY,\bZ)\}$ if $K\to\infty$ (proof of Theorem \ref{corr:pdb1}). Thus $\hat{\sigma}_{N,K}^{-1}dcov_{N,K}^{2}(\bY,\bZ)$ is asymptotic normal with unit variance regardless $dcov_N^2(\bY,\bZ)$ is degenerate or not. 
	Therefore, we can use $\hat{\sigma}_{N,K}^{-1}dcov_{N,K}^{2}(\bY,\bZ)$ to measure the dependence between $\bY$ and $\bZ$ for massive data. The larger $\hat{\sigma}_{N,K}^{-1}dcov_{N,K}^{2}(\bY,\bZ)$ is, the more dependent $\bY$ and $\bZ$ are.
\end{remark}



\section{Simulations on distributed distance covariance}\label{sec:sup_simu}

The distributed version of distance covariance and its usage in testing independence between two multivariate random vectors has been studied in Section \ref{sec:dist_covar}. 
In this section, we investigate the performance of the distributed distance covariance $dcov_{N,K}^2(\bY,\bZ)$ in testing independence by numerical studies. 
All simulation results in this section are based on $1000$ iterations with the nominal significant level at $5\%$. The sample size is fixed at $N=100000$ and the number of data blocks $K$ is selected in the set $\{20, 50, 100, 200, 500, 1000, 2000, 5000\}$.

For the null hypothesis, we generated i.i.d.~samples $\{\bY_1, \ldots, \bY_N\}$ and $\{\bZ_1, \ldots, \bZ_N\}$ independently from distributions $\bG_1$ and $\bG_2$, respectively. Three different combinations of $\bG_1$ and $\bG_2$ were considered: (I) $\bG_1$ and $\bG_2$ are both $\mathcal{N}(\mathbf{0}_{p}, \mathbf{I}_{p})$, where $\mathbf{I}_{p}$ is the $p$-dimensional identity matrix; (II) $\bG_1$ is $\mathcal{N}(\mathbf{0}_{p}, \mathbf{I}_{p})$, for $\bG_2$, its $p$ components are i.i.d.~from student-$t$ distribution with $6$ degrees of freedom; (III) For both $\bG_1$ and $\bG_2$, their components are i.i.d.~from student-$t$ distribution with $6$ degrees of freedom. 
The dimension $p$ was chosen as $5$, $10$, $20$ and $40$ for each scenario.

Table \ref{tab:simu_size} reports the empirical sizes of 
$T_{\Var}$ and $T_{\mathrm{PDB}}$, which stand for the testing procedures based on \eqref{eq:dist_covar_test1} and \eqref{eq:dist_covar_test3}, respectively. Table \ref{tab:simu_size} shows that the empirical sizes of both methods are close to the nominal level $5\%$ for all combinations of $p$ and $K$ under all three scenarios. Thus, these two test procedures both have good control of Type-I error for a wide range of $K$. In addition, the performance of these two methods is very comparable to each other. However, for relatively large $K$, $T_{\mathrm{PDB}}$ has slightly better control of the empirical sizes than $T_{\Var}$.

\renewcommand{\arraystretch}{1}
\begin{table}[h]
	\centering
	\caption{Sizes of Independence tests based on $dcov_{N,K}^2(\bY,\bZ)$. $T_{\Var}$: test using variance estimation in (\ref{eq:dist_covar_test1}); $T_{\mathrm{PDB}}$: test using the pseudo-distributed bootstrap.}
	\begin{tabular}{c|cccccccc}
		\hline
		$K$	   & \multicolumn{2}{c}{$p=5$} & \multicolumn{2}{c}{$p=10$} & \multicolumn{2}{c}{$p=20$} & \multicolumn{2}{c}{$p=40$} \\
		& $T_{\Var}$ & $T_{\mathrm{PDB}}$ & $T_{\Var}$ & $T_{\mathrm{PDB}}$ & $T_{\Var}$ & $T_{\mathrm{PDB}}$ & $T_{\Var}$ & $T_{\mathrm{PDB}}$ \\\hline
		& \multicolumn{8}{c}{scenario I} \\
		$20$   & $0.051$ & $0.050$ & $0.053$ & $0.058$ & $0.053$ & $0.057$ & $0.059$ & $0.068$ \\
		$50$   & $0.055$ & $0.047$ & $0.046$ & $0.049$ & $0.051$ & $0.052$ & $0.050$ & $0.051$ \\
		$100$  & $0.051$ & $0.048$ & $0.041$ & $0.043$ & $0.050$ & $0.051$ & $0.050$ & $0.059$ \\
		$200$  & $0.045$ & $0.041$ & $0.035$ & $0.035$ & $0.061$ & $0.060$ & $0.048$ & $0.048$ \\
		$500$  & $0.051$ & $0.048$ & $0.049$ & $0.044$ & $0.053$ & $0.053$ & $0.048$ & $0.048$ \\
		$1000$ & $0.060$ & $0.052$ & $0.046$ & $0.042$ & $0.071$ & $0.071$ & $0.049$ & $0.047$ \\
		$2000$ & $0.067$ & $0.062$ & $0.059$ & $0.053$ & $0.057$ & $0.050$ & $0.060$ & $0.054$ \\
		$5000$ & $0.063$ & $0.050$ & $0.070$ & $0.045$ & $0.063$ & $0.043$ & $0.077$ & $0.061$ \\\hline
		& \multicolumn{8}{c}{scenario II} \\
		$20$   & $0.071$ & $0.065$ & $0.059$ & $0.060$ & $0.058$ & $0.059$ & $0.050$ & $0.063$ \\
		$50$   & $0.060$ & $0.058$ & $0.046$ & $0.044$ & $0.043$ & $0.043$ & $0.059$ & $0.062$ \\
		$100$  & $0.052$ & $0.044$ & $0.040$ & $0.044$ & $0.058$ & $0.056$ & $0.063$ & $0.067$ \\
		$200$  & $0.048$ & $0.050$ & $0.045$ & $0.043$ & $0.051$ & $0.051$ & $0.069$ & $0.067$ \\
		$500$  & $0.042$ & $0.042$ & $0.044$ & $0.044$ & $0.053$ & $0.054$ & $0.062$ & $0.065$ \\
		$1000$ & $0.069$ & $0.061$ & $0.052$ & $0.052$ & $0.056$ & $0.054$ & $0.063$ & $0.064$ \\
		$2000$ & $0.071$ & $0.063$ & $0.059$ & $0.053$ & $0.056$ & $0.047$ & $0.064$ & $0.059$ \\
		$5000$ & $0.056$ & $0.043$ & $0.062$ & $0.045$ & $0.066$ & $0.055$ & $0.064$ & $0.052$ \\\hline
		& \multicolumn{8}{c}{scenario III} \\
		$20$   & $0.048$ & $0.040$ & $0.041$ & $0.041$ & $0.042$ & $0.057$ & $0.045$ & $0.054$ \\
		$50$   & $0.054$ & $0.047$ & $0.044$ & $0.048$ & $0.056$ & $0.050$ & $0.041$ & $0.041$ \\
		$100$  & $0.053$ & $0.048$ & $0.049$ & $0.048$ & $0.050$ & $0.050$ & $0.046$ & $0.047$ \\
		$200$  & $0.046$ & $0.046$ & $0.041$ & $0.042$ & $0.052$ & $0.050$ & $0.035$ & $0.036$ \\
		$500$  & $0.058$ & $0.053$ & $0.046$ & $0.047$ & $0.045$ & $0.042$ & $0.053$ & $0.054$ \\
		$1000$ & $0.042$ & $0.040$ & $0.049$ & $0.045$ & $0.046$ & $0.042$ & $0.051$ & $0.049$ \\
		$2000$ & $0.049$ & $0.049$ & $0.058$ & $0.052$ & $0.050$ & $0.043$ & $0.053$ & $0.048$ \\
		$5000$ & $0.069$ & $0.053$ & $0.068$ & $0.048$ & $0.075$ & $0.062$ & $0.056$ & $0.043$ \\\hline
	\end{tabular}
	\label{tab:simu_size}
\end{table}

To compare the powers of these two tests, we generated i.i.d.~samples $\{\bY_1, \ldots, \bY_N\} \sim \bG_1$, $\{\bZ_1, \ldots, \bZ_N\} \sim \bG_2$, and the same three different combinations of $\bG_1$ and $\bG_2$ were considered as before. For $\bY_i=(Y_{i1}, \ldots, Y_{ip})^T$ and $\bZ_i=(Z_{i1}, \ldots, Z_{ip})^T$ under each scenario, we simulated $cor(Y_{ij}, Z_{ik})=\varrho^{|j-k-p|}$ for $j, k = 1, \ldots, p$ and $\varrho=0.05$, $0.1$ were considered. Under these setups, $\bY_i$ and $\bZ_i$ are dependent with each other. 
Figure \ref{fig:dcov} gives the empirical powers of $T_{\Var}$ and $T_{\mathrm{PDB}}$.

\begin{figure}[h]
	\centering
	\subfigure[Scenario I, $\rho=0.05$]{%
		\includegraphics[width=0.45\linewidth]{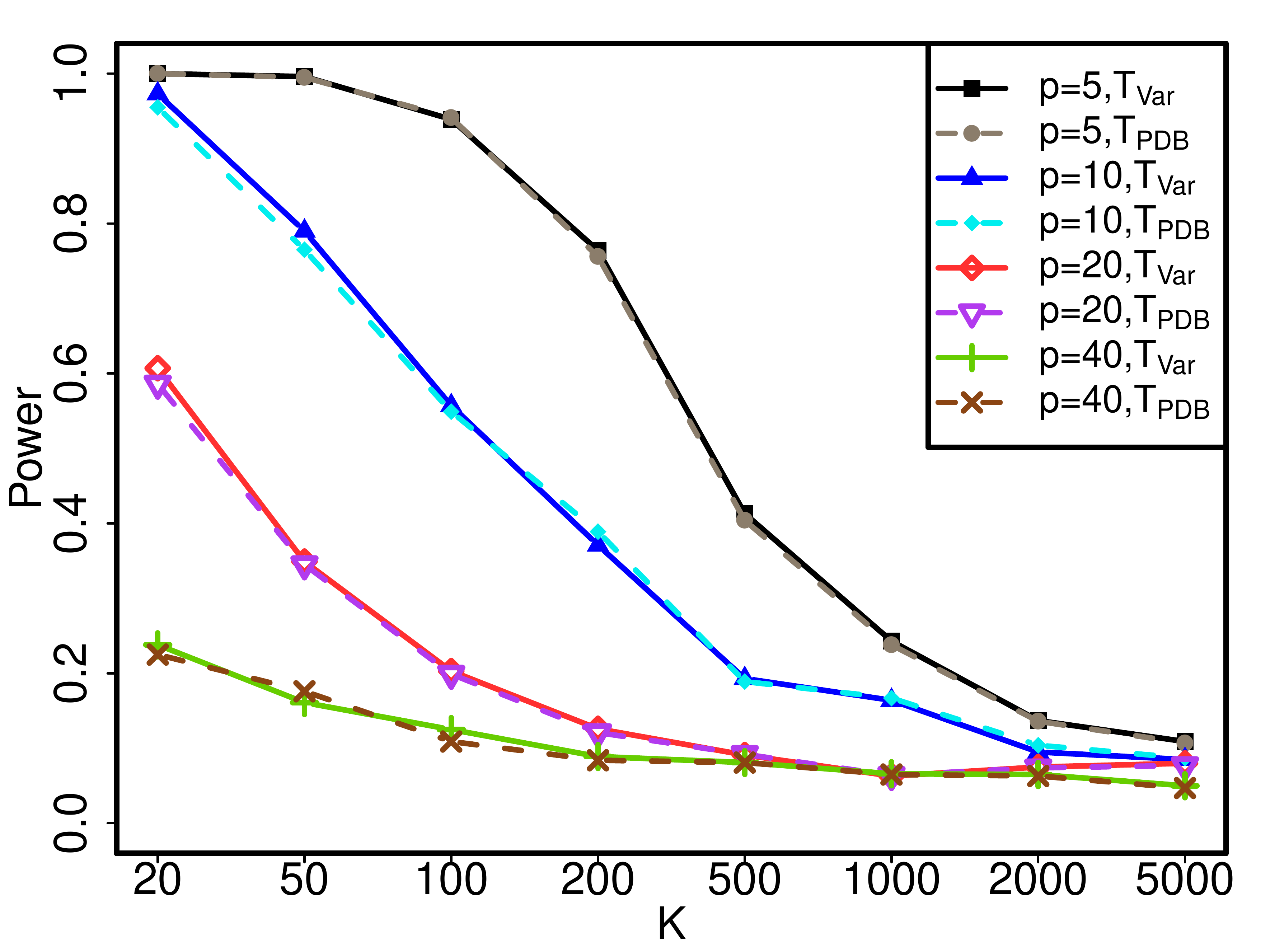}
		\label{fig:dcov_n_5}}
	\quad
	\subfigure[Scenario I, $\rho=0.1$]{%
		\includegraphics[width=0.45\linewidth]{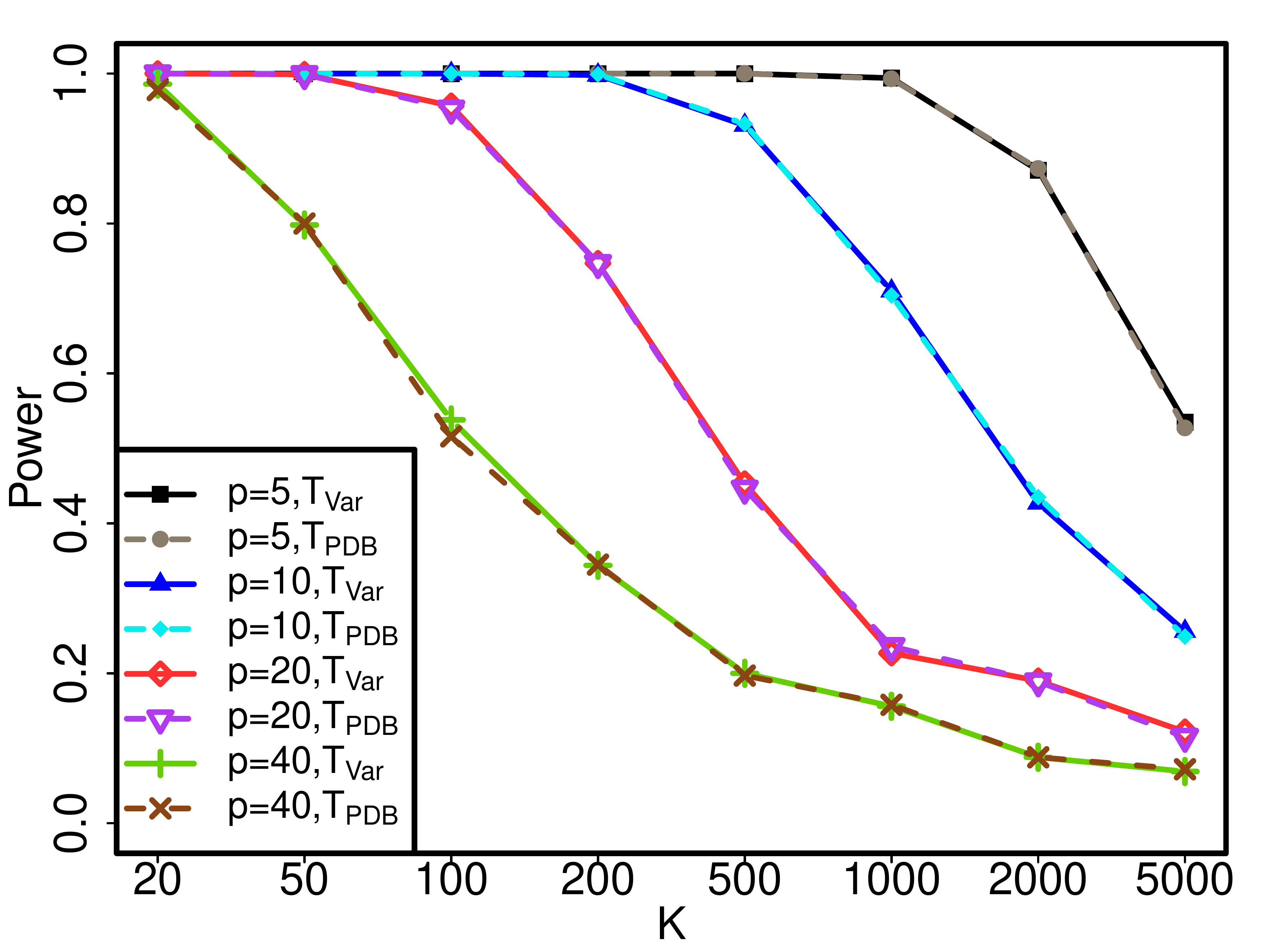}
		\label{fig:dcov_n_10}}
	\quad
	\subfigure[Scenario II, $\rho=0.05$]{%
		\includegraphics[width=0.45\linewidth]{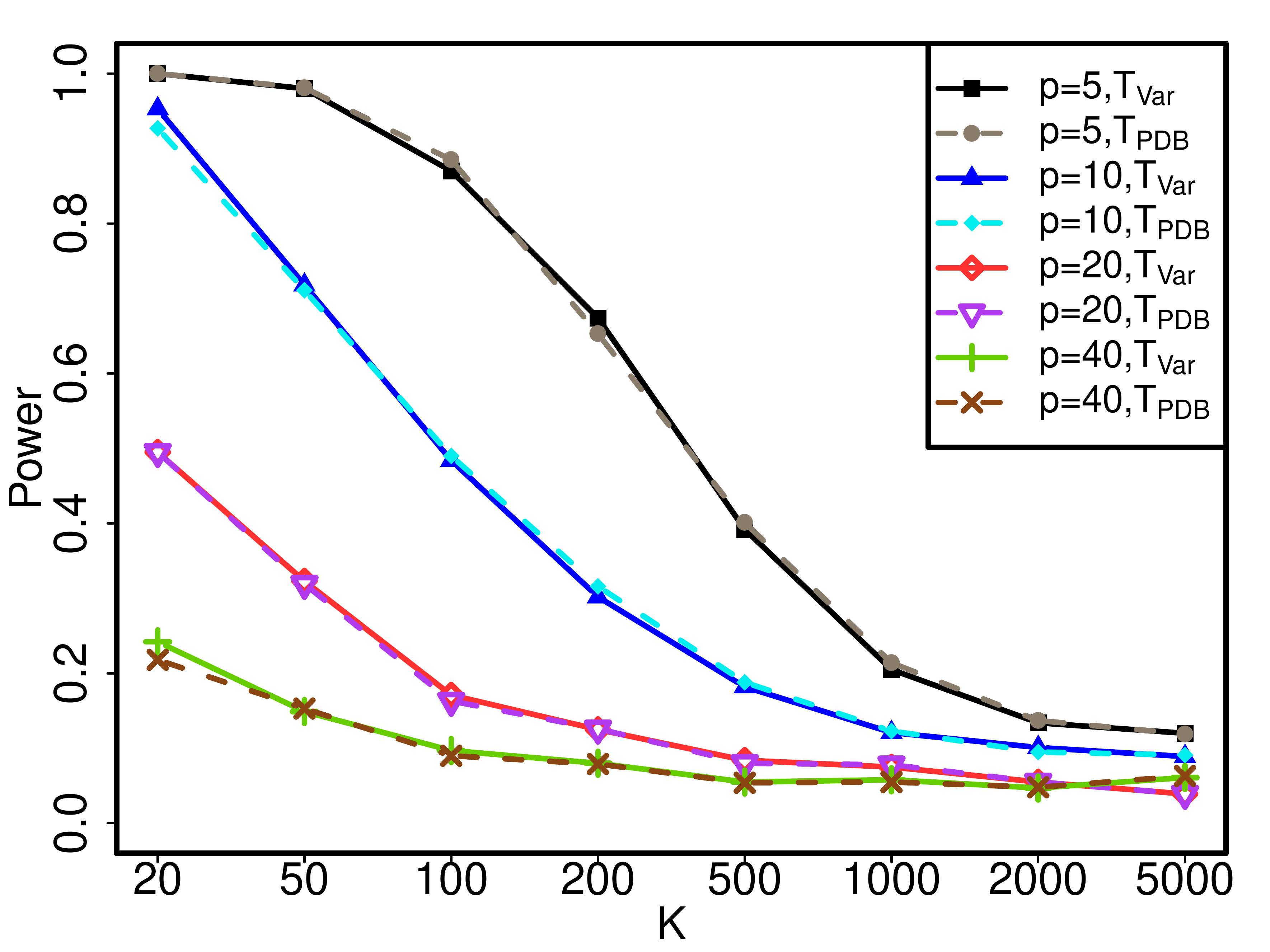}
		\label{fig:dcov_nt_5}}
	\quad
	\subfigure[Scenario II, $\rho=0.1$]{%
		\includegraphics[width=0.45\linewidth]{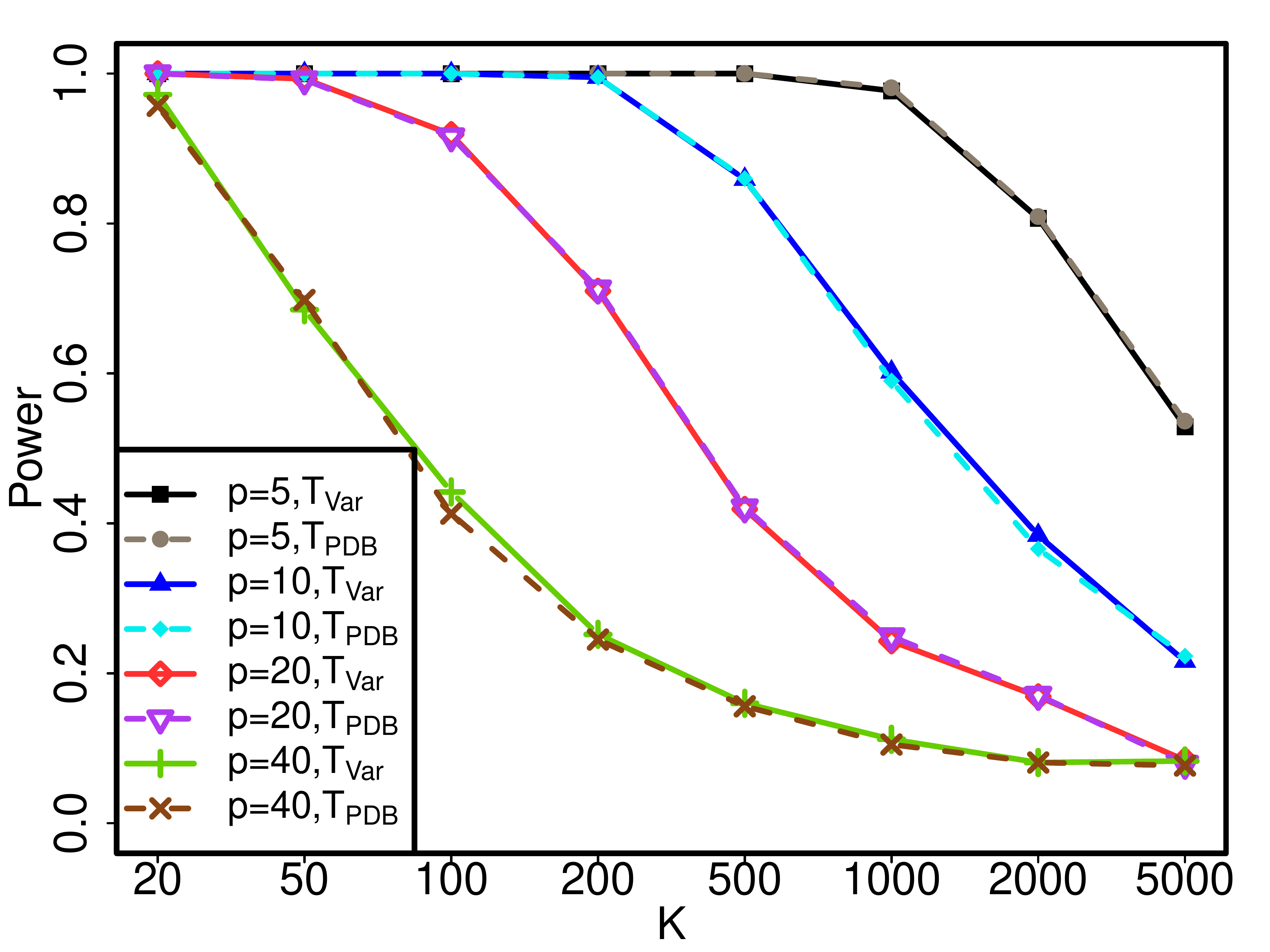}
		\label{fig:dcov_nt_10}}
	\quad
	\subfigure[Scenario III, $\rho=0.05$]{%
		\includegraphics[width=0.45\linewidth]{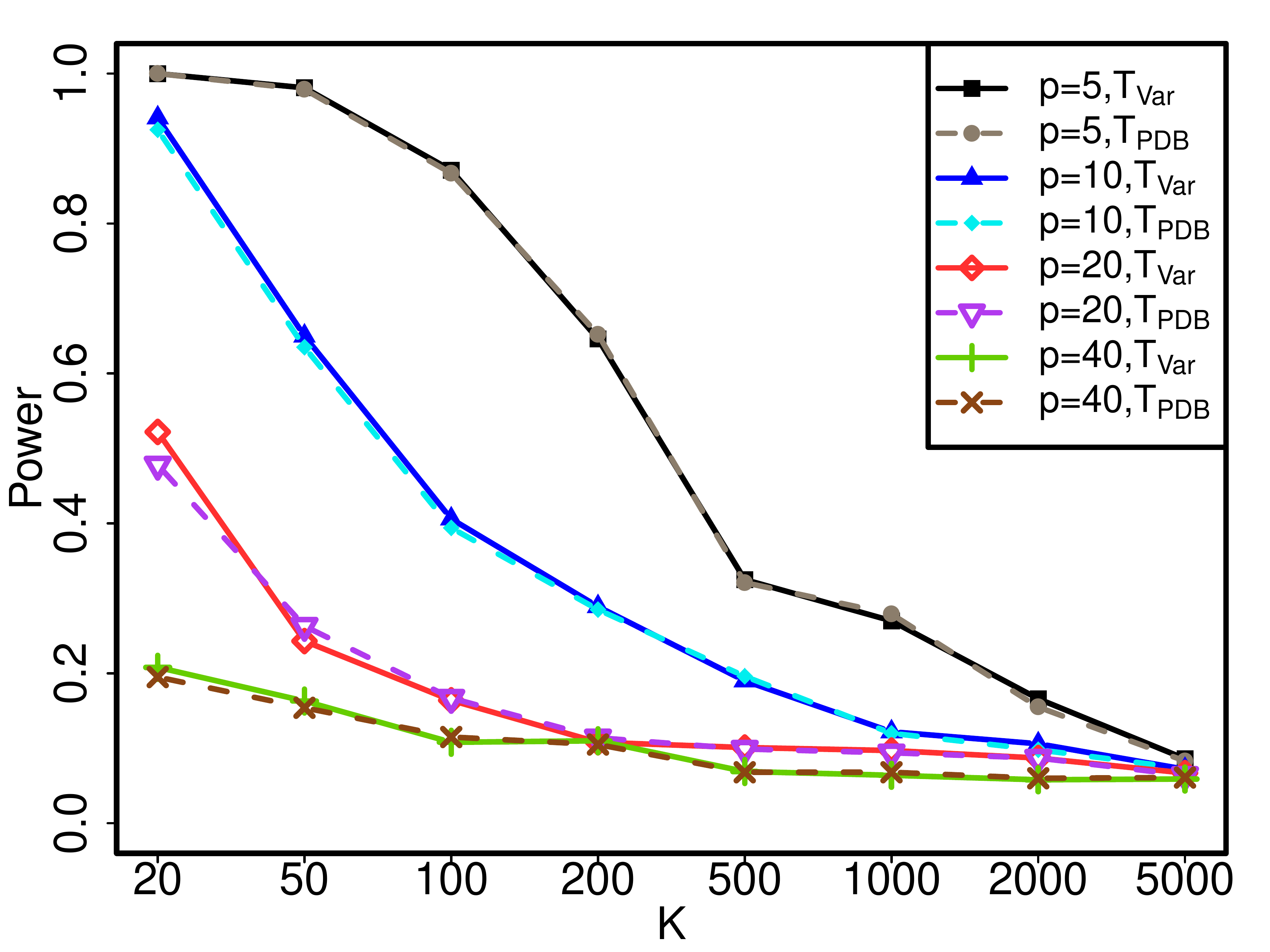}
		\label{fig:dcov_t_5}}
	\quad
	\subfigure[Scenario III, $\rho=0.1$]{%
		\includegraphics[width=0.45\linewidth]{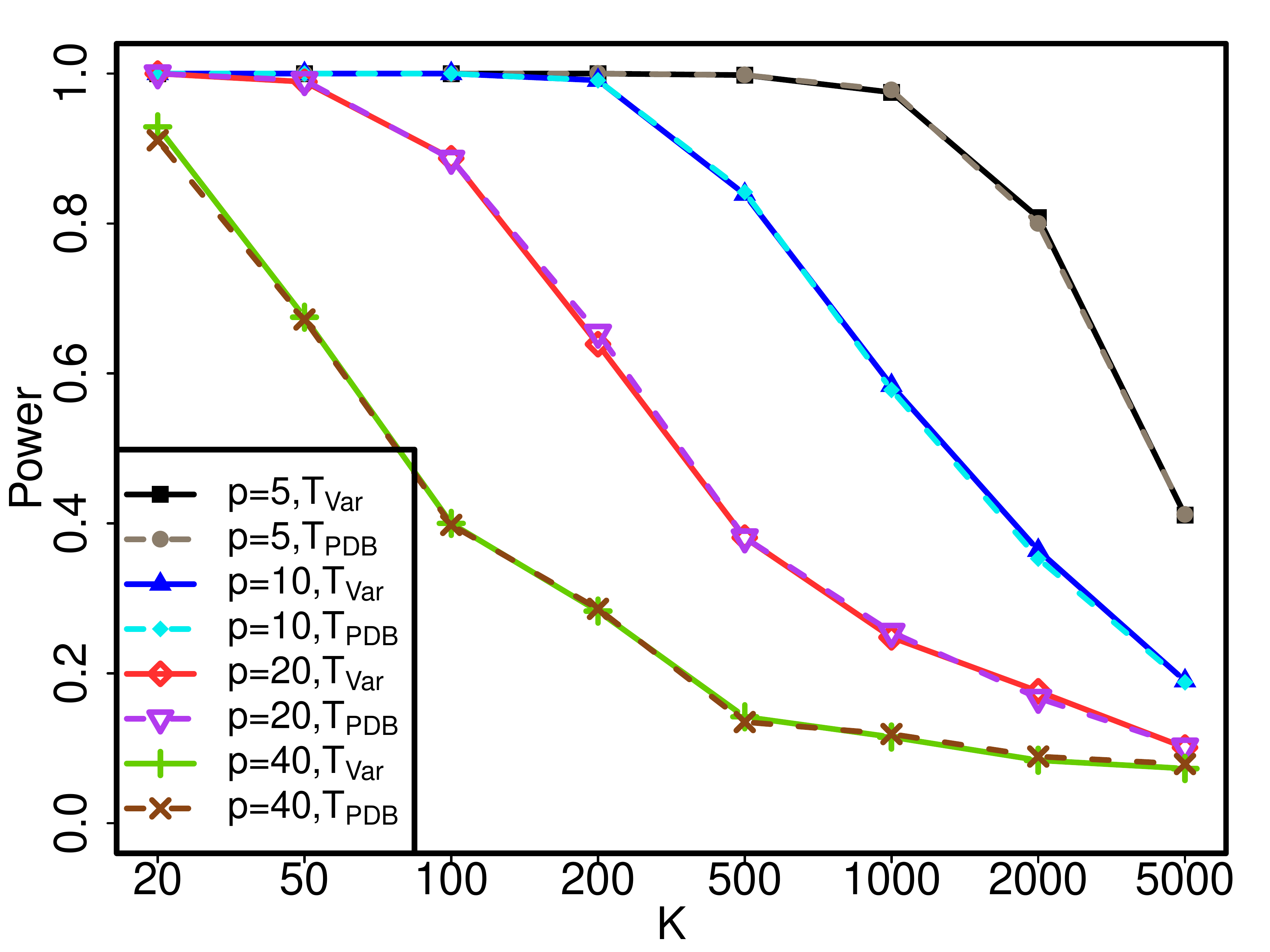}
		\label{fig:dcov_t_10}}
	\quad	
	\caption{Powers of Independence tests based on $dcov_{N,K}^2(\bY,\bZ)$.}
	\label{fig:dcov}
\end{figure}

From Figure \ref{fig:dcov}, it is clear that the empirical powers of these two tests decrease as the dimension $p$ increases. In addition, as the number of data blocks $K$ increases, the empirical powers of the tests also decrease. This is due to the increase in the variance of $dcov_{N,K}^2(\bY,\bZ)$ when $K$ increases. This is the price we need to pay for using the distributed distance covariance. The computing time and memory requirement can be reduced by increasing the number of data blocks, however, this will result in the power loss of the tests.

\subsection{Additional tables and figures}

\begin{table}[p]
	\centering
	\caption{Algorithm for the distributed bootstrap.}
	\begin{tabular}{|l|}
		\hline
		\\
		\textit{Input}: data $\mathfrak{X}_{N,K}^{(1)}, \ldots, \mathfrak{X}_{N,K}^{(K)}$; $K$, number of subsets; $n_1,\ldots,n_K$, subset sizes; \\
		\ \ \ \ \  $B$, number of Monte Carlo iterations; $T$, function deriving statistic \\	
		\textit{Output}: an estimate of the distribution of $N^{-1/2}(T_{N,K}-\theta)$ \\	 
		For $ k \leftarrow 1 $ to $K$ do \\
		\ \ \ \ \ 	compute $\hat{\theta}^{(k)}_{N,K}=\theta(F_{N,K}^{(k)})$ \\
		\ \ \ \ \ 	for $ b\leftarrow1$ to $B$ do \\
		\ \ \ \ \ \ \ \ \ \  generate resample $\mathfrak{X}_{N,K}^{*b(k)}$ from $\mathfrak{X}_{N,K}^{(k)}$ \\
		\ \ \ \ \ \ \ \ \ \  compute $T_{N,K}^{*b(k)}=T\big(\mathfrak{X}_{N,K}^{*b(k)}\big)$ \\
		\ \ \ \ \ 	end \\ 
		End \\
		Compute $\hat{\theta}_{N,K}=N^{-1}\sum_{k=1}^{K}n_k\hat{\theta}^{(k)}_{N,K}$ \\
		For $b\leftarrow1$ to $B$ do \\
		\ \ \ \ \ 	compute $T_{N,K}^{*b}=N^{-1}\sum_{k=1}^{K}n_kT_{N,K}^{*b(m)}$ \\
		\ \ \ \ \  compute $N^{1/2}(T_{N,K}^{*b}-\hat{\theta}_{N,K})$ \\
		End \\
		\\
		\hline
	\end{tabular}
	\label{tab:01}
\end{table}

\begin{table}[h]
	\centering
	\caption{Algorithm for the pseudo-distributed bootstrap.}
	\begin{tabular}{|l|}
		\hline
		\\
		\textit{Input}: data $\mathfrak{X}_{N,K}^{(1)}, \ldots, \mathfrak{X}_{N,K}^{(K)}$; $K$, number of subsets; $n_1,\ldots,n_K$, subset sizes; \\
		\ \ \ \ \  $B$, number of Monte Carlo iterations; $T$, function deriving statistic \\	 
		\textit{Output}: an estimate of the distribution of $N^{1/2}(T_{N, K}-\theta)$ \\	 
		For $ k \leftarrow 1 $ to $ K $ do \\
		\ \ \ \ \ 	compute $T_{N, K}^{(k)}=T\big(\mathfrak{X}_{N,K}^{(k)}\big)$ \\
		\ \ \ \ \   compute $\mathcal{T}_{N,K}^{(k)}=N^{-1/2}K^{1/2}n_kT_{N,K}^{(k)}$ \\
		End \\
		Compute $T_{N,K}=N^{-1}\sum_{k=1}^Kn_kT_{N, K}^{(k)}$ \\
		For $ b \leftarrow 1 $ to $ B $ do \\
		\ \ \ \ \  generate resample $\mathcal{T}_{N, K}^{*b(1)},\ldots, \mathcal{T}_{N, K}^{*b(K)}$ from $\mathcal{T}_{N, K}^{(1)},\ldots, \mathcal{T}_{N, K}^{(K)}$ \\
		\ \ \ \ \  compute $\mathcal{T}_{N,K}^{*b}=K^{-1}\sum_{k=1}^K\mathcal{T}_{N, K}^{*b(k)}$ \\
		End \\
		For $b\leftarrow1$ to $B$ do \\
		\ \ \ \ \ 	compute $K^{1/2}\big(\mathcal{T}_{N,K}^{*b}-N^{1/2}K^{-1/2}T_{N,K}\big)$ \\
		End \\
		\\
		\hline
	\end{tabular} 
	\label{tab:02}
\end{table}

\renewcommand{\arraystretch}{1}
\begin{table}[h]
	\centering
	\caption{Coverage probabilities and widths (in parentheses) of the $95\%$ confidence intervals of the four bootstrap methods with $10$ seconds time budget for the Gamma data. DB: the distributed bootstrap; PDB: the pseudo-distributed bootstrap; BLB: the bag of little bootstrap; SDB: the subsampled double bootstrap.}
	\begin{tabular}{c|cccccc}
		\hline
		& \multicolumn{6}{c}{K} \\
		& $20$ & $50$ & $100$ & $200$ & $500$ & $1000$ \\\hline
		PDB & $0.924$ & $0.936$ & $0.940$ & $0.942$ & $0.942$ & $0.944$ \\
		& $(0.01940)$ & $(0.01998)$ & $(0.02023)$ & $(0.02031)$ & $(0.02033)$ & $(0.02041)$ \\
		\hline
		DB & $0.928$ & $0.938$ & $0.942$ & $0.946$ & $0.942$ & $0.946$ \\
		& $(0.01963)$ & $(0.02002)$ & $(0.02024)$ & $(0.02018)$ & $(0.02022)$ & $(0.02011)$ \\
		\hline
		BLB & NA & $0.934$ & $0.936$ & $0.926$ & $0.930$ & $0.930$ \\
		& (NA) & $(0.01951)$ & $(0.01948)$ & $(0.01938)$ & $(0.01925)$ & $(0.01899)$ \\
		\hline
		SDB & $0.926$ & $0.938$ & $0.940$ & $0.948$ & $0.944$ & $0.948$ \\
		& $(0.01964)$ & $(0.02003)$ & $(0.02024$ & $(0.02024)$ & $(0.02040)$ & $(0.02036)$ \\\hline
	\end{tabular}
	\label{tab:simu_cr_gam}
\end{table}

\renewcommand{\arraystretch}{1}
\begin{table}[h]
	\centering
	\caption{Coverage probabilities and widths (in parentheses) of the $95\%$ confidence intervals of the four bootstrap methods with $10$ seconds time budget for the Poisson data.}
	\begin{tabular}{c|cccccc}
		\hline
		& \multicolumn{6}{c}{K} \\
		& $20$ & $50$ & $100$ & $200$ & $500$ & $1000$ \\\hline
		PDB & $0.921$ & $0.938$ & $0.941$ & $0.945$ & $0.949$ & $0.949$ \\
		& $(0.01986)$ & $(0.02032)$ & $(0.02046)$ & $(0.02053)$ & $(0.02059)$ & $(0.02066)$ \\
		\hline
		DB & $0.944$ & $0.932$ & $0.946$ & $0.944$ & $0.948$ & $0.952$ \\
		& $(0.01987)$ & $(0.02017)$ & $(0.02035)$ & $(0.02052)$ & $(0.02046)$ & $(0.02042)$ \\
		\hline
		BLB & NA & $0.936$ & $0.940$ & $0.938$ & $0.932$ & $0.932$ \\
		& (NA) & $(0.01970)$ & $(0.01968)$ & $(0.01962)$ & $(0.01953)$ & $(0.01940)$ \\
		\hline
		SDB & $0.922$ & $0.936$ & $0.952$ & $0.946$ & $0.950$ & $0.948$ \\
		& $(0.01976)$ & $(0.02017)$ & $(0.02041)$ & $(0.02049)$ & $(0.02054)$ & $(0.02052)$ \\\hline
	\end{tabular}
	\label{tab:simu_cr_pois}
\end{table}

\begin{figure}[htp]
	\centering
	\subfigure{%
		\includegraphics[width=0.45\linewidth]{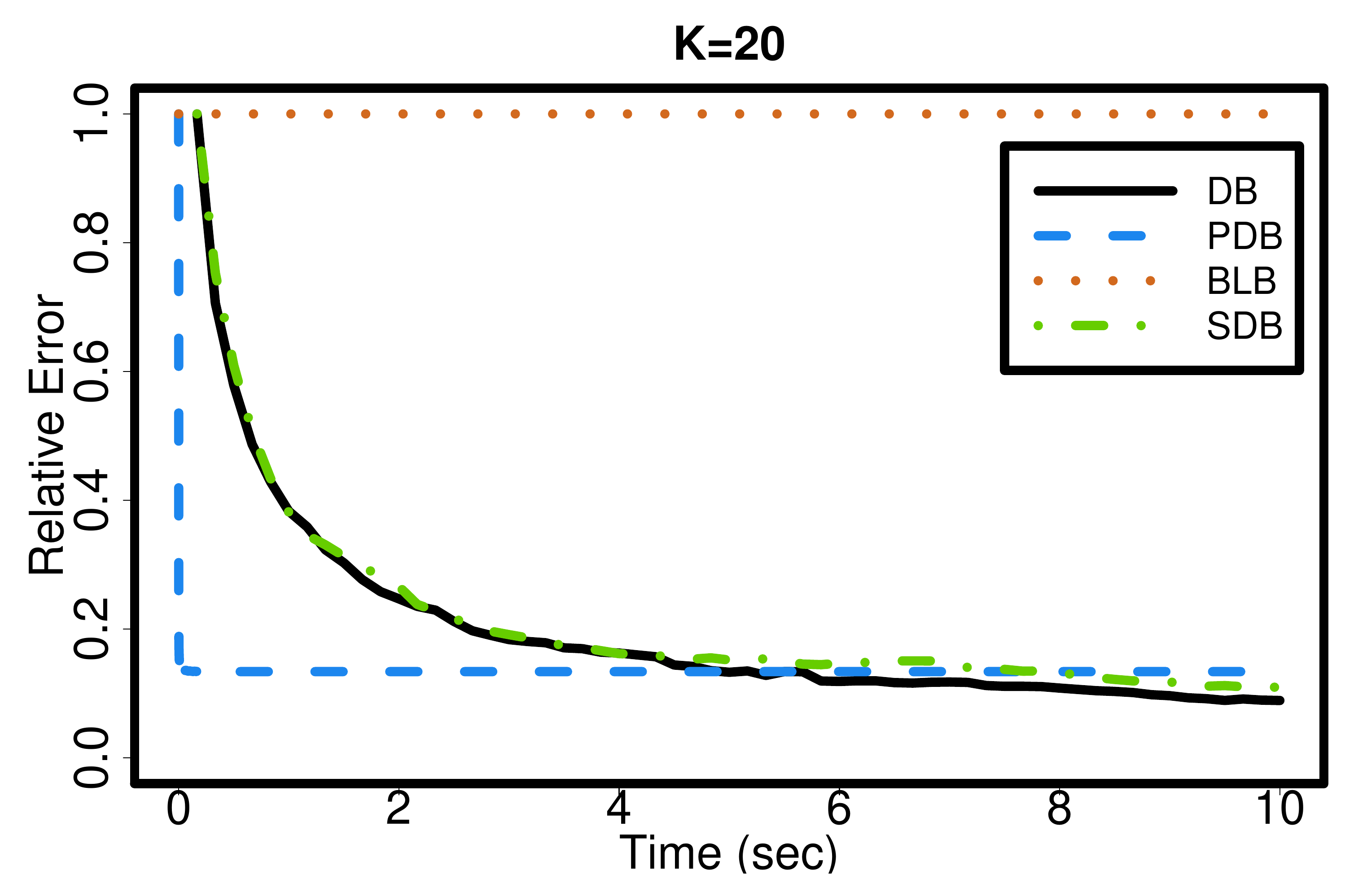}
		\label{fig:gam_K_20}}
	\quad
	\subfigure{%
		\includegraphics[width=0.45\linewidth]{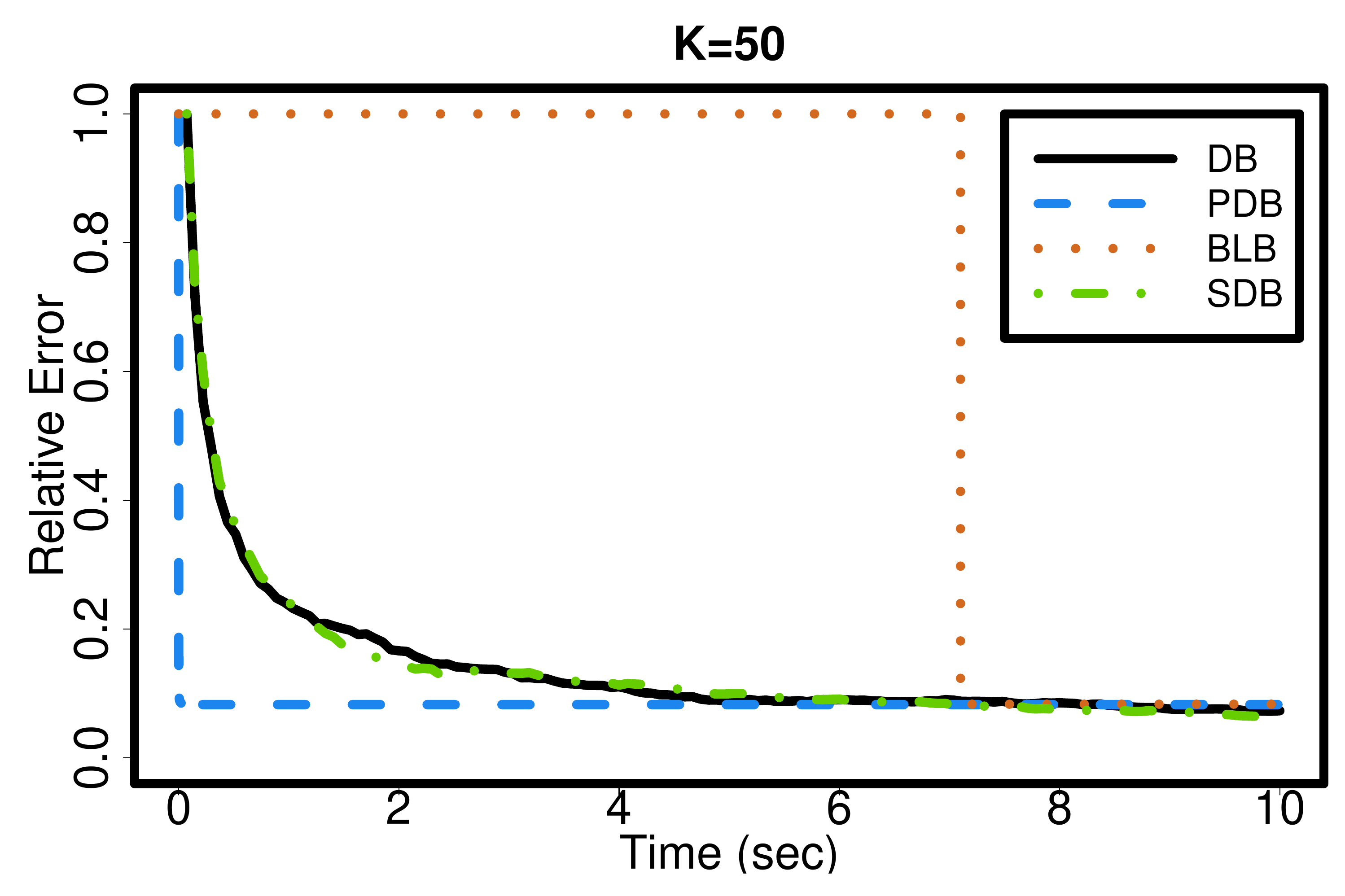}
		\label{fig:gam_K_50}}
	\quad
	\subfigure{%
		\includegraphics[width=0.45\linewidth]{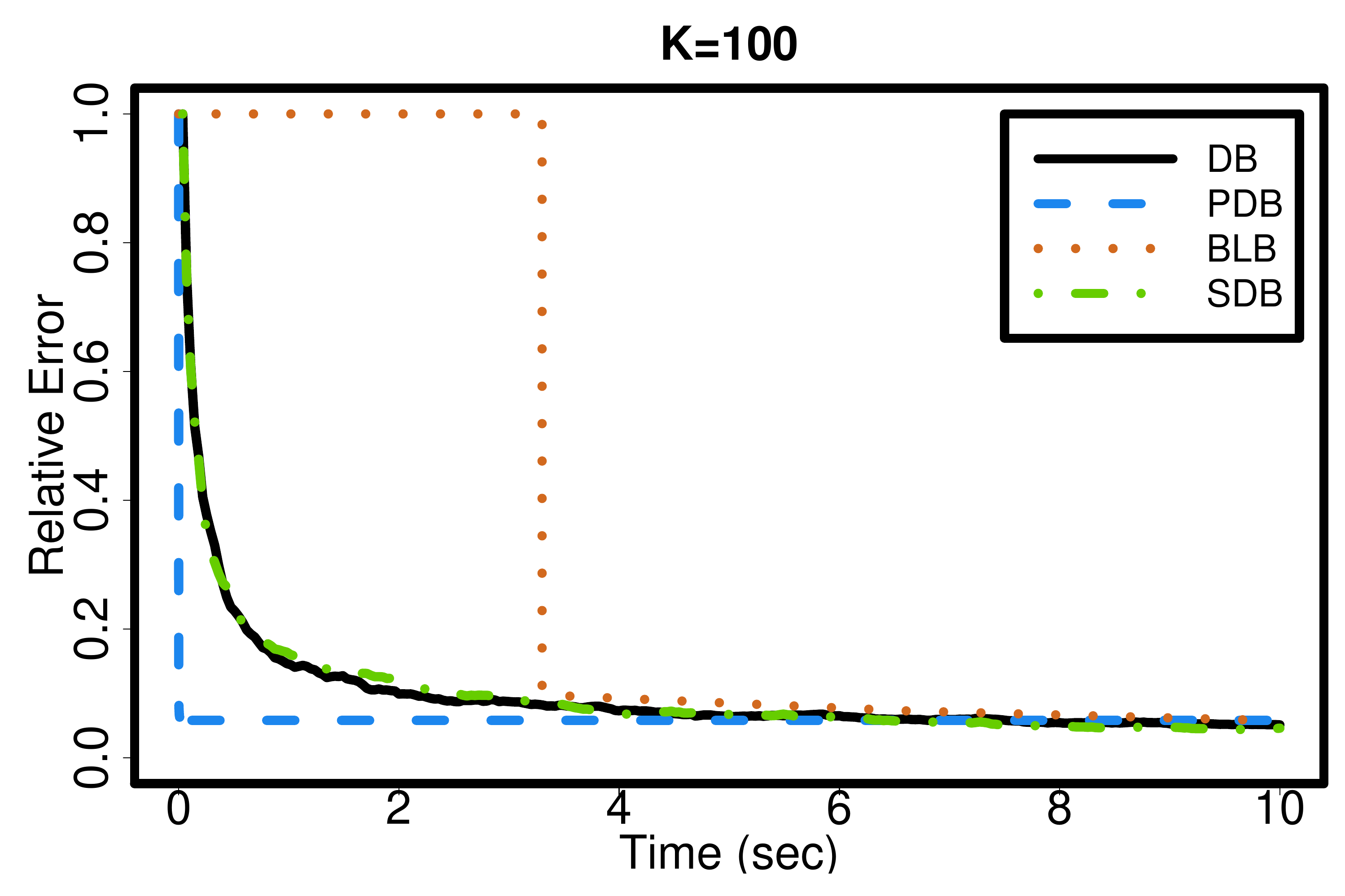}
		\label{fig:gam_K_100}}
	\quad
	\subfigure{%
		\includegraphics[width=0.45\linewidth]{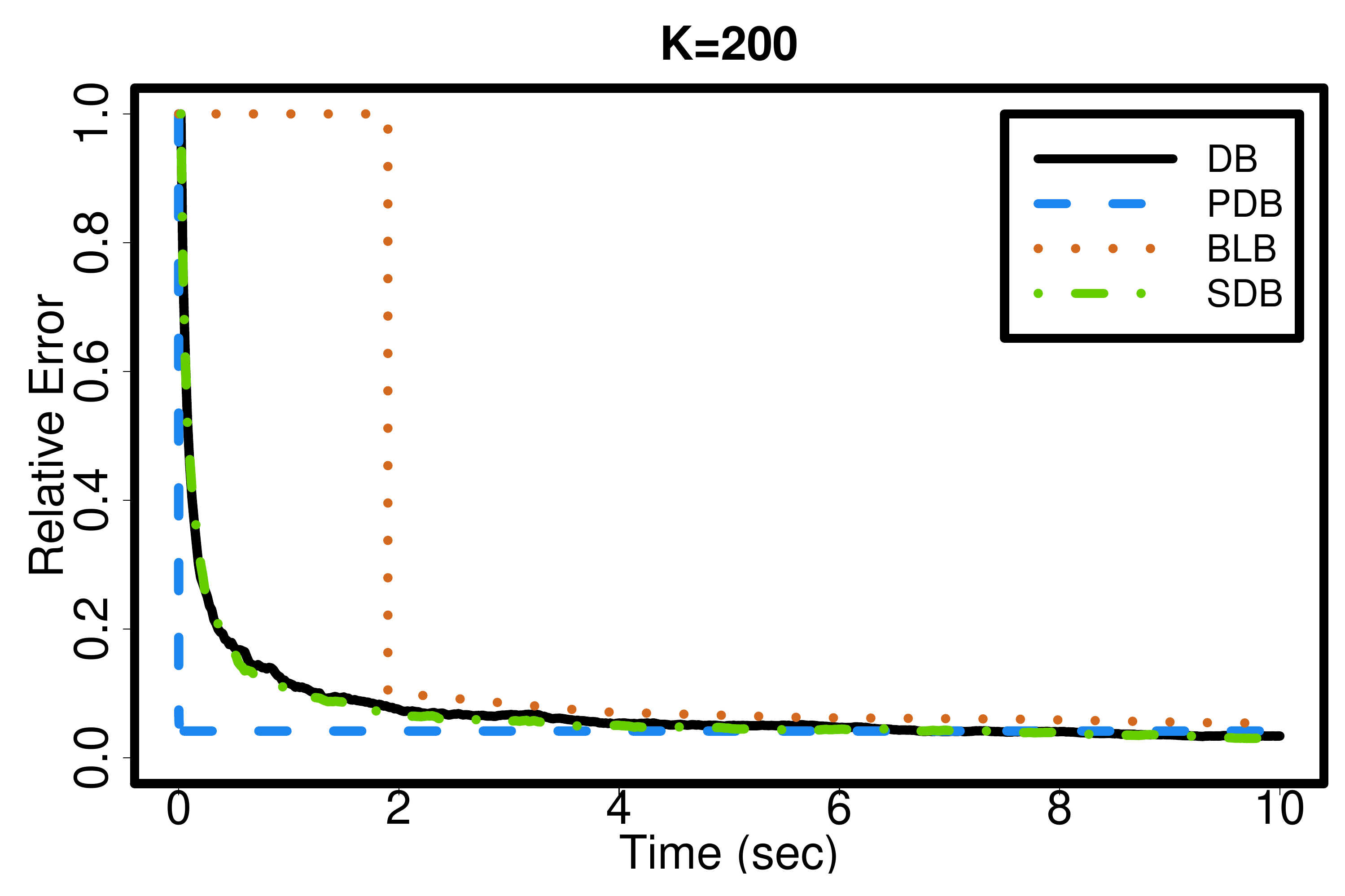}
		\label{fig:gam_K_200}}
	\quad
	\subfigure{%
		\includegraphics[width=0.45\linewidth]{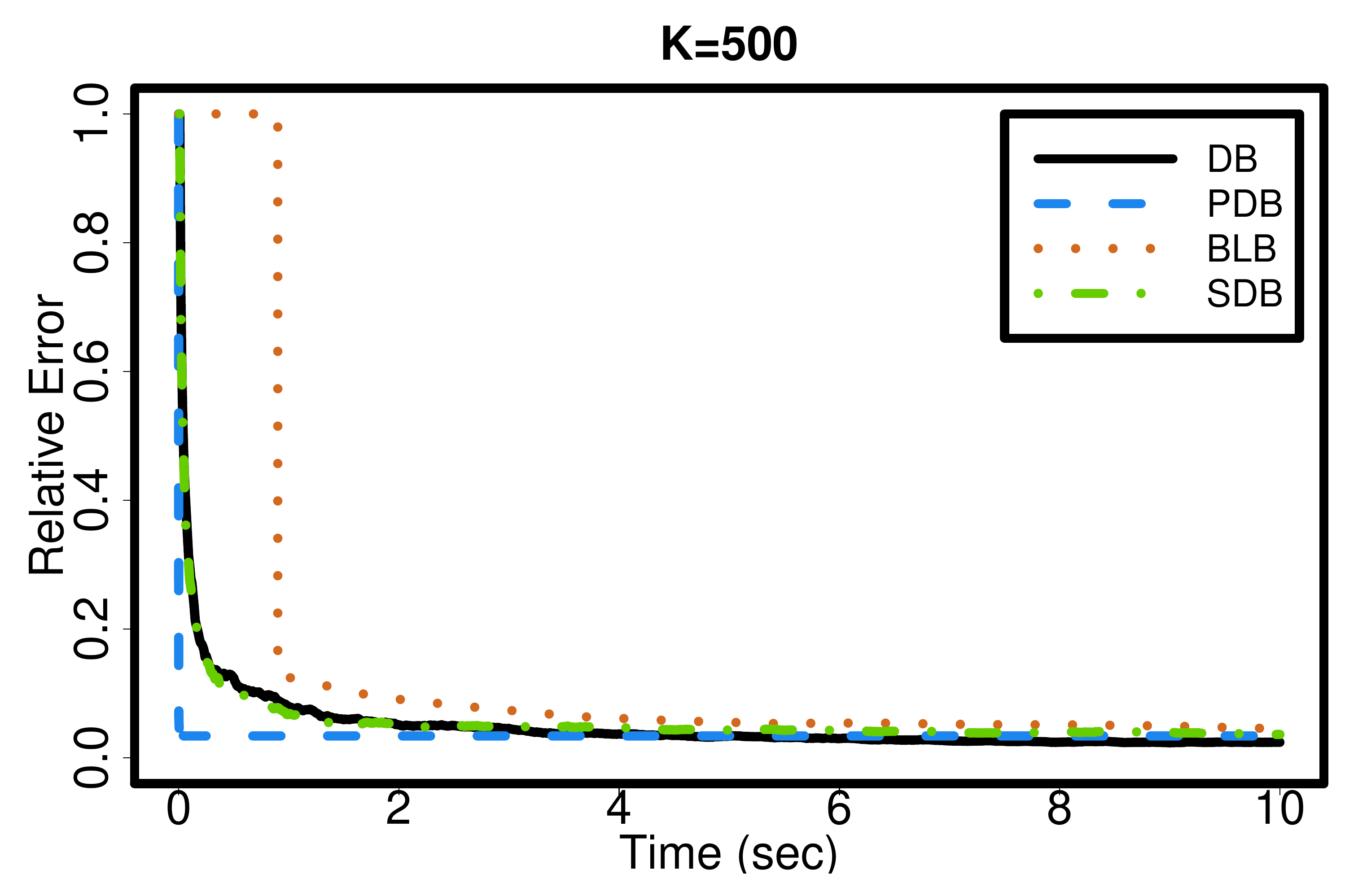}
		\label{fig:gam_K_500}}
	\quad
	\subfigure{%
		\includegraphics[width=0.45\linewidth]{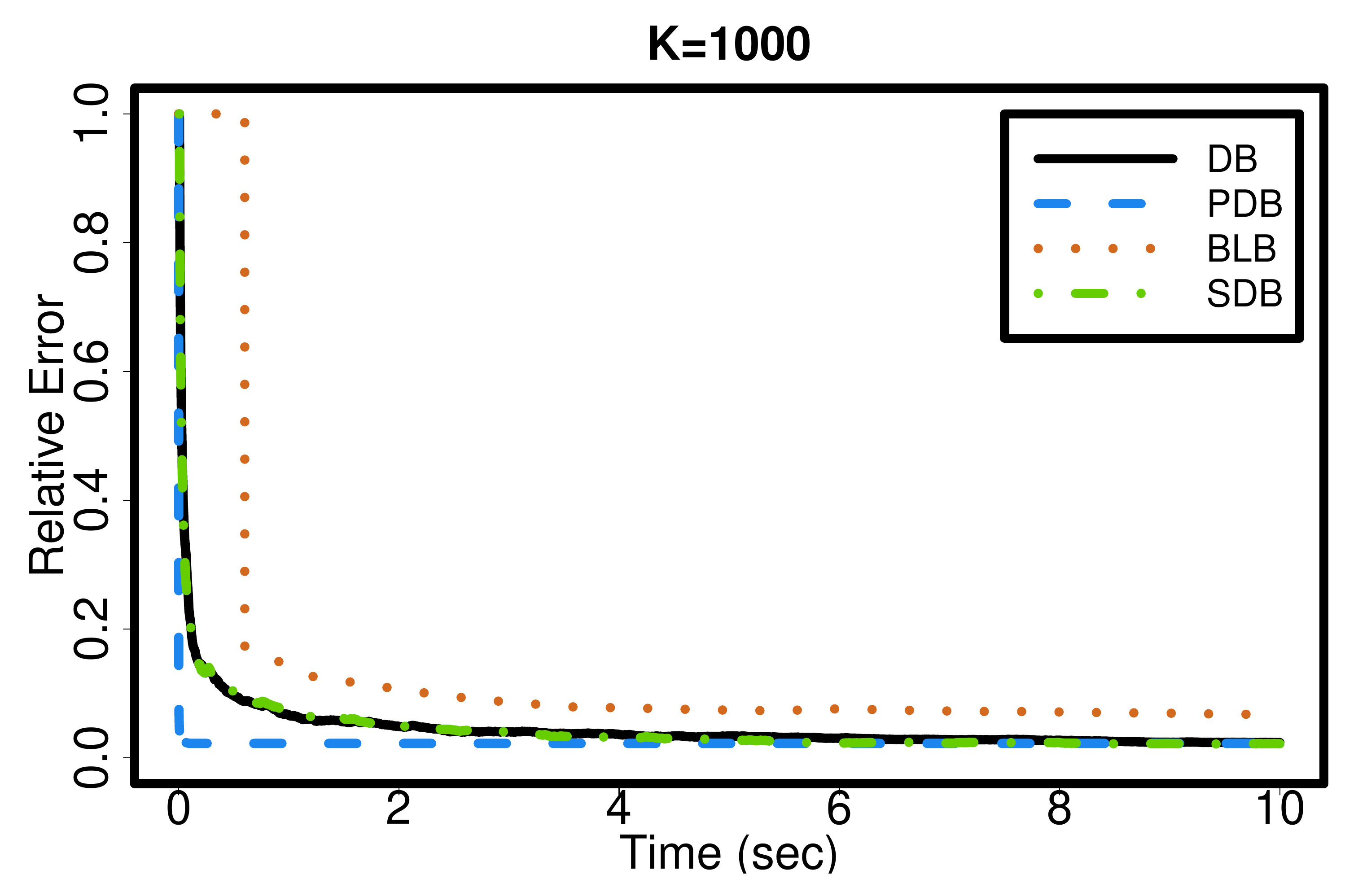}
		\label{fig:gam_K_1000}}
	\quad	
	\caption{Time evolution of the relative errors $|\hat{d}-d|/d$ in the width of the confidence interval under the Gamma scenario with respect to different black size. DB: the distributed bootstrap; PDB: the pseudo-distributed bootstrap (dashed lines); BLB: the bag of little bootstrap (dotted lines); SDB: the subsampled double bootstrap (dot-dashed lines).}
	\label{fig:gam_ciw}
\end{figure}

\begin{figure}[htp]
	\centering
	\subfigure{%
		\includegraphics[width=0.45\linewidth]{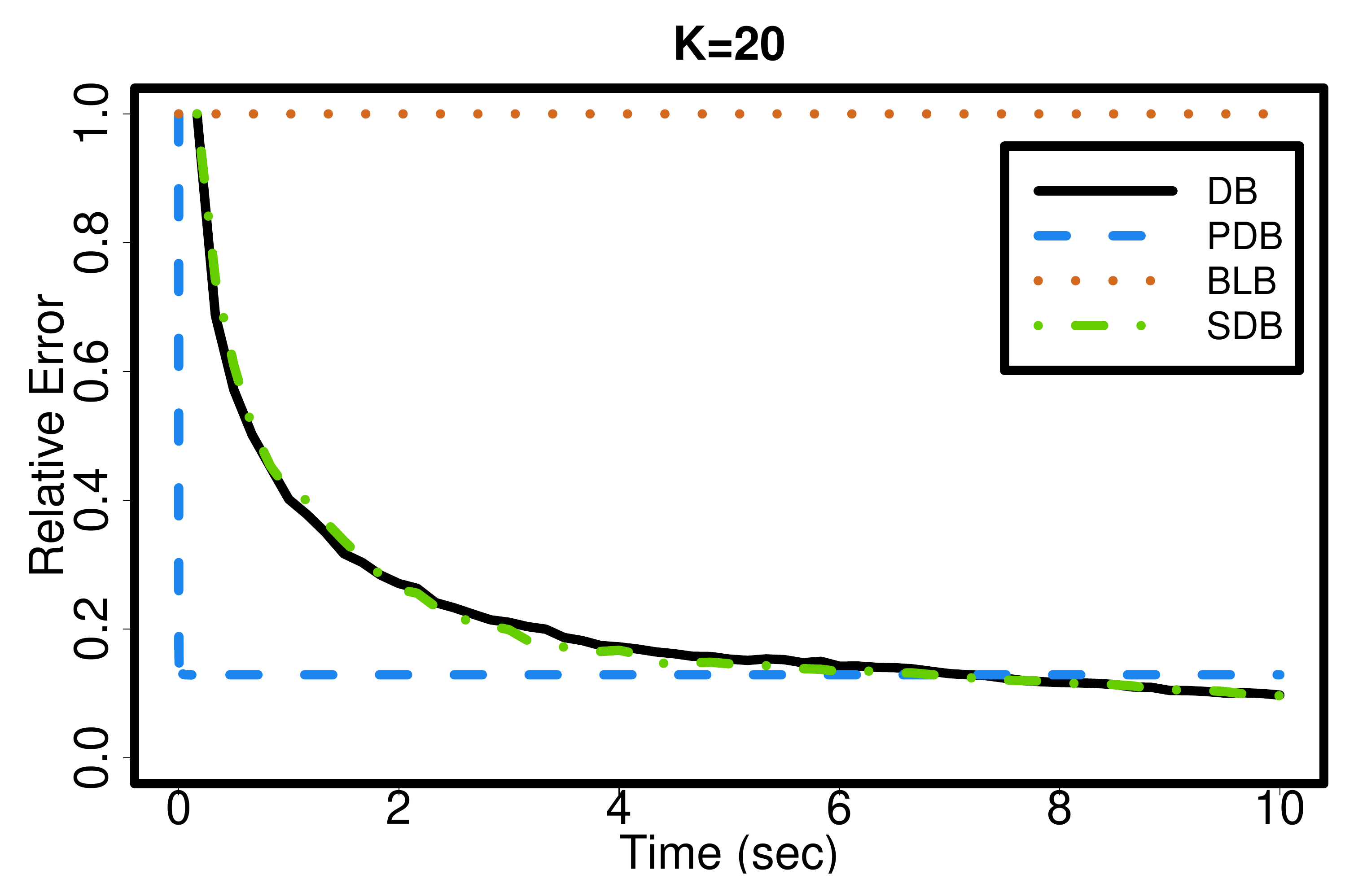}
		\label{fig:pois_K_20}}
	\quad
	\subfigure{%
		\includegraphics[width=0.45\linewidth]{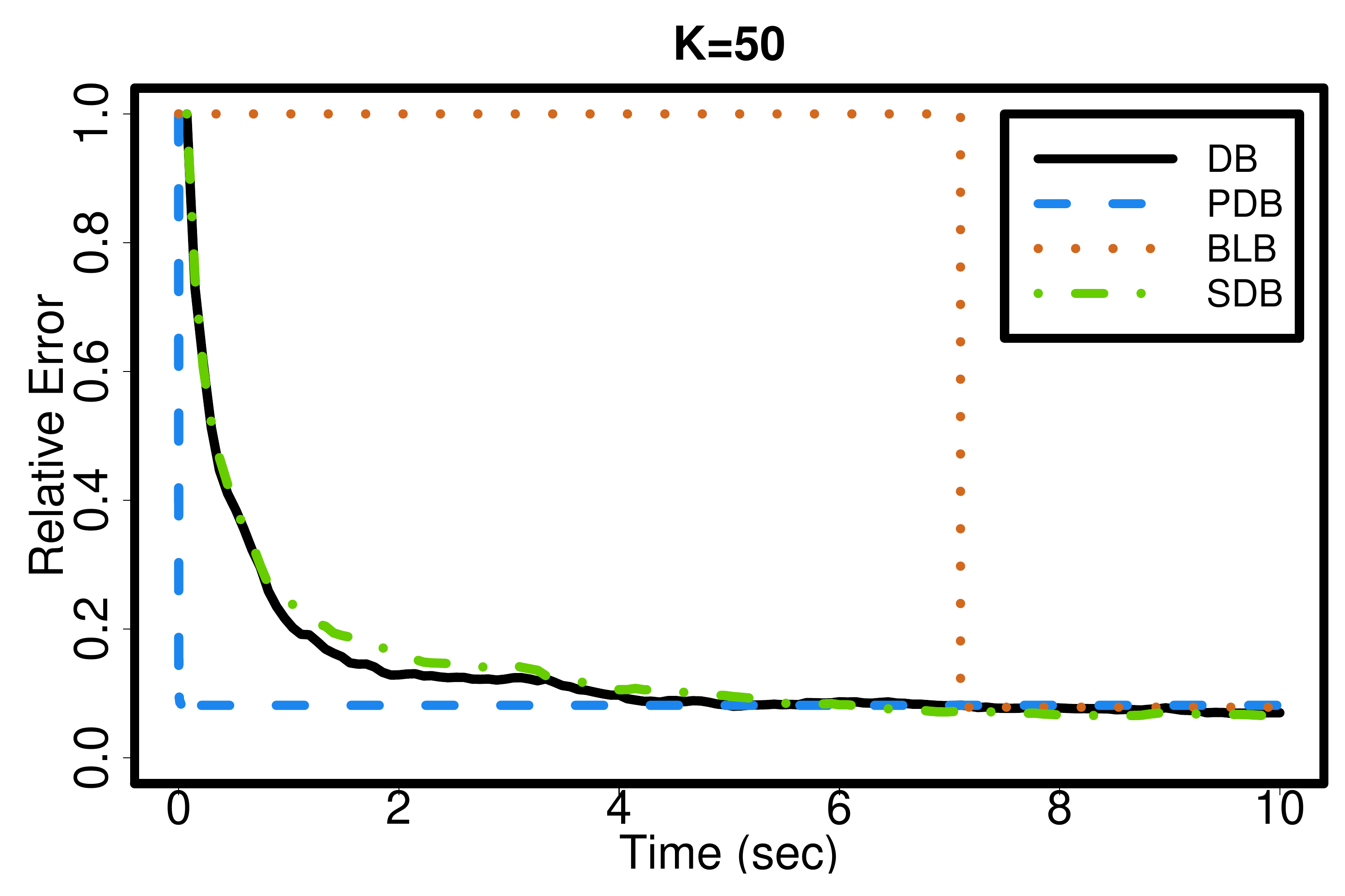}
		\label{fig:pois_K_50}}
	\quad
	\subfigure{%
		\includegraphics[width=0.45\linewidth]{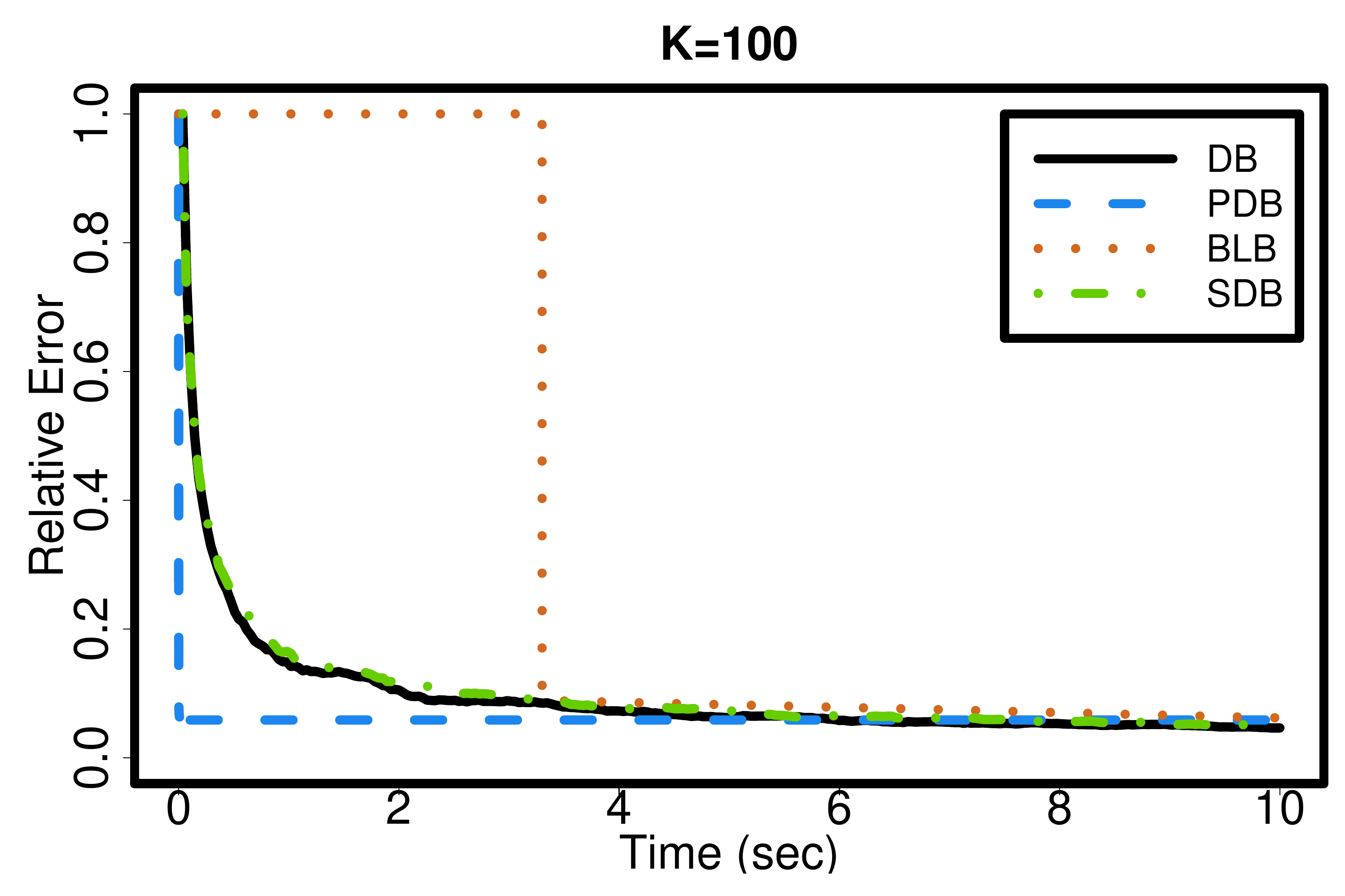}
		\label{fig:pois_K_100}}
	\quad
	\subfigure{%
		\includegraphics[width=0.45\linewidth]{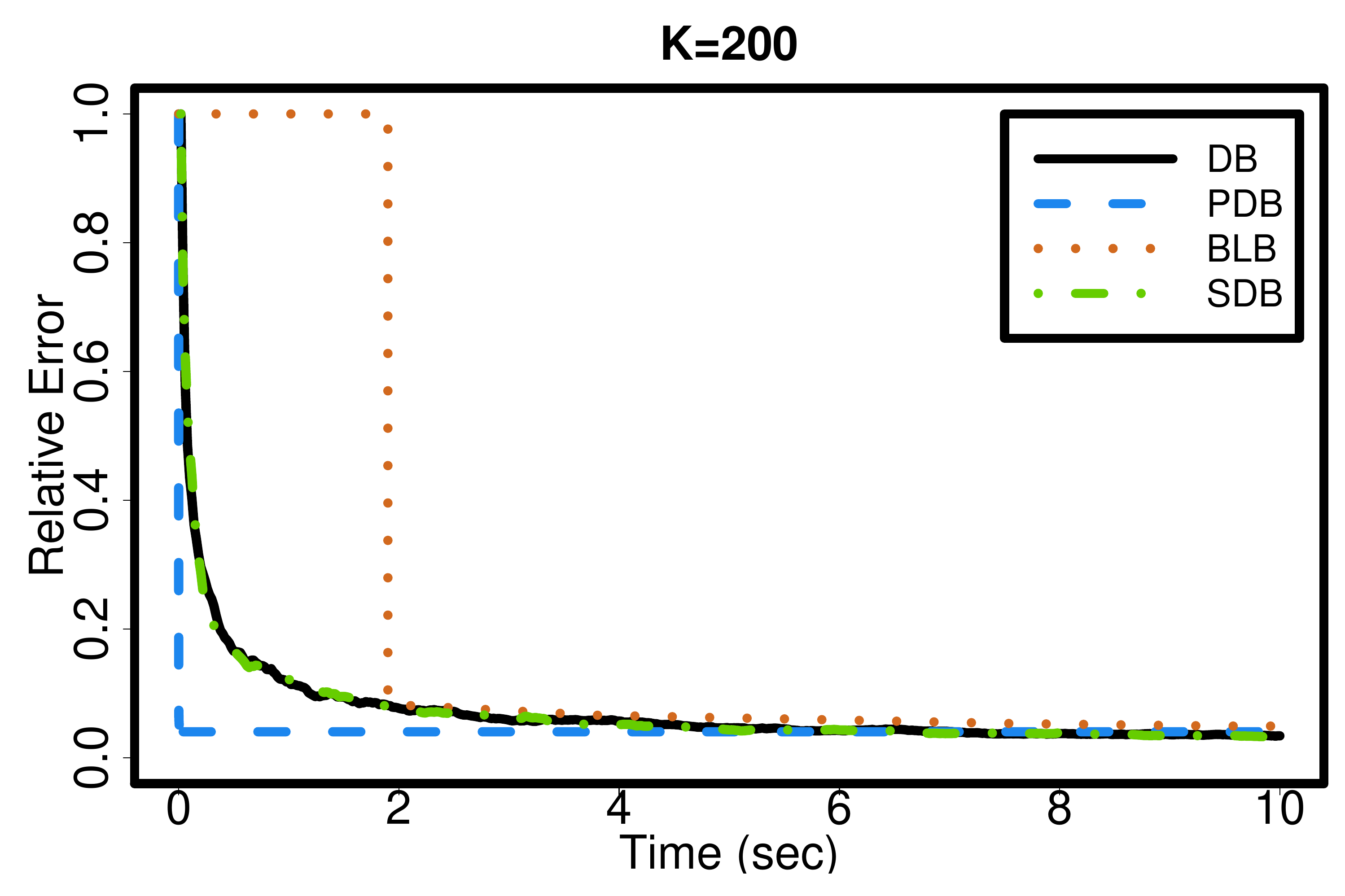}
		\label{fig:pois_K_200}}
	\quad
	\subfigure{%
		\includegraphics[width=0.45\linewidth]{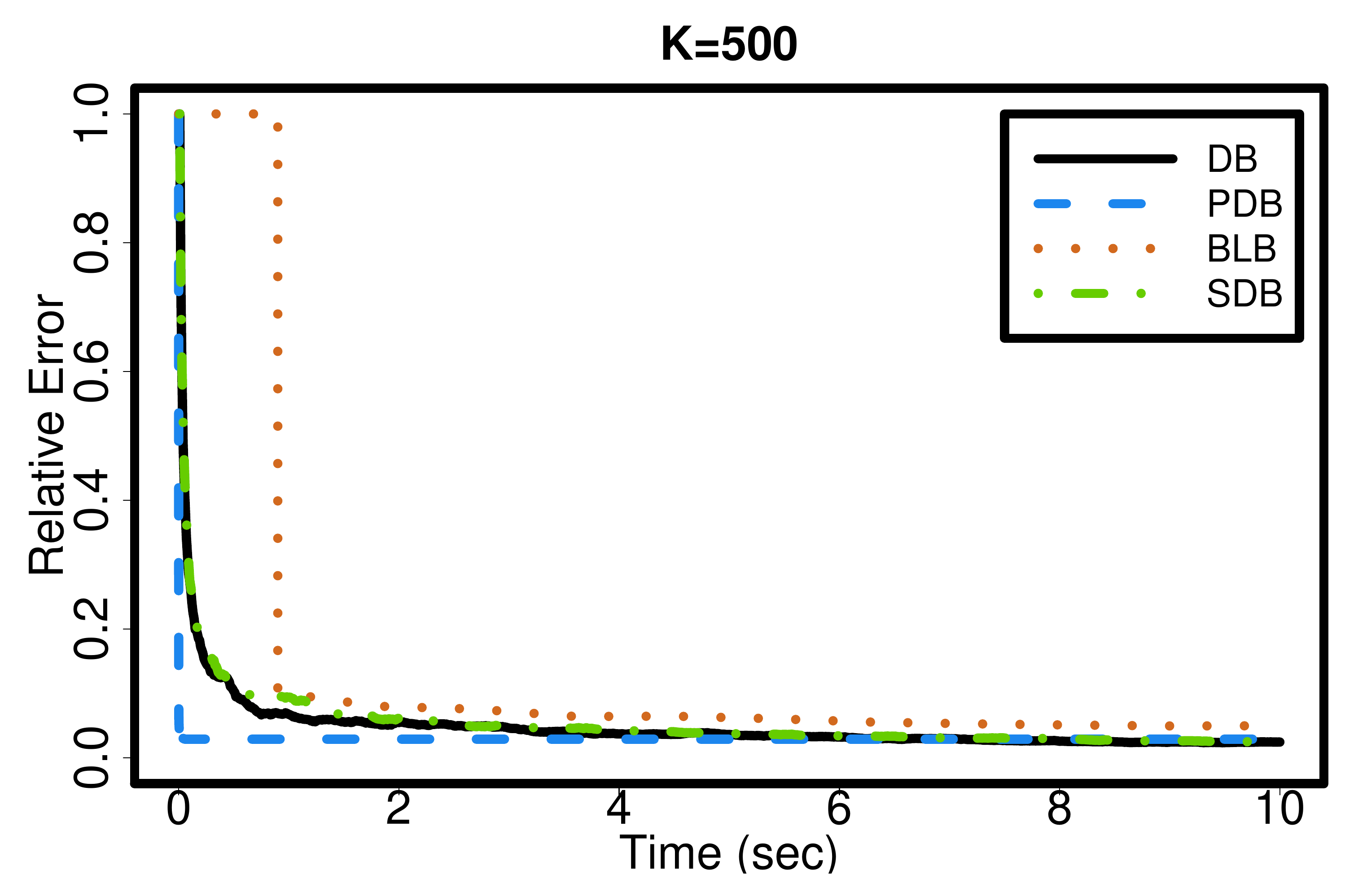}
		\label{fig:pois_K_500}}
	\quad
	\subfigure{%
		\includegraphics[width=0.45\linewidth]{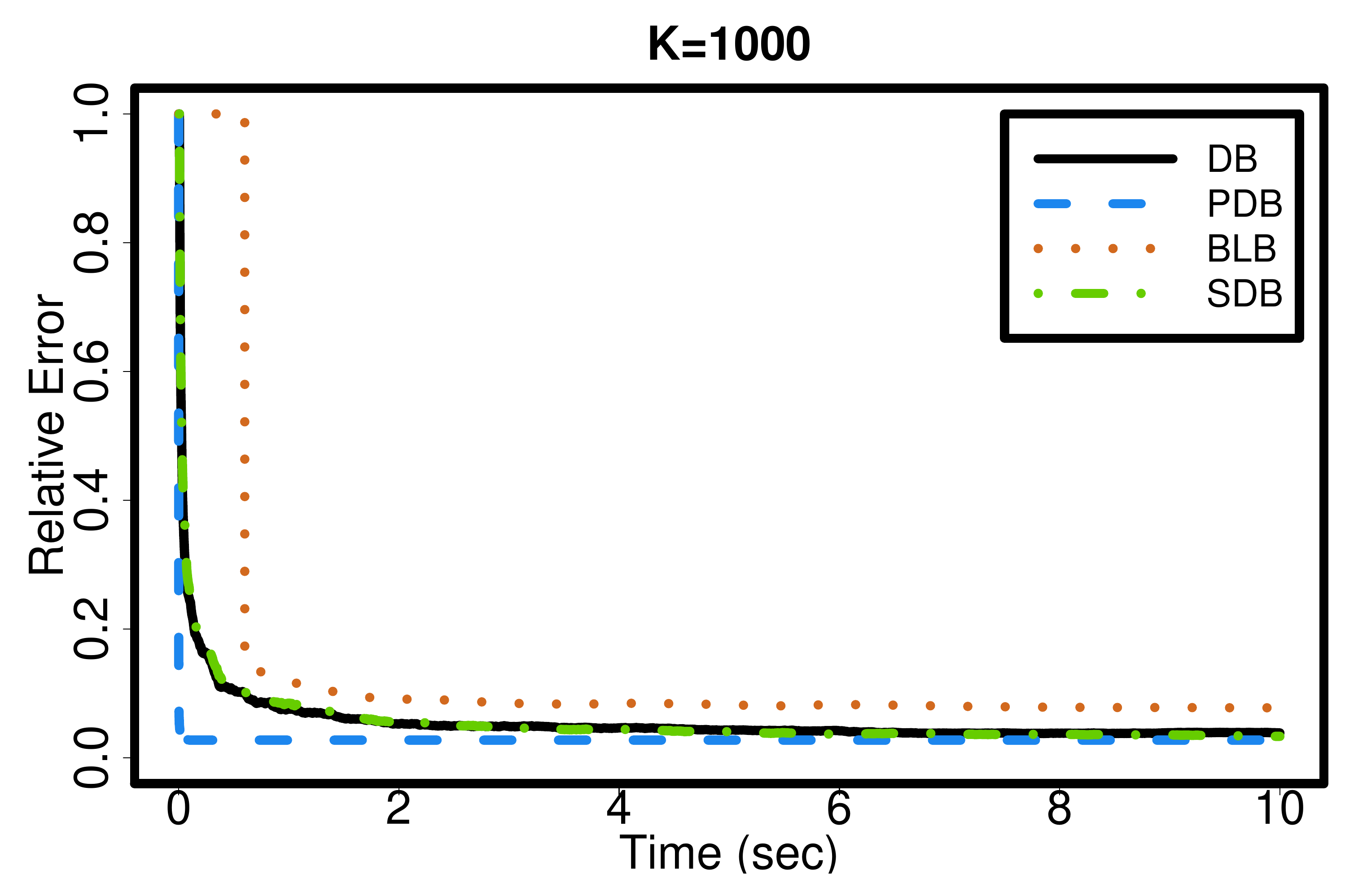}
		\label{fig:pois_K_1000}}
	\quad	
	\caption{Time evolution of the relative errors $|\hat{d}-d|/d$ in the width of the confidence interval under the Poisson scenario with respect to different black size. DB: the distributed bootstrap; PDB: the pseudo-distributed bootstrap (dashed lines); BLB: the bag of little bootstrap (dotted lines); SDB: the subsampled double bootstrap (dot-dashed lines).}
	\label{fig:pois_ciw}
\end{figure}

\renewcommand{\arraystretch}{1.25}
\begin{table}[h]
	\centering
	\caption{Variables considered in the airline on-time performance data.}
	\begin{tabular}{|c|l|}
		\hline
		Feature Variable & Description \\\hline
		\multirow{2}{*}{DIF\_ELAPS (DIF\_ELP)} & Difference in minutes between actual and scheduled elapsed \\
		&  time.  ($\text{ACT\_ELAPS} - \text{CRS\_ELAPS}$) \\\hline
		\multirow{2}{*}{DEP\_DELAY (DEP\_DLY)} & Difference in minutes between scheduled and actual \\
		& departure time. \\\hline
		\multirow{2}{*}{TAXI\_OUT (TX\_OUT)} & Difference in minutes between departure from the origin \\
		& airport gate and wheels off. \\\hline
		\multirow{2}{*}{TAXI\_IN (TX\_IN)} & Differences in minutes between wheels down and arrival at \\
		& the destination airport gate. \\\hline
		\multirow{2}{*}{ACT\_ELAPS (ACT\_ELP)} & Actual elapsed time between departure from the origin airport \\
		& gate and arrival at the destination airport gate in minutes. \\\hline
		\multirow{2}{*}{CRS\_ELAPS (CRS\_ELP)} & Scheduled elapsed time between departure and arrival in \\
		& minutes, CRS is short for Computer Reservation System. \\\hline
		DISTANCE (DIST) & Distance between departure and arrival airports in miles. \\\hline
		MONTH (MON) & Month of flight. \\\hline
		DAY\_OF\_MONTH (DOM) & Day of month. \\\hline
		DAY\_OF\_WEEK (DOW) & Day of week. \\\hline
	\end{tabular}
	\label{tab:features}
\end{table}


\begin{thebibliography}{99}
	
	
	\bibitem[Arcones and Gin\'{e}(1992)]{Arcones1992}
	{Arcones, M. A. and Gin\'{e}, E.} (1992).  On the bootstrap of U and V statistics. \textit{Ann. Statist.} 20, 655-674.
	
	\bibitem[Battey et al.(2015)]{Battey2015}
	{Battey, H., Fan, J., Liu, H., Lu, J. and Zhu, Z.} (2015).  Distributed estimation and inference with
	statistical guarantees. \textit{arXiv preprint arXiv:1509.05457}.
	
	
		
	
	\bibitem[Chang and Hall(2015)]{Chang2015}
	{Chang, J. and Hall, P.} (2015). Double-bootstrap methods that use a single double-bootstrap simulation. \textit{Boimetrika} 102(1), 203-214.
	
	
	\bibitem[Chen and Xie(2014)]{ChenXie2014}
	{Chen, X. and Xie, M.} (2014).  A split-and-conquer approach for analysis of extraordinarily large data. \textit{Statistica Sinica}, 24, 1655-1684.
	
	\bibitem[Davidson and MacKinnon(2002)]{Davidson2002}
	{Davidson, R. and MacKinnon, J. G.} (2002).  Fast double bootstrap tests of nonnested linear regression models. \textit{Econometric Reviews}, 21(4), 419-429.
	
	\bibitem[Dunford and Schwartz(1963)]{Dunford1963}
	{Dunford, N. and Schwartz, J. T.} (1963). \textit{Linear Operators.} Wiley-Interscience, New York.	

	\bibitem[Efron(1979)]{Efron1979}
	{Efron, B.} (1979).  Bootstrap methods: another look at the jackknife. \textit{Ann. Statist.} 7, 1-16.
	
	\bibitem[Efron and Stein(1981)]{Efron1981}
	{Efron, B. and Stein, C.} (1981).  The jackknife estimate of variance. \textit{Ann. Statist.} 9, 586-596.
	

	\bibitem[Hall(1986)]{Hall1986}
	{Hall, P.} (1986). On the bootstrap and confidence intervals. \textit{Ann. Statist.} 14, 1431-1452.	

	\bibitem[Hall(1988)]{Hall1988}
	{Hall, P.} (1988). Theoretical comparison of bootstrap confidence intervals. \textit{Ann. Statist.} 16, 927-953.
	
	\bibitem[Hall(1992)]{Hall1992}
	{Hall, P.} (1992). \textit{The Bootstrap and Edgeworth Expansion.} Springer-Verlag, New York.
	
	\bibitem[He and Shao(1996)]{HeAndShao1996}
	{He, X. and Shao, Q.-M.} (1996). A general Bahadur representation of M-estimators and its application to linear regression with nonstochastic designs. \textit{Ann. Statist.} 24, 2608-2630.	
	
	
	\bibitem[Hoeffding(1948)]{Hoeffding1948}
	{Hoeffding, W.} (1948). A class of statistics with asymptotically normal distribution. \textit{Ann. Math. Statist.} 19, 293-325.
	
	\bibitem[Hoeffding(1961)]{Hoeffding1961}
	{Hoeffding, W.} (1961). The strong law of large numbers for U-statistics. \textit{University of North Carolina Institute of Statistics Mimeo Series, No. 203}.
	
	\bibitem[Huang and Huo(2015)]{ChenHuo2015}
	{Huang, C. and Huo, X.} (2015). A Distributed One-Step Estimator. \textit{arXiv preprint arXiv:1511.01443}.
	
	
	\bibitem[Jing and Wang(2010)]{Jing2010}
	{Jing, B.-Y. and Wang, Q.} (2010). A unified approach to Edgeworth expansions for a general class of statistics. \textit{Statistica Sinica}, 20, 613-636.
	
	\bibitem[Kleiner et al.(2014)]{Kleiner2014}
	{Kleiner, A., Talwalkar, A., Sarkar, P. and Jordan, M. I.} (2014). A scalable bootstrap for massive
	data. \textit{Journal of the Royal Statistical Society: Series B (Statistical Methodology)}, 76, 795-816.
	
	
	
	\bibitem[Lai and Wang(1993)]{Lai1993}
	{Lai, T. L. and Wang, J. Q.} (1993). Edgeworth expansions for symmetric statistics with applications to bootstrap methods. \textit{Statistica Sinica}, 3, 517-542.
	
	\bibitem[Lee et al.(2015)]{Lee2015}
	{Lee, J. D., Liu, Q., Sun, Y. and Taylor, J. E.} (2015). Communication-efficient sparse regression: a one-shot approach. \textit{arXiv preprint arXiv:1503.04337}.
	
	
	\bibitem[Lin and Xi(2010)]{LinandXi2010}
	{Lin, N. and Xi, R.} (2010). Fast surrogates of U-statistics. \textit{Computational Statistics and Data Analysis}, 54, 16-24.
	
	\bibitem[Liu(1988)]{Liu1988}
	{Liu, R. Y.} (1988). Bootstrap procedures under some non-i.i.d. models. \textit{Ann. Stat.} 4, 1696-1708.
	
	\bibitem[Lomnicki(1952)]{Lomnicki1952}
	{Lomnicki, Z. A.} (1952). The standard error of Gini’s mean difference. \textit{The Annals of Mathematical Statistics}, 23, 635-637.
	
	\bibitem[Petrov(1998)]{Petrov1998}
	{Petrov, V. V.} (1998). \textit{Limit Theorems of Probability Theory: Sequences of Independent Random Variables}. Oxford University Press Inc., New York.
	
	
	\bibitem[Sengupta, Volgushev and Shao(2016)]{Sengupta2015}
	{Sengupta, S., Volgushev, S. and Shao, X.} (2016). A subsampled double bootstrap for massive data. \textit{Journal of the American
		Statistical Association}, \textbf{111} 1222-1232.
	
	\bibitem[Serfling(1980)]{Serfling1980}
	{Serfling R. J.} (1980). \textit{Approximation Theorems of Mathematical Statistics.} John Wiley $\&$ Sons, Inc.	
	
	\bibitem[Shao and Tu(1995)]{Shao1995}
	{Shao, J. and Tu, D.} (1995). \textit{The Jackknife and Bootstrap.} Springer-Verlag New York, Inc.
	
	\bibitem[Sz\'{e}kely and Rizzo(2014)]{SR2014}
	{Sz$\acute{e}$kely and G. J., Rizzo, M. L.} (2014). Partial distance correlation with methods for dissimilarities. \textit{Ann. Stat.}, \textbf{42} 2382-2412.
	
	\bibitem[Sz\'{e}kely, Rizzo and Bakirov(2007)]{SRB2007}
	{Sz$\acute{e}$kely, G. J., Rizzo, M. L. and Bakirov N. K.} (2007). Measuring and testing dependence by correlation of distance. \textit{Ann. Stat.}, \textbf{35} 2769-2794.
	
	
	\bibitem[Zhang, Duchi and Wainwright(2013)]{Zhang2013}
	{Zhang, Y., Duchi, J. C. and Wainwright, M. J.} (2013). Communication-efficient algorithms for statistical optimization. \textit{Journal of Machine Learning Research}, 14, 3321-3363.
	
\end{thebibliography}

\begin{thebibliography}{99}
	
	\bibitem[Arcones and Gin\'{e}(1992)]{Arcones1992}
	{Arcones, M. A. and Gin\'{e}, E.} (1992).  On the bootstrap of U and V statistics. \textit{Ann. Statist.} 20, 655-674.

	\bibitem[Callaert and Veraverbeke(1981)]{Callaert1981}
	{Callaert, H. and Veraverbeke, N.} (1981). The order of the normal approximation for a studentized U-statistic. \textit{Ann. Statist.} 1, 194-200.
	
	\bibitem[Efron and Stein(1981)]{Efron1981}
	{Efron, B. and Stein, C.} (1981).  The jackknife estimate of variance. \textit{Ann. Statist.} 9, 586-596.
	
	\bibitem[Kleiner et al.(2014)]{Kleiner2014}
	{Kleiner, A., Talwalkar, A., Sarkar, P. and Jordan, M. I.} (2014). A scalable bootstrap for massive
	data. \textit{Journal of the Royal Statistical Society: Series B (Statistical Methodology)} 76, 795-816.
	
	\bibitem[Koroljuk and Borovskich(1994)]{Koroljuk1994}
	{Koroljuk, V. S. and Borovskich, Yu. V.} (1994). \textit{Theory of U-statistics}. Kluwer Academic Publishers.
	
	\bibitem[Liu(1988)]{Liu1988}
	{Liu, R. Y.} (1988). Bootstrap procedures under some non-i.i.d. models. \textit{Ann. Stat.} 4, 1696-1708.
	
	\bibitem[Petrov(1998)]{Petrov1998}
	{Petrov, V. V.} (1998). \textit{Limit Theorems of Probability Theory: Sequences of Independent Random Variables}. Oxford University Press Inc., New York.
	
	\bibitem[Sengupta, Volgushev and Shao(2016)]{Sengupta2015}
	{Sengupta, S., Volgushev, S. and Shao, X.} (2016). A subsampled double bootstrap for massive data. \textit{Journal of the American
		Statistical Association}, \textbf{111} 1222-1232.
	
	\bibitem[Serfling(1980)]{Serfling1980}
	{Serfling R. J.} (1980). \textit{Approximation Theorems of Mathematical Statistics.} John Wiley $\&$ Sons, Inc.	
	
	\bibitem[Sz\'{e}kely and Rizzo(2014)]{SR2014}
	{Sz$\acute{e}$kely and G. J., Rizzo, M. L.} (2014). Partial distance correlation with methods for dissimilarities. \textit{Ann. Stat.}, \textbf{42} 2382-2412.
	
	\bibitem[Sz\'{e}kely, Rizzo and Bakirov(2007)]{SRB2007}
	{Sz$\acute{e}$kely, G. J., Rizzo, M. L. and Bakirov N. K.} (2007). Measuring and testing dependence by correlation of distance. \textit{Ann. Stat.}, \textbf{35} 2769-2794.
	
	\bibitem[Yao, Zhang and Shao(2017)]{YZS2016}
	{Yao, S., Zhang, X. and Shao, X.} (2017). Testing mutual independence in high dimension via distance covariance. \textit{arXiv preprint arXiv:1609.09380}.
	
\end{thebibliography}
\end{document}